\documentclass[12pt,a4paper]{amsart}

\usepackage{pgf,tikz,pgfplots}%these four lines are for geogebra
\pgfplotsset{compat=1.14}
\usepackage{mathrsfs}
\usetikzlibrary{arrows}

\makeatletter
\newcommand{\labeltext}[2]{%
	\@bsphack
	\def\@currentlabel{#1}{\label{#2}}%
	\@esphack
}
\makeatother

\def\step#1#2#3{\par \noindent{{\vskip 5pt \bf Step~\labeltext{#1}{#3}#1. }{\bf #2. }}}

\makeatletter
\long\def\unmarkedfootnote#1{{\long\def\@makefntext##1{##1}\footnotetext{#1}}}
\makeatother

\usepackage{amsmath,amssymb}%,a4}
\usepackage{enumerate,amssymb}
\usepackage{color}
\theoremstyle{plain}
\newtheorem{thm}{Theorem}[section]

\newtheorem{lemma}[thm]{Lemma}

\newtoks\prt
\newtheorem{proclaim}[thm]{\the\prt}
\theoremstyle{definition}

\def\eqn#1$$#2$${\begin{equation}\label#1#2\end{equation}}

\numberwithin{equation}{section}

\newcommand{\ignore}[1]{}

\def\dist{\operatorname{dist}}

\def\epsilon{\varepsilon}
\def\en{\mathbb N}
\def\er{\mathbb R}

\def\mir1{\mathcal L_1}
\def\mir2{\mathcal L_{2}}
\def\oint{-\hskip -13pt \int}

\def\phi{\varphi}

\def\R{\widetilde{R}}

\def\rn{\mathbb R^n}

\def\x{\widetilde{x}}

\def\zet{\mathbb Z}

\def\R{\mathcal{R}}
\def\S{\mathbb{S}^1}
\def\r{\mathbf{r}}

\newtoks\by
\newtoks\paper
\newtoks\book
\newtoks\jour
\newtoks\yr
\newtoks\pages
\newtoks\vol
\newtoks\publ
\def\ota{{\hbox\vol{???}}}
\def\cLear{\by=\ota\paper=\ota\book=\ota\jour=\ota\yr=\ota
\pages=\ota\vol=\ota\publ=\ota}
\def\endpaper{\the\by, {\the\paper},
\textit{\the\jour} \textbf{\the\vol} (\the\yr), \the\pages.\cLear}
\def\endbook{\the\by, \textit{\the\book}, \the\publ.\cLear}
\def\endprep{\the\by, \textit{\the\paper}, \the\jour.\cLear}
\def\endyearprep{\the\by, \textit{\the\paper}, \the\jour, (\the\yr).\cLear}
\def\name#1#2{#2 #1}
\def\nom{ \rm no. }

\title[Approximation of planar Sobolev $W^{2,1}$ homeomorphisms]{Approximation of planar Sobolev $W^{2,1}$ homeomorphisms by Piecewise Quadratic Homeomorphisms and Diffeomorphisms}

\author{Daniel Campbell and Stanislav Hencl }
\thanks{The first author was supported by the grant GA\v{C}R 20-19018Y. The second author was supported by the grant GA\v{C}R P201/18-07996S}
\address{Department of Mathematics, University of Hradec Kr\'alov\'e, Rokitansk\'eho 62, 500 03 Hradec Kr\'alov\'e, Czech Republic}
\address{Faculty of Economics, University of South Bohemia, Studentsk\' a 13, Cesk\' e Budejovice, Czech Republic}
\email{campbell@uhk.cz}
\address{Department of Mathematical Analysis, Charles University, So\-ko\-lovsk\'a 83, 186~00 Prague 8, Czech Republic}
\email{\tt hencl@karlin.mff.cuni.cz}

\date{\today}

\begin{document}

\begin{abstract}
Given a Sobolev homeomorphism $f\in W^{2,1}$ in the plane we find a piecewise quadratic homeomorphism that approximates it up to a set of $\epsilon$ measure. We show that this piecewise quadratic map can be approximated by diffeomorphisms in the $W^{2,1}$ norm on this set.  
\end{abstract}

\maketitle

\section{Introduction}

In this paper we address the issue approximation of Sobolev homeomorphisms with diffeomorphisms. Let us briefly explain the motivation for this problem that comes from Nonlinear Elasticity. %In the setting of nonlinear elasticity   
	%Ball~\cite{B2} and \cite{B4}. 
	Let $\Omega \subset \rn$ be a domain which models a body made out of homogeneous elastic material, and let $f : \Omega \to \rn$ be a mapping modeling the deformation of this body with prescribed boundary values. In the theory of nonlinear elasticity pioneered by Ball and Ciarlet, see e.g. \cite{B, BB, Ci}, we study the existence and regularity properties of minimizers of the energy functionals 
	\begin{align*}
	I(f) = \int_{\Omega} W(Df) \; dx,
	\end{align*}  
	where $W : \er^{n \times n} \to \er$ is the so-called stored-energy functional, and $Df$ is the differential matrix of the mapping $f$. The physically relevant assumptions on the model include:
	\begin{itemize}
		%\item[(W1)] $f$ is a homeomorphism, which corresponds to the non-impenetrability of the material.
		\item[(W1)] $W(A) \to +\infty$ as $\det A \to 0$, i.e. mapping does not compress too much,
		\item[(W2)] $W(A) = +\infty$ if $\det A \le 0$, which guarantees that the orientation is preserved.
	\end{itemize} 
	In particular, any admissible deformation $f$ satisfies 
	\begin{align*}
	J_{f}(x) := \det Df(x) > 0 \quad \text{for a.e. $x \in \Omega$.}
	\end{align*} 
	With the help of some growth assumptions on $W$ we can prove that a mapping with finite energy is continuous and one-to-one, which corresponds to the non-impenetrability of the matter. Hence it is natural to study Sobolev homeomorphisms with $J_f>0$ a.e. that minimize the energy.
	
	As pointed out by Ball in~\cite{B2, Ball2} (who ascribes the question to Evans~\cite{E}), an important issue toward understanding the regularity of the minimizers in this setting would be to show the existence of minimizing sequences given by piecewise affine homeomorphisms or by diffeomorphisms. This question is called the Ball-Evans approximation problem and asks as a first step to approximate any homeomorphism $u\in W^{1,p}(\Omega;\er^n)$, $p\in[1,+\infty)$ by piecewise affine homeomorphisms or by diffeomorphisms in $W^{1,p}$ norm.
	The motivation is that regularity is typically proven by testing the weak equation or the variation formulation by the solution itself; but without some a priori regularity of the solution, the integrals are not finite. Thus we need to test the equation 
	with a smooth test mapping of finite energy which is close to the given homeomorphism instead. %Here we see the necessity for the approximations to be homeomorphisms whose image is the same as that of the approximated map, otherwise this sequence would have nothing in common with our original problem. 
	Besides Nonlinear Elasticity, an approximation result of homeomorphisms with diffeomorphisms would be a very useful tool as it allows a number of proofs to be significantly simplified. 
	Let us note that finding diffeomorphisms near a given homeomorphism is not an easy task, as the usual approximation techniques like mollification or Lipschitz extension using the maximal operator destroy, in general, injectivity.

	Let us describe the known results about the Ball-Evans approximation problem.  
	The problems of approximation by diffeomorphisms or piecewise affine planar homeomorphisms are in fact equivalent by the result of Mora-Corral and Pratelli \cite{MP} (see also \cite{IO2}).
	The first positive results on approximation of planar homeomorphisms smooth outside a point are by Mora-Corral~\cite{M}.
	% and by Bellido and Mora-Corral~\cite{BMC} on approximation in H\" older continuous maps.  
	The celebrated breakthrough result in the area which stimulated much interest in the subject was given by 
	Iwaniec, Kovalev and Onninen in~\cite{IKO}, where they found diffeomorphic approximations to any homeomorphism $f\in W^{1,p}(\Omega,\er^2)$, for any $1<p<\infty$ in the $W^{1,p}$ norm. The remaining missing case $p=1$ in the plane has been solved by Hencl and Pratelli in~\cite{HP} by a different method. This method was extended to cover other function spaces like Orlicz-Sobolev spaces (see Campbell \cite{Ca}), BV space (see Pratelli, Radici \cite{PR}) or $WX$ for nice rearrangement invariant Banach function space $X$ 
	(see Campbell, Greco, Schiattarella, Soudsk\'y \cite{CS}). It is possible to approximate also $f^{-1}$ in the Sobolev norm for $p=1$ 
	(see Pratelli \cite{BiP}) or for $1\leq p<\infty$ under the additional assumption that the mapping is bi-Lipschitz (see Daneri and Pratelli in \cite{DP}). Moreover, it is possible to characterize all strong limits of Sobolev diffeomorphisms (not only homeomorphisms) as shown by Iwaniec and Onninen \cite{IO} for $p\geq 2$ and by De Philippis and Pratelli \cite{DPP} for $1\leq p<2$. The higher dimensional case $n\geq 3$ is widely open and essentially nothing is known for $n=3$. However, for $n\geq 4$ and $1\leq p<[\frac{n}{2}]$ there exists a Sobolev homeomorphism in $W^{1,p}$ which cannot be approximated by diffeomorphisms (see Hencl and Vejnar \cite{HV}, Campbell, Hencl and Tengvall \cite{CHT}, Campbell, D'Onofrio and Hencl \cite{CDH}).
	
	Our aim is to find the corresponding planar result for models with second gradient, i.e. we would like to approximate $W^{2,q}$ homeomorphisms by diffeomorphisms. Models with the second gradient 
	\eqn{functional2}
	$$
	E(f)=\int_{\Omega}\bigl(W(Df(x))+\delta_0|D^2f(x)|^q \bigr)\; dx,
	$$
	where $q\in[1,\infty)$ and $\delta_0>0$, 
	were introduced by Toupin \cite{T}, \cite{T2}
	and later considered by many other authors, see e.g. Ball, Curie, Olver \cite{BCO}, Ball, Mora-Corral \cite{BMC}, M\"uller \cite[Section 6]{Mbook}, Ciarlet \cite[page 93]{Ci} and references given there. 
	The contribution of the higher gradient is usually connected with interfacial energies and is used to model various phenomena like elastoplasticity or damage.
	If $q$ is much bigger than the dimension $2$, then under some additional assumptions we can actually conclude that $J_f\geq \sigma>0$ and we can approximate (see \cite{HKr}). The more physically relevant assumptions are $q=1$ or $q=2$. In that case 
	the usual convolution approximation is not useful for approximation as it in general destroys injectivity in places where the Jacobian is close to zero.
	
	In this paper we start to study the case $q=1$. We cannot use the approach of \cite{IKO} as there is no analogy of the key extension procedure, i.e. of Rado-Choquet-Knesser theorem. Indeed, even in the one-dimensional case the minimizers of the $W^{2,q}((0,1))$ energy do not need to be injective once we prescribe the boundary data $f(0), f(1)>f(0), f'(0)>0$ and $f'(1)>0$ with derivatives much bigger than $f(1)-f(0)$. 
	Instead we use some ideas of \cite{HP}, we cover $\Omega$ by triangles and we divide triangles into good and bad according to the behavior of $f$ like differentiability on them. The total measure of bad triangles is small and we approximate $f$ on good triangles by quadratic polynomials. Then we smoothen this piecewise quadratic mapping along the edges of triangles and we obtain the desired diffeomorphism. 
	Note that piecewise linear approximation on triangles as in \cite{HP} is not good as the second derivative on linear pieces is zero and we would not be able to approximate strongly the second derivative.

	We call $\Omega\subset\er^2$ a polygonal domain, if we can find triangles $\{T_i\}_{i=1}^k$ with pairwise disjoint interiors so that $\overline{\Omega}=\bigcup_{i=1}^k T_i$. We say that $f:\overline{\Omega}\to\er$ is piecewise quadratic, if it is continuous and there is a triangulation $\overline{\Omega}=\bigcup_{i=1}^k T_i$ such that $f|_{T_i}$ is a quadratic function in each coordinate, i.e. 
	$$
	f(x,y)=[a_1+a_2x+a_3y+a_4x^2+a_5xy+a_6y^2,b_1+b_2x+b_3y+b_4x^2+b_5xy+b_6y^2]\text{ on }T_i. 
	$$ 
	Note that with our quadratic approximation we cannot achieve the continuity of the derivative in the direction perpendicular to sides of the triangles, but we can achieve that the jumps of the derivative there are small. Thus our piecewise quadratic mapping does not belong to $W^{2,1}$ but it belongs to $WBV$, i.e. its derivative is a $BV$ mapping. By $D^2_s f$ we denote the singular part of the second derivative which is supported in $\bigcup_{i=1}^k\partial T_i$ and corresponds to jump of derivatives between touching triangles. 
	
	Our first result is an analogy of \cite{MP} for second derivatives, but with no control of derivative of the inverse. It states that given a nice piecewise quadratic approximation (with small jumps of derivatives, see \eqref{smalljump}) we can find a diffeomorphic approximation. 
	
	\prt{Theorem}
	\begin{proclaim}\label{Dan}
		Let $\Omega\subset\er^2$ be a polygonal domain. Let $\delta,d>0$ and assume that $f:\overline{\Omega}\to\er^2$ is a piecewise quadratic homeomorphism so that 
		\eqn{smalljump}
		$$
		\int_{\Omega}|D^2_s f|<\delta\text{ and }J_f>d\text{ a.e. in }\Omega.
		$$
		Then for every $\epsilon>0$ we can find a $C^{\infty}$ diffeomorphism $g:\overline{\Omega}\to\er^2$ such that 
		$$
		\|f-g\|_{WBV(\Omega,\er^2)}<\epsilon+C\delta\text{ and } \|f-g\|_{L^{\infty}(\Omega,\er^2)}<\epsilon. 
		$$
	\end{proclaim}
	
	In our second result we apply the previous result to show a diffeomorphic approximation of $W^{2,1}$ homeomorphism up to a set of small measure. 
	This part is more difficult than the corresponding result in \cite{HP} as we have to deal also with second derivatives and with piecewise quadratic approximation.

\prt{Theorem}
\begin{proclaim}\label{Standa}
Let $\Omega\subset\er^2$ be a domain of finite measure. Let $f\in W^{2,1}(\Omega,\er^2)$ be a homeomorphism such that $J_f>0$ a.e.
Then for every $\nu>0$ we can find squares $\{Q_i\}_{i=1}^{\infty}$ which are locally finite (i.e. each compact set $K\subset\Omega$ intersects only finitely many of them) with
$$
\mir2\Bigl(\bigcup_{i=1}^{\infty}Q_i\Bigr)<\nu
$$
and we can find $C^{\infty}$ diffeomorphism $g:\Omega\setminus\bigcup_{i=1}^{\infty}Q_i \to\er^2$ such that 
$$
\|f-g\|_{W^{2,1}(\Omega\setminus\bigcup_{i=1}^{\infty}Q_i,\er^2)}<\nu. 
$$
\end{proclaim}

The natural plan for our future research is to obtain some analogy of the key extension result \cite[Theorem 2.1]{HP} which will lead to the full approximation result of $W^{2,1}$ homeomorphisms, i.e. we would be able to deal also with a set of small measure $\bigcup Q_i$. Moreover, we could try to obtain an analogy of \cite[Theorem 3.1]{HP}, which would even remove the assumption $J_f>0$ a.e., even-though it is quite natural in models of Nonlinear Elasticity.   

\section{Preliminaries}

%Second derivatives in polar coordinates - EXPLAIN IN DETAIL 
By $[x,y]$ we denote the point in $\er^2$ with coordinates $x$ and $y$. The scalar product of $u,v\in\er^2$ is denoted by 
$\langle u,v\rangle$. By $B(c,r)$ we denote the ball centered at $c\in\er^2$ with radius $r>0$ and $Q(c,r)$ denotes the corresponding square. 

Let $u,v\in \er^2$ be nonzero. Then we have the following elementary estimate
\eqn{elementary}
$$
\Bigl|\frac{u}{|u|}-\frac{v}{|v|}\Bigr|=
\Bigl|\frac{u}{|u|}-\frac{v}{|u|}+v\Bigl(\frac{|v|-|u|}{|u|\ |v|}\Bigr)\Bigr|
\leq 2\frac{|u-v|}{|u|}. 
$$
 
	We introduce a smooth function which grows from 0 to 1 on $[0,1]$ and plays an important role in our construction. 
	\prt{Notation}
	\begin{proclaim}\label{Notation}
		Let $\eta:\er \to \er$ be a fixed smooth function with $\eta(x) = 0$, $x\leq 0$ and $x= 1$, $x\geq 1$, $\eta$ increasing on $[0,1]$, $0\leq \eta'\leq2$ and $|\eta''|\leq4$.
	\end{proclaim}

For $f:\er^2\to\er^2$ we use the notation for first derivatives $D_x f=\frac{\partial f}{\partial x}$, $D_y=\frac{\partial f}{\partial x}$ and similarly for second derivatives $D_{xx} f=D_x (D_xf)$, $D_{yy}f=D_y(D_yf)$ and $D_{xy}=D_x(D_yf)$. Similarly for any vector $u\in\er^2$ we denote by $D_u f$ the derivative of $f$ in the $u$ direction. 

It is well-known that for $C^1$ mapping the classical and distributional derivatives agree. Hence 
	for any domain $G\subset\er^2$, $f\in C^{1}(G,\er^2)$ and $\{u,v\}$ and $\{\vec{u}, \vec{v}\}$ a pair of positively oriented orthonormal bases of $\er^2$ we have 
		\eqn{JackOfAllTrades}
		$$
		J_f(x,y) = \det Df(x,y) =  \langle D_u f(x,y), \vec{u}\rangle\langle D_vf(x,y), \vec{v}\rangle - \langle D_u f(x,y), \vec{v}\rangle\langle D_vf(x,y), \vec{u}\rangle,
		$$
		for almost every $[x,y] \in G$ where $J_f$ is the weak Jacobian of $f$. This is essentially the invariance of the determinant with respect to the choice of a positively oriented orthonormal basis.  

\subsection{Representation of higher order derivatives}\label{higherderivatives} 
Given $f\in C^2(\Omega, \er^2)$ we can view $Df$ as the the mapping from $\er^{2\times 2}\to C(\Omega)$ (i.e. as a matrix)
and we can define the symbol $D^2 f$ as the mapping from $\er^{2\times 2\times 2}\to C(\Omega)$ (i.e. as the operator on $2\times 2$ matrices). We know that $D(fg)=(Df)g+f(Dg)$ as matrices 
and similarly we can symbolically write 
$$
D^2(fg)=D\bigl((Df)g+f(Dg)\bigr)=D^2f g+Df Dg+DfDg+fD^2g
$$
where on the righthand side we see the correct terms of the product. At the end we will just estimate the norm of this by the corresponding product of norms and the exact terms will not be important for us.  

\subsection{ACL condition} Let $\Omega\subset\er^2$ be an open set. 
It is a well-known fact (see e.g. \cite[Section 3.11]{A}) that a mapping
$u\in L^1(\Omega,\er^m)$ is in $W^{1,1}(\Omega,\er^m)$
if and only there is a representative which is an absolutely continuous function 
on almost all lines parallel to coordinate axes and the variation on these lines is integrable. 

Analogously for any given direction 
$v\in\er^2$ we can fix $v^{\bot}\bot v$ and define $h_s(t)=u(sv^{\bot}+tv)$ for the right representative of $u$. 
Then for a.e. $s$ the function $h_s$ is absolutely continuous on $L_s:=\{t:sv^{\bot}+tv\in\Omega\}$ and 
$$
\int_{\infty}^{\infty}\int_{L_s}|h_s'(t)|\; dt\; ds\leq \int_{\Omega} |Du(x)|\; dx.
$$

\subsection{FEM quadratic approximation on triangles}\label{FEMapr}

We need to define a quadratic polynomial $A$ that approximates our mapping $f$ on a triangle $T$. Without loss of generality let $T$ have vertices $v_1=[0,0]$, $v_2=[r,0]$ and $v_3=[0,r]$ for some $r>0$ (other triangles we use are just a translation and rotation). We have a mapping $f:T\to\er^2$ and we want to define mapping
$$
A(x,y)=[a_1+a_2x+a_3y+a_4x^2+a_5xy+a_6y^2,b_1+b_2x+b_3y+b_4x^2+b_5xy+b_6y^2],
$$
where $a_i,b_i\in\er$, $i\in\{1,2,3,4,5,6\}$. We choose these constants so that for $j=1,2,3$
\eqn{FEM}
$$
\begin{aligned}
A(v_j)&=\oint_{B(v_j,\frac{r}{10})}f,\ D_xA(v_1)=\oint_{B(v_1,\frac{r}{10})}D_xf,\ 
-D_yA(v_3)=\oint_{B(v_3,\frac{r}{10})}-D_yf\\
&\text{ and }
(-D_{x}+D_y)A(v_3)=\oint_{B(v_3,\frac{r}{10})}(-D_{x}+D_y)f,\\
\end{aligned}
$$
that is values of $A$ in vertices corresponds to values of $f$ (in averaged sense) and values of derivatives of $A$ in vertices along three sides correspond to derivatives of $f$. 
Note that these $6$ equations in first coordinate (resp. second coordinates) determine the $6$ coefficients $a_i$ (resp. $b_i$) uniquely. 
Moreover, imagine that we have two triangles $T_1$ and $T_2$ with common side and that we define $A_1$ on $T_1$ and $A_2$ on $T_2$ by 
procedure \eqref{FEM} described above. Then $A_1=A_2$ on $\partial T_1\cap \partial T_2$ since it is a quadratic polynomial of one variable on this segment (in each coordinate) and it has the same value at two vertices and the same derivative along the segment in one of the vertices (see Fig. \ref{Fig:Triangulization}).  

%\ignore{
\definecolor{aqaqaq}{rgb}{0.6274509803921569,0.6274509803921569,0.6274509803921569}
	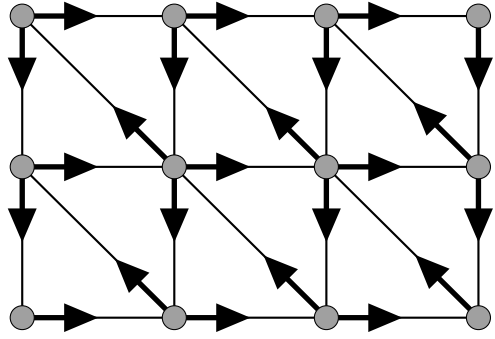
\begin{figure}
		\centering
\begin{tikzpicture}[line cap=round,line join=round,>=triangle 45,x=1.0cm,y=1.0cm]
\clip(-4.47882110339554,-2.718606347553082) rectangle (9.9545376348005625,2.756871309177344);
\draw [line width=0.8pt] (0.,2.)-- (0.,0.);
\draw [line width=0.8pt] (0.,0.)-- (0.,-2.);
\draw [line width=0.8pt] (0.,-2.)-- (2.,-2.);
\draw [line width=0.8pt] (2.,-2.)-- (4.,-2.);
\draw [line width=0.8pt] (4.,-2.)-- (6.,-2.);
\draw [line width=0.8pt] (6.,-2.)-- (6.,0.);
\draw [line width=0.8pt] (6.,0.)-- (6.,2.);
\draw [line width=0.8pt] (6.,2.)-- (4.,2.);
\draw [line width=0.8pt] (4.,2.)-- (2.,2.);
\draw [line width=0.8pt] (2.,2.)-- (0.,2.);
\draw [line width=0.8pt] (0.,0.)-- (2.,0.);
\draw [line width=0.8pt] (2.,0.)-- (4.,0.);
\draw [line width=0.8pt] (4.,0.)-- (6.,0.);
\draw [line width=0.8pt] (4.,-2.)-- (4.,0.);
\draw [line width=0.8pt] (4.,0.)-- (4.,2.);
\draw [line width=0.8pt] (2.,0.)-- (2.,2.);
\draw [line width=0.8pt] (2.,-2.)-- (2.,0.);
\draw [line width=0.8pt] (2.,-2.)-- (0.,0.);
\draw [line width=0.8pt] (2.,0.)-- (0.,2.);
\draw [line width=0.8pt] (4.,0.)-- (2.,2.);
\draw [line width=0.8pt] (6.,0.)-- (4.,2.);
\draw [line width=0.8pt] (4.,-2.)-- (2.,0.);
\draw [line width=0.8pt] (6.,-2.)-- (4.,0.);
\draw [->,line width=2.pt] (0.,-2.) -- (1.,-2.);
\draw [->,line width=2.pt] (2.,-2.) -- (3.,-2.);
\draw [->,line width=2.pt] (4.,-2.) -- (5.,-2.);
\draw [->,line width=2.pt] (4.,0.) -- (5.,0.);
\draw [->,line width=2.pt] (2.,0.) -- (3.,0.);
\draw [->,line width=2.pt] (0.,0.) -- (1.,0.);
\draw [->,line width=2.pt] (0.,2.) -- (1.,2.);
\draw [->,line width=2.pt] (2.,2.) -- (3.,2.);
\draw [->,line width=2.pt] (4.,2.) -- (5.,2.);
\draw [->,line width=2.pt] (0.,2.) -- (0.,1.);
\draw [->,line width=2.pt] (2.,2.) -- (2.,1.);
\draw [->,line width=2.pt] (4.,2.) -- (4.,1.);
\draw [->,line width=2.pt] (6.,2.) -- (6.,1.);
\draw [->,line width=2.pt] (4.,0.) -- (4.,-1.);
\draw [->,line width=2.pt] (6.,0.) -- (6.,-1.);
\draw [->,line width=2.pt] (2.,0.) -- (2.,-1.);
\draw [->,line width=2.pt] (0.,0.) -- (0.,-1.);
\draw [->,line width=2.pt] (2.,-2.) -- (1.220742276646666,-1.220742276646666);
\draw [->,line width=2.pt] (4.,-2.) -- (3.1745784185393138,-1.1745784185393136);
\draw [->,line width=2.pt] (6.,-2.) -- (5.168157733010821,-1.1681577330108213);
\draw [->,line width=2.pt] (6.,0.) -- (5.1348352459604545,0.865164754039546);
\draw [->,line width=2.pt] (4.,0.) -- (3.1968963730993507,0.8031036269006495);
\draw [->,line width=2.pt] (2.,0.) -- (1.1874197895962986,0.8125802104037014);
\begin{scriptsize}
\draw [fill=aqaqaq] (0.,2.) circle (4.5pt);
\draw [fill=aqaqaq] (0.,0.) circle (4.5pt);
\draw [fill=aqaqaq] (0.,-2.) circle (4.5pt);
\draw [fill=aqaqaq] (2.,-2.) circle (4.5pt);
\draw [fill=aqaqaq] (4.,-2.) circle (4.5pt);
\draw [fill=aqaqaq] (6.,-2.) circle (4.5pt);
\draw [fill=aqaqaq] (6.,0.) circle (4.5pt);
\draw [fill=aqaqaq] (6.,2.) circle (4.5pt);
\draw [fill=aqaqaq] (4.,2.) circle (4.5pt);
\draw [fill=aqaqaq] (2.,2.) circle (4.5pt);
\draw [fill=aqaqaq] (2.,0.) circle (4.5pt);
\draw [fill=aqaqaq] (4.,0.) circle (4.5pt);
\end{scriptsize}
\end{tikzpicture}
\caption{Triangulation and direction in vertices}\label{Fig:Triangulization}
\end{figure}
%}

\subsection{Estimates of piecewise quadratic homeomorphisms around the vertices}

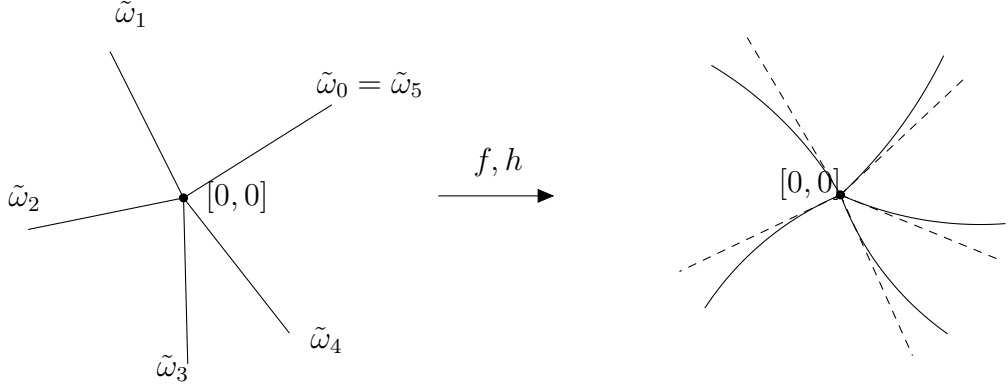
\begin{figure}
		\centering
\begin{tikzpicture}[line cap=round,line join=round,>=triangle 45,x=1.0cm,y=1.0cm]
\clip(-2.1310926041610667,-4.698703502974346) rectangle (12.690486994105237,1.1935927080873935);
\draw (0.9,-2.04)-- (2.8466178914461615,-0.807899900184615);
\draw (0.9,-2.04)-- (-0.06430412350770845,-0.10334552713845621);
\draw (0.9,-2.04)-- (-1.1396765876307942,-2.4580404054769343);
\draw (0.9,-2.04)-- (0.9554456269538384,-4.237967242646178);
\draw (0.9,-2.04)-- (2.2903907548307725,-3.8300673424615597);
\draw [->] (4.255726637538481,-2.0130586961846237) -- (5.776080810953878,-2.0130586961846237);
\draw (4.569264338288268,-1.2351575127529553) node[anchor=north west] {$f, h$};
\draw (2.5076160483038574,-0.21826342377415897) node[anchor=north west] {$\tilde{\omega}_0=\tilde{\omega}_5$};
\draw (-0.1390945942436969,0.7150503291241883) node[anchor=north west] {$\tilde{\omega}_1$};
\draw (-1.5460302515979232,-1.7227094732222412) node[anchor=north west] {$\tilde{\omega}_2$};
\draw (0.4041775902792221,-3.9793785473943646) node[anchor=north west] {$\tilde{\omega}_3$};
\draw (2.424035712223408,-3.603267035032344) node[anchor=north west] {$\tilde{\omega}_4$};
\draw (1.0449601668959985,-1.6669892491686085) node[anchor=north west] {$[0,0]$};
\draw [shift={(5.238394578892335,2.585085633169255)}] plot[domain=5.465898545977723:5.8318153895435,variable=\t]({1.0*6.287035824806309*cos(\t r)+-0.0*6.287035824806309*sin(\t r)},{0.0*6.287035824806309*cos(\t r)+1.0*6.287035824806309*sin(\t r)});
\draw [shift={(11.115861322461614,-5.684157797846188)}] plot[domain=1.9749854112537033:2.5644855026627082,variable=\t]({1.0*4.007038505812159*cos(\t r)+-0.0*4.007038505812159*sin(\t r)},{0.0*4.007038505812159*cos(\t r)+1.0*4.007038505812159*sin(\t r)});
\draw [shift={(5.3125581971077205,-4.478999001846179)}] plot[domain=0.5303642527029514:1.0276704731858834,variable=\t]({1.0*4.900683651083333*cos(\t r)+-0.0*4.900683651083333*sin(\t r)},{0.0*4.900683651083333*cos(\t r)+1.0*4.900683651083333*sin(\t r)});
\draw [shift={(11.375433986215462,2.0844812102154053)}] plot[domain=4.29006098782213:4.78632801804263,variable=\t]({1.0*4.477924147901289*cos(\t r)+-0.0*4.477924147901289*sin(\t r)},{0.0*4.477924147901289*cos(\t r)+1.0*4.477924147901289*sin(\t r)});
\draw [shift={(13.229524441600093,-0.6410317591999984)}] plot[domain=3.494504229445001:4.092671126462362,variable=\t]({1.0*3.931842479635665*cos(\t r)+-0.0*3.931842479635665*sin(\t r)},{0.0*3.931842479635665*cos(\t r)+1.0*3.931842479635665*sin(\t r)});
\draw [dash pattern=on 3pt off 3pt] (9.54,-2.0)-- (11.190024940676999,-0.43708180910768935);
\draw [dash pattern=on 3pt off 3pt] (9.54,-2.0)-- (8.316184734830822,0.08206351840000664);
\draw [dash pattern=on 3pt off 3pt] (9.54,-2.0)-- (7.426221316246199,-3.014267542092323);
\draw [dash pattern=on 3pt off 3pt] (9.54,-2.0)-- (10.48547056763084,-4.1267218153231005);
\draw [dash pattern=on 3pt off 3pt] (9.54,-2.0)-- (11.597924840861618,-2.8473994011077064);
\draw (8.581120470149823,-1.4998285770077104) node[anchor=north west] {$[0,0]$};
\begin{scriptsize}
\draw [fill=black] (0.9,-2.04) circle (1.5pt);
\draw [fill=black] (9.54,-2.0) circle (1.5pt);
\end{scriptsize}
\end{tikzpicture}
\caption{Mapping $f$ maps rays $[0,\rho_0)\times\tilde{\omega}_i$ onto quadratic curves (bold on the right side) and $h$ maps these rays onto touching segments (dashed on the right side).}\label{Fig:Vertex}
\end{figure}

	\begin{lemma}\label{GirlsNames}
		Let $Q_1, Q_2,\dots Q_N:\er^2\to\er^2$ be quadratic mappings with $Q_i(0,0) = [0,0]$. Let $0\leq \omega_0< \omega_1< \dots< \omega_{N-1} < \omega_N = \omega_0+ 2\pi < 4\pi$ and let $\tilde{\omega}_i = [\cos\omega_i, \sin\omega_i]\in \S$ be angles ordered anti-clockwise around $\S$. Let $R>0$ and $f:B(0,R)\to\er^2$ be the map defined by 
		$$
		f(t\cos\theta, t\sin\theta) = Q_i(t\cos\theta, t\sin\theta)\text{ for all }0\leq t\leq R \text{ and all }\omega_{i-1} \leq \theta \leq \omega_{i}.
		$$
		Further assume that this $f$ is a homeomorphism and $\det DQ_i \geq d >0$ on $B(0,R)$. Let $L$ and $M$ denote positive numbers such that $|DQ_i|\leq L$ on $B(0,R)$ and $|D^2Q_i|\leq M$. Then the map
		\eqn{aha}
		$$
			h(t\cos\theta, t\sin \theta) = DQ_i(0,0)[t\cos\theta, t\sin \theta] \qquad 0\leq t\leq R \text{ and } \omega_{i-1} \leq \theta \leq \omega_{i}
		$$
		is a piecewise linear homeomorphism (see Fig. \ref{Fig:Vertex}). Moreover
		\begin{equation}\label{Linda}
			h\bigl(t \cos\theta,t\sin\theta\bigr) = t\bigl(D_{[\cos\theta,\sin\theta]}h\bigl(\cos\theta,\sin\theta\bigr)\bigr)
		\end{equation}
		for all $0<t\leq R$ and all $\theta\in \er$. Further it holds that %for $0<R<\tfrac{d}{8LM}$ we have 
		\begin{equation}\label{Andrea}
		 \frac{d}{L} \leq |D_{w}h(x, y)| \leq L\text{ and }\frac{d}{L} \leq |D_{w}f(x ,y)| \leq L \quad \text{ for all } w\in \S
		\end{equation}
		and
		\begin{equation}\label{Becca}
			\bigl|h(x,y) - f(x,y)\bigr| \leq\frac{M}{2}\bigl|[x,y]\bigr|^2
		\end{equation}
		for all $[x,y] \in B(0,R)$. Moreover for any pair $u\bot v \in \S$, $u$ clockwise from $v$, denoting $\vec{v} = \tfrac{D_v f(x,y)}{|D_v f(x,y)|}$ and $\vec{u} \bot \vec{v}$ is clockwise from $\vec{v}$, it holds that
		\begin{equation}\label{Simona}
			\frac{d}{L} \leq \bigl\langle D_{u}f(x ,y), \vec{u}\bigr \rangle
		\end{equation}
		for all $[x,y] \in B(0,R)$. For the mapping $f$ we have 
		\begin{equation}\label{Cathy}
			\frac{d}{L}t - \frac{M}{2}t^2 \leq \bigl|f(t\cos \theta, t\sin\theta)\bigr| \leq Lt
		\end{equation}
		especially if $|[x,y]| \leq \tfrac{d}{LM}$ we have 
		\begin{equation}\label{Dolly}
			\frac{d}{2L} \bigl|[x,y]\bigr| \leq |f(x,y)| \leq L\bigl|[x,y]\bigr|.
		\end{equation}
		Finally
		\begin{equation}\label{Ella}
			|Df(x,y)-Dh(x,y)| \leq M\bigl|[x,y]\bigr|.
		\end{equation}
		
		Similarly if $r,\ell >0$ and $f(x,y)=Q_1(x,y)$ on $[-r,0]\times[0,\ell]$ and $f(x,y)=Q_2(x,y)$ on $[0,r]\times[0,\ell]$ is a homeomorphism with $|Df|\leq L$, $J_f\geq d$ and $|D^2f|\leq M$ then
		\begin{equation}\label{Sandra}
			\frac{d}{L} \leq |D_{w}f(x ,y)| \leq L \quad \text{ for all } w\in \S.
		\end{equation}
		Further, for any pair $u\bot v \in \S$, $u$ clockwise from $v$, denoting $\vec{v} = \tfrac{D_v f(x,y)}{|D_v f(x,y)|}$ and $\vec{u} \bot \vec{v}$ is clockwise from $\vec{v}$, we have 
		\begin{equation}\label{Ramona}
			\frac{d}{L} \leq \bigl\langle D_{u}f(x ,y), \vec{u}\bigr \rangle
		\end{equation}
		for all $[x,y] \in [-r,r]\times[0,\ell]$.
	\end{lemma}
	
	\begin{proof}
		First we prove that the map $h$ defined by \eqref{aha} 
		%$$
		%h(t\cos\theta, t\sin\theta) = DQ_i(0,0)(t\cos\theta, t\sin\theta)$ for $t\in [0,\infty)$ and $\omega_{i-1}\leq \theta \leq \omega_i$ 
		is a piecewise linear homeomorphism. Denote $\tilde{\theta} = [\cos\theta,\sin\theta]$ and 
		note that on each $\omega_i\leq \theta\leq \omega_{i+1}$ we have $Dh(t\tilde{\theta}) = DQ_i(0,0)$. Since $\det DQ_i(0,0)\geq d>0$ for all $1\leq i \leq N$ we know that $h$ is a homeomorphism on each $\omega_i \leq \theta \leq \omega_{i+1}$. 
		Further $f$ is continuous on $\tilde{\omega}_{i}\times [0,\rho_0]$ and hence $ \partial_{\tilde{\omega}_i} Q_i(0,0) = \partial_{\tilde{\omega}_{i}} Q_{i+1}(0,0)$ which implies that $h$ is continuous. By the piecewise linearity of $h$, the continuity of $h$ and $J_h>0$ a.e. it is not difficult to deduce that $h$ is a homeomorphism (see Fig. \ref{Fig:Vertex}). 
		
%		Using $f(0,0) = h(0,0)=0$, $|h(t\tilde{\theta})|\geq \tfrac{d}{L}t$ (by \eqref{Candy}), 
%		$$
%		\lim_{t\to 0}Df(t\cos\theta,t\sin\theta) = DQ_i(0,0)= Dh(\cos\theta,\sin\theta)\text{ for all }\omega_i<\theta<\omega_{i+1}
%		$$ 
%		 and $|D^2f|\leq M$, we have
%		$$
%			\|f-h\|_{L^{\infty}(B(0,t))}\leq \frac{Mt^2}{2} \ \  \text{ and } \ \ |f(t\cos\theta, t\sin\theta)| \geq \frac{d}{L}t -\frac{M}{2}t^2 \geq \frac{15d}{16L}t
%		$$
%		for all $0\leq t \leq R$ since $0<R<\tfrac{d}{8LM}$.

		The equality \eqref{Linda} is obvious from the piece-wise linearity of $h$. 
		Since $|Df| \leq L$ and $J_f\geq d$ we obtain for any pair $u,v\in \S$ with $u\bot v$ that 
		$$
		d\leq J_f(x,y) \leq |D_u f(x,y)|\cdot |D_v f(x,y)| \leq L|D_v f(x,y)|,
		$$ 
		which shows \eqref{Andrea} for $f$.
		Analogously $J_h\geq d$ a.e. (as $\det D Q_i(0,0) \geq d$ for all $i$) and $|D_{w} h|\leq L$ for any $w\in \S$ imply \eqref{Andrea} for $h$. 
		%Further we have from $\det D Q_i(0,0) \geq d$ that $J_h \geq d$ almost everywhere. 
		
		For all $\omega_i \leq \theta \leq \omega_{i+1}$ we have $DQ_i(0,0) = \lim_{s\to 0}Df(s\cos\theta, s\sin\theta)$. For any $\omega_i \leq \theta \leq \omega_{i+1}$, calling $\tilde{\theta} = [\cos\theta, \sin\theta]$ we have using 
		$D_{\tilde{\theta}} Q_i(0, 0) = D_{\tilde{\theta}} h(s\tilde{\theta})$
		\eqn{secondder}
		$$
		\begin{aligned}
			f(t\tilde{\theta})- h(t\tilde{\theta})&= \int_{0}^{t} D_{\tilde{\theta}} f(s\tilde{\theta})-\int_{0}^{t} D_{\tilde{\theta}} Q_i(0, 0) \ ds\\
			&= \int_{0}^{t} \int_0^s D_{\tilde{\theta}\tilde{\theta}}f(z\tilde{\theta}) dz\ ds\\
			\end{aligned}
		$$
		but since $|D_{\tilde{\theta}\tilde{\theta}}f| \leq M$ we have
		$$
			\bigl|f(t\tilde{\theta}) - h(t\tilde{\theta})\bigr| \leq \frac{M}{2}t^2
		$$
		which is \eqref{Becca}.
		
		Since $\vec{v} = \tfrac{D_v f(x,y)}{|D_v f(x,y)|}$ and $\vec{u} \bot \vec{v}$ we obtain $\langle D_v f, \vec{u}\rangle =0$. Thus we can use \eqref{JackOfAllTrades} to obtain  
		$$
			d\leq J_f(x,y) = \langle D_{u}f(x ,y), \vec{u}\rangle \langle D_{v}f(x ,y), \vec{v}\rangle
		$$
		and using \eqref{Andrea} we get \eqref{Simona}. The equation \eqref{Cathy} follows from \eqref{secondder} with the help of $|D^2f|\leq M$ and \eqref{Andrea}. Further \eqref{Dolly} follows immediately from \eqref{Cathy}. The equation \eqref{Ella} follows from the fact that $|D^2f|\leq M$ and $Dh(\tilde{\theta}) = DQ_i(0,0)$ for $\omega_i < \theta< \omega_{i+1}$. The proof of \eqref{Sandra} and \eqref{Ramona} is analogous to that of \eqref{Andrea} and \eqref{Simona}.
	\end{proof}	

	\begin{lemma}\label{Partials}
		Let $f\in W^{1,\infty}(B(0,R), \er^2)$ and $f$ is $\mathcal{C}^1$ smooth except on a finite number of rays $\tilde{\omega}_1\er^+, \tilde{\omega}_2\er^+, \dots \tilde{\omega}_N\er^+$, $\tilde{\omega}_i \in \S$, $f(0,0)=0$ and $|f(x,y)|>0$ for $[x,y] \neq [0,0]$. Let $\R:B(0,R) \to [0,\infty)$ and $\phi:B(0,R)\to\S$ be a pair of functions such that for all $[x,y] \in B(0,R)$ we have
		$$
			f(x,y) = \R(x,y)\phi(x,y). 
		$$
		Then for any $t\in(0,R)$ and any $\theta \in [0,2\pi)$ (calling $\tilde{\theta} = [\cos\theta,\sin\theta]$, $\tilde{\theta}^{\bot} = [-\sin\theta,\cos\theta]$ and calling $\phi^{\bot}(t\tilde{\theta}) \in \S$ the vector anti-clockwise perpendicular to $\phi(t\tilde{\theta})$) it holds that
		$$
			\bigl\langle \frac{\partial}{\partial \theta}\phi(t\cos\theta, t\sin\theta) , \phi^{\bot}(t\tilde{\theta}) \bigr\rangle = \frac{t}{\R(t\tilde{\theta})}\bigl\langle D_{\tilde{\theta}^{\bot}} f(t\tilde{\theta}), \phi^{\bot}(t\tilde{\theta}) \bigr\rangle.
		$$
	\end{lemma}

	\begin{proof}
		Without loss of generality we may assume that $\tilde{\theta} = e_1 = \phi(t,0)$ and $\tilde{\theta}^{\bot} = e_2 = \phi^{\bot}(t,0)$ (just consider suitable rotations). 
		Since $\phi = \tfrac{f}{|f|}$, $f$ is Lipschitz and $|f(x,y)|>0$ for $|[x,y]|>0$ we obtain that $\phi$ is locally Lipschitz outside of $0$. This and the fact that 
		$$
		\bigl|[t\cos\theta, t\sin\theta] - [t, t\tan\theta]\bigr | \leq \theta^2\text{ for small }\theta
		$$ 
	implies 
	\eqn{using}
		$$
			\begin{aligned} \ 
				\Bigl[\frac{\partial}{\partial \theta}\phi(t\cos\theta, t\sin\theta) \Bigr]_{\theta = 0} 
				= & \lim_{\theta\to 0} \frac{\phi(t\cos\theta, t\sin\theta) - \phi(t,0)}{\theta}\\
				= & \lim_{\theta\to 0} \frac{\phi(t\cos\theta, t\sin\theta) - \phi(t, t\tan\theta)}{\theta} \\
				&+ \lim_{\theta\to 0} \frac{t \tan\theta}{\theta}\frac{\phi(t, t\tan\theta) - \phi(t,0)}{t\tan\theta}\\
				= & t D_{\tilde{\theta}^{\bot}}\phi(t,0)
			\end{aligned}
		$$
		because $\frac{t \tan\theta}{\theta} \to t$. 
		Now
		$$
		D_{\tilde{\theta}^{\bot}} f(t,0)=D_yf(t,0)=D_yR(t,0) \phi(t,0)+R(t,0)D_y\phi(t,0)
		$$
		and hence
		$$
		\bigl\langle D_{\tilde{\theta}^{\bot}} f(t,0),\phi^{\bot}(t,0)\bigr\rangle=
		\bigl\langle R(t,0)D_{\tilde{\theta}^{\bot}}\phi(t,0),\phi^{\bot}(t,0)\bigr\rangle
		$$
		and our conclusion follows using \eqref{using}. 
	%	The fact that 
		%$$\bigl\langle D_{\tilde{\theta}^{\bot}}\phi(t,0)  , \phi^{\bot}(t, 0) \bigr\rangle = \tfrac{1}{\R(t,0)}\bigl\langle D_{\tilde{\theta}^{\bot}} f(t,0),\phi^{\bot}(t,0)\bigr\rangle
		%$$ 
		%is a question of simple trigonometry.
	\end{proof}

\section{Approximation of piecewise quadratic homeomorphisms around the edges}
	
	Recall that $\eta$ denotes the function from the Preliminaries, Notation~\ref{Notation}.

	\begin{lemma}[Approximation along the edge]\label{Reptiles}
		Let $Q_1,Q_2:\er^2 \to \er^2$ be a pair of quadratic mappings coinciding on the line $\{x=0\}$ and let $\rho_0, \ell >0$ be such that the map $f =Q_1$ on $[-\rho_0, 0]\times[0, \ell]$ and $f =Q_2$ on $[0,\rho_0]\times[0, \ell]$ is a homeomorphism with 
		$$
		d = \min\bigl\{\det DQ_1(x,y), \det DQ_2(x,y); [x,y]\in[-\rho_0,\rho_0]\times[0,\ell]\bigr\} >0.
		$$ 
		 Let $L$ and $M$ denote positive numbers such that $|DQ_j|\leq L$ and $|D^2Q_j|\leq M$ on $[-\rho_0,\rho_0]\times[0,\ell]$ for $j=1,2$.
		Fix $N\in \en, N\geq 4$ such that 
		\eqn{defrho}
		$$
		\rho = \frac{\ell}{2N} <\min\Bigl\{\rho_0, \frac{d}{1000 (M+1) (L+1)}
		,\frac{1}{320}\frac{d^2}{ML^3}\Bigr\}.
		$$ 
		Then there exists an $r_0$ (depending on the geometry of $f(\{0\}\times[0,\ell])$) such that for every 
		$0< r < \min\{r_0, \tfrac{\rho^2}{2(L+1)},\frac{\rho}{40}, \tfrac{2M \rho^2}{L}\}$ there exists a diffeomorphism $g : [-\rho_0,\rho_0]\times[0,\ell]$ satisfying
		$$
			g(x,y) = f(x,y) \text{ for all } |x| > r \text{ and } y\in [0,\ell]
		$$
		and 
		\begin{equation}\label{Iguana}
			\int_{[r,r]\times[0,\ell]} |D^2g| \leq C\int_0^{\ell} |D_{xx} f(0,y)| + C M\ell r
		\end{equation}
		%\begin{equation}\label{Iguana}
	%		\int_{[r,r]\times[0,\ell]} |D^2g| \leq 8\int_0^{\ell} |D_{xx} f(0,y)| + 22 M\ell r
%		\end{equation}
		where by $|D_{xx} f(0,y)| = |D_xQ_2(0,y) - D_xQ_1(0,y)|$ we denote the size of the Dirac measure of $D_{xx}f$ at $[0,y]$.
		Further, call $\vec{u}(x,y) \in \S$ the vector clockwise perpendicular to $\tfrac{D_y g(0,y)}{|D_y g(0,y)|}$ for all $[x,y]\in [-r,r]\times[0,\ell]$, then we have
		\begin{equation}\label{Gecco}
			\bigl\langle D_x g(x,y), \vec{u}(x,y) \bigr\rangle \geq \frac{9d}{10L}.
		\end{equation}
	\end{lemma}
	
	\begin{proof}
		\step{1}{Initial setup and definition of $\tilde{f}_r$}{TheEgg}
		
		%Call 
		%$$
		%L = \max\bigl\{|DQ_1(x,y)|, |DQ_2(x,y)|: [x,y]\in [-\rho_0,\rho_0]\times[0,\ell]\bigr\}%, v\in \S\bigr\}
		%$$ 
		%and call $M = \max \{|D^2Q_1|, |D^2Q_2|\}$. 
		We have $\rho$ which satisfies \eqref{defrho} and $\tfrac{\ell}{2\rho} = N \in \en$. 
		We divide $[-r,r]\times[0,\ell]$ into $N$ rectangles $[-r,r]\times[(2i-2)\rho,2i\rho]$ for $i=1,\hdots,N$. 
		We denote 
		\eqn{defvi}
		$$
		\vec{v}_i = \frac{\partial_yQ_1(0,(2i-1)\rho)}{|\partial_yQ_1(0,(2i-1)\rho)|}\text{ for }i=1,2,\dots, N
		\text{ and we fix }\vec{u}_i\in \S , \vec{u}_i\bot \vec{v}_i
		$$ 
		so that $\{\vec{u}_i,\vec{v}_i\}$ is a positively oriented basis of $\er^2$. That is $\vec{u}_i$ is clockwise purpendicular to $\vec{v}_i$. Then we define a tentative map $\tilde{f}_r$ as an appropriate convex combination of $Q_1$ and $Q_2$, i.e.
		\begin{equation}\label{RozsekniTo}
			%\begin{aligned}
				\tilde{f}_r(x,y) =
				\begin{cases}
				(1-\eta(\tfrac{x}{r}))Q_1(x,y) + \eta(\tfrac{x}{r})Q_2(x,y)\\
				\phantom{ahoj}\text{if } a) \  \langle D_xQ_2(0,(2i-1)\rho), \vec{u}_i\rangle \geq \langle D_xQ_1(0,(2i-1)\rho), \vec{u}_i\rangle, \\
				%\tilde{f}_r(x,y) =
				(1-\eta(\tfrac{x+r}{r}))Q_1(x,y) + \eta(\tfrac{x+r}{r})Q_2(x,y)\\
				\phantom{ahoj}\text{if } b) \ \langle D_xQ_1(0,(2i-1)\rho), \vec{u}_i\rangle > \langle D_xQ_2(0,(2i-1)\rho), \vec{u}_i\rangle.
			  \end{cases}
			%\end{aligned}
		\end{equation}
		for every $[x,y]\in [-r,r]\times[(2i-2)\rho,2i\rho]$. Note that 
		$\tilde{f}_r=Q_1$ for $x<-r$ and $\tilde{f}_r=Q_2$ for $x>r$.

		\step{2}{Define $g$}{The Hatchling}
		
		The above definition of $\tilde{f}_r$ is fine if all rectangles $[-r,r]\times[(2i-2)\rho,2i\rho]$ are of type a) or if all of them are of type b). Otherwise we have to continuously connect rectangles of type a) to rectangles of type b).  
		
		Suppose that we have a pair of neighboring rectangles $(-r,r)\times((2i-2)\rho, 2i\rho)$ of type $a)$ and $(-r,r)\times(2i\rho, (2i+2)\rho)$ of type $b)$ then we define $g$ as follows
		\begin{equation}\label{Terapin1}
			\begin{aligned}
				g(x,y) = \bigl[1-\eta\big(\tfrac{x}{r}+\eta(y\rho^{-1}-2i)\big)\bigr]Q_1(x,y) 
				+ \eta\big(\tfrac{x}{r}+\eta(y\rho^{-1}-2i)\big)Q_2(x,y)
			\end{aligned}
		\end{equation}
		for all $[x,y]\in (-\rho_0,\rho_0)\times[2i\rho, (2i+1)\rho]$. Note that for 
		$y_1=2i\rho$ and $y_2=(2i+1)\rho$ we have
		$$
		\eta\big(\tfrac{x}{r}+\eta(y\rho^{-1}-2i)\big)=
		\eta\big(\tfrac{x}{r}\big)
		\text{ and }\eta\big(\tfrac{x}{r}+\eta(y\rho^{-1}-2i)\big)=
		\eta\big(\tfrac{x + r}{r}\big)
		$$
		and so it agrees with \eqref{RozsekniTo} there. 
		
		Similarly when we have a pair of adjacent rectangles $(-r,r)\times((2i-2)\rho, 2i\rho)$ of type $b)$ and $(-r,r)\times(2i\rho, (2i+2)\rho)$ of type $a)$ then we define $g$ as follows
		\begin{equation}\label{Terapin2}
			\begin{aligned}
				g(x,y) = \bigl[1-\eta\big(\tfrac{x + r}{r} -\eta(y\rho^{-1}-2i)\big)\bigr]Q_1(x,y) 
				+ \eta\big(\tfrac{x + r}{r} -\eta(y\rho^{-1}-2i)\big)Q_2(x,y)
			\end{aligned}
		\end{equation}
		on $(-\rho_0,\rho_0)\times[2i\rho, (2i+1)\rho]$. On the rest of $[-\rho_0,\rho_0]\times[0,\ell]$ we let $g(x,y) = \tilde{f}_r(x,y)$. Immediately we see from the smoothness of $Q_1, Q_2$ and $\eta$ that $g$ is smooth on $(-\rho_0,\rho_0)\times(0,\ell)$.
		
		\step{3}{The injectivity of $g$}{Little Lizzard}
		
		We firstly show that $J_g>0$ which shows that $g$ is locally a homeomorphism. Secondly we show that $g$ is injective on $\partial([-\rho_0,\rho_0]\times[0,\ell])$. Together these two facts imply that the smooth mapping $g$ is in fact a diffeomorphism (see e.g. \cite{Kr}).

		The following calculations are for rectangles where $a)$-type transfers into $b)$-type and $g$ is given by \eqref{Terapin1}. The calculations are analogous for \eqref{Terapin2} and are even simpler on rectangles where $g\equiv \tilde{f}_r$. Since $Q_1$, $Q_2$ and $\eta$ are smooth we immediately get that $\tilde{f}_{r}$ is smooth on each rectangle $(-\rho_0,\rho_0)\times((2i-2)\rho, 2i\rho)$. Further
		\begin{equation}\label{Crocodile}
		\begin{aligned}
			D_x\tilde{f}_r(x,y) &= (1-\eta(\tfrac{x}{r}))D_x Q_1(x,y) + \eta(\tfrac{x}{r})D_x Q_2(x,y) + \tfrac{1}{r}\eta'(\tfrac{x}{r}) (Q_2(x,y) - Q_1(x,y))\\
			\text{ and }D_y\tilde{f}_r(x,y) &= (1-\eta(\tfrac{x}{r}))D_y Q_1(x,y) + \eta(\tfrac{x}{r})D_y Q_2(x,y).\\
			\end{aligned}
		\end{equation}
		If $(-r,r)\times((2i-2)\rho, 2i\rho)$ is an $a)$-type rectangle and $(-r,r)\times(2i\rho, (2i+2)\rho)$ is a $b)$-type rectangle then on $(-\rho_0,\rho_0)\times (2i\rho, (2i+1)\rho)$ we calculate
		\begin{equation}\label{Aligator}
			\begin{aligned}
				D_xg(x,y) =&  \big[1-\eta\big(\tfrac{x}{r}+\eta(y\rho^{-1}-2i)\big)\big]D_x Q_1(x,y)
				  + \eta\big(\tfrac{x}{r}+\eta(y\rho^{-1}-2i)\big)D_x Q_2(x,y) \\
				& +  \tfrac{1}{r}\eta'\big(\tfrac{x}{r}+\eta(y\rho^{-1}-2i)\big)(Q_2(x,y) - Q_1(x,y)).
			\end{aligned}
		\end{equation}
		Moreover we calculate
		\begin{equation}\label{Dinosaur}
			\begin{aligned}
				D_y g(x,y) =&  \big[1-\eta\big(\tfrac{x}{r}+\eta(y\rho^{-1}-2i)\big)\big]D_y Q_1(x,y)
				  + \eta\big(\tfrac{x}{r}+\eta(y\rho^{-1}-2i)\big)D_y Q_2(x,y) \\
			%	& +  \tfrac{1}{\rho}\eta'\big(\tfrac{x + r\eta(y\rho^{-1}-2i)}{r}\big)(Q_2(x,y) - Q_1(x,y)).\\
				& +  \tfrac{1}{\rho}\eta'\big(\tfrac{x}{r}+\eta(y\rho^{-1}-2i)\big)\eta'(y\rho^{-1}-2i)(Q_2(x,y) - Q_1(x,y)).
			\end{aligned}
		\end{equation}
		Because $Q_1(0,y) = Q_2(0,y)$ for $y\in [0,\ell]$ we have
		$$
			|Q_2(x,y) - Q_1(x,y)| \leq \int_0^x |D_xQ_2(s,y)| + |D_xQ_2(s,y)| \ ds \leq 2Lr.
		$$
		Utilizing this fact, \eqref{Crocodile}, \eqref{Aligator}, \eqref{Dinosaur}, $0\leq \eta\leq 1$, $|\eta'|\leq 2$ and $r < \tfrac{\rho}{40}$ we get that
		\begin{equation}\label{MuskTurtle}
			|Dg| \leq 8L.
		\end{equation}
		on each $(-r,r)\times((2i-2)\rho, 2i\rho)$. %Further, because $g(0,(2i+1)\rho) = Q_1(0,(2i+1)\rho) = Q_2(0,(2i+1)\rho)$ we have that
		Since $g$ is a convex combination of $Q_1$ and $Q_2$, $f$ is equal either to $Q_1$ or $Q_2$ and $Q_1=Q_2$ on ${0}\times[0,\ell]$ we get
		\begin{equation}\label{BeardedDragon}
			\|g - f\|_{\infty}\leq \|Q_1(x,y)-Q_1(0,y)\|_{\infty}+\|Q_2(x,y)-Q_2(0,y)\|_{\infty} \leq 2Lr.
		\end{equation}
		
		Using $Q_1(0,y) = Q_2(0,y)$ we also have 
		$$
		\bigl\langle Q_2(x,y) - Q_1(x,y), \vec{u}_i \bigr\rangle = \int_0^x \langle D_xQ_2(s,y),\vec{u}_i \rangle  - \langle D_xQ_1(s,y), \vec{u}_i \rangle \ ds.
		$$
		%Assume that $\langle D_xQ_2(0,(2i-1)\rho),\vec{u}_i\rangle \geq  \langle D_xQ_1(0,(2i-1)\rho), \vec{u}_i\rangle$ 
		%(which holds for a) type rectangles) and 
		Use $|D^2Q_{j}|\leq M$ for $j=1,2$ and $r\leq \tfrac{\rho}{40}$ to get 
		\eqn{ttt}
		$$
		\bigl|D Q_{j}(s,y) - D Q_{j}(0,(2i-1)\rho)\bigr| \leq 2M\rho \text{ for }s\in[-r,r]\text{ and }y\in[(2i-2)\rho,2i\rho]
		$$ 
	and hence with the help of $\rho\leq \frac{d}{1000 ML}$ 
		$$
		\bigl\langle Q_2(x,y) - Q_1(x,y), \vec{u}_i \bigr\rangle \geq - 4M\rho x \geq -\frac{dr}{250L}.
		$$
		%We have $0\leq (\eta(\tfrac{x}{r}))'\leq \tfrac{2}{r}$ and 
		From \eqref{Ramona} at the point $[0,(2i-1)\rho]$ for $v=[0,1]$ and $u=[1,0]$ we obtain
		$$
		\bigl\langle D_xQ_1(0,(2i-1)\rho), \vec{u}_i \bigr\rangle\geq \frac{d}{L}
		$$
		and hence we can combine it with the previous inequality to obtain
		\eqn{tttt}
		$$
			\frac{2}{r} \bigl\langle Q_2(x,y) - Q_1(x,y), \vec{u}_i\bigr\rangle \geq 
			-\frac{1}{125}\bigl\langle D_xQ_1(0,(2i-1)\rho), \vec{u}_i \bigr\rangle.
		$$
		Using \eqref{Crocodile}, $\langle D_xQ_2(0,(2i-1)\rho),\vec{u}_i\rangle \geq  \langle D_xQ_1(0,(2i-1)\rho), \vec{u}_i\rangle$ 
		(which holds for a) type rectangles) and \eqref{ttt} we obtain (for $[x,y]$ where $g=\tilde{f}_r$) 
		\eqn{jjj}
		$$
		\begin{aligned}
		\langle D_xg(x,y), \vec{u}_i\rangle=&\bigl\langle (1-\eta(\tfrac{x}{r}))D_x Q_1(x,y) + \eta(\tfrac{x}{r})D_x Q_2(x,y) + \tfrac{1}{r}\eta'(\tfrac{x}{r}) (Q_2(x,y) - Q_1(x,y)), \vec{u}_i\bigr\rangle\\
		\geq &\bigl\langle (1-\eta(\tfrac{x}{r}))D_x Q_1(0,(2i-1)\rho) + \eta(\tfrac{x}{r})D_x Q_2(0,(2i-1)\rho), \vec{u}_i\bigr\rangle\\
		&-|D_x Q_1(0,(2i-1)\rho)-D_x Q_1(x,y)|-|D_x Q_2(0,(2i-1)\rho)-D_x Q_2(x,y)|\\
		&+\tfrac{1}{r}\eta'(\tfrac{x}{r}) \bigl\langle Q_2(x,y) - Q_1(x,y), \vec{u}_i\bigr\rangle\\
		\geq& \bigl\langle D_x Q_1(0,(2i-1)\rho), \vec{u}_i\bigr\rangle-4M\rho+\tfrac{1}{r}\eta'(\tfrac{x}{r})\bigl\langle Q_2(x,y) - Q_1(x,y), \vec{u}_i\bigr\rangle,\\
		\geq& \bigl\langle D_x Q_1(0,(2i-1)\rho), \vec{u}_i\bigr\rangle \Bigl(1-\frac{1}{250}-\frac{1}{125}\Bigr),\\
		\end{aligned}
		$$
		where we have used \eqref{Ramona} and $\rho\leq\frac{d}{1000 ML}$ to estimate the term $4M\rho$ and the term 
		$\tfrac{1}{r}\eta'(\tfrac{x}{r}) \bigl\langle Q_2(x,y) - Q_1(x,y), \vec{u}_i\bigr\rangle$ is either positive and then we can estimate it by $0$ or it is negative and then we use $|\eta'|\leq 2$ and \eqref{tttt}. Similarly we can use \eqref{Aligator} and also in this case we obtain 
		\begin{equation}\label{BlackMamba}
		\langle D_xg(x,y), \vec{u}_i\rangle \geq  
		\frac{123}{125}\bigl\langle D_xQ_1(0,(2i-1)\rho), \vec{u}_i \bigr\rangle 
		\end{equation}
		on $(-r,r)\times(0, \ell)$ which together with \eqref{Ramona} implies
		\eqn{ppp}
		$$
		\langle D_xg(x,y), \vec{u}_i\rangle \geq  
		\frac{123}{125}\frac{d}{L}. 
		$$
		Now $d\leq J_{Q_1}\leq |D_x Q_1| |D_y Q_2|$ implies $|D_xQ_1|\geq \frac{d}{L}$. Recall that $\vec{u} = \vec{u}(y)$ is the vector in $\S$ clockwise purpendicular to $\frac{D_yQ_i(0,y)}{|D_yQ_i(0,y)|}$. Using \eqref{defvi}, \eqref{elementary}, \eqref{MuskTurtle} and \eqref{defrho} we obtain
		$$
		\begin{aligned}
		\bigl|\langle D_xg(x,y), \vec{u}_i\rangle-\langle D_xg(x,y), \vec{u}\rangle \bigr|
		&\leq |D_xg(x,y)|\ |\vec{u}_i-\vec{u}|=|D_xg(x,y)|\ |\vec{v}_i-\vec{v}|\\
		&\leq 8L \frac{|D_y Q_1(0,(2i-1)\rho)-D_yQ_1(0,y)|}{|D_y Q_1(0,(2i-1)\rho)|}2\\
		&\leq 8L \frac{M\rho}{\frac{d}{L}}2\leq \frac{1}{20}
\frac{d}{L}		\end{aligned}
		$$
		which together with \eqref{ppp} imply \eqref{Gecco} since $g=Q_1$ in $[0,y]$ (see \eqref{RozsekniTo} and \eqref{Terapin1}). 
		%shows \eqref{Gecco}  NA PRAVE STRANE NENI Ui - TOTO JE POTREBA OPRAVIT.  
		In \eqref{Aligator} we dealt only with the $a)$-type to $b)$-type transitions but the calculations easily extend also for the $b)$-type to $a)$-type transitions. The only difference is that we use 
		$\langle D_xQ_1(0,(2i-1)\rho),\vec{u}_i\rangle \geq  \langle D_xQ_2(0,(2i-1)\rho), \vec{u}_i\rangle$ in \eqref{jjj} above and hence we have $\langle D_xQ_2(0,0), \vec{u}_i \rangle$ on the righthand side of \eqref{BlackMamba}.

%	We can calculate that for all $[x,y] \in (-r,r)\times((2i-2)\rho, 2i\rho)$ 
	%by $|D^2Q_i|\leq M$ that
	%\eqn{odhadder}
	%$$
	%	\bigl|D_yQ_{j}(x,y) - D_yQ_{j}(0, (2i-1)\rho)\bigr| \leq 2M\rho 
	%	\text{ for }j=1,2.
	%	$$
			By the definition of $\vec{v}_i$ \eqref{defvi}, $\vec{u}_i\bot \vec{v}_i$ and $Q_1=Q_2$ on $\{0\}\times [0,\ell]$ we have 
		$\langle D_y Q_{j}(0, (2i-1)\rho), \vec{u}_i\rangle = 0$. It follows using \eqref{ttt} that
		$$
		|\langle D_y Q_j (x,y),\vec{u}_i\rangle|\leq 2M\rho\text{ for }j=1,2. 
		$$
With the help of $r<\frac{2M \rho^2}{L}$ we obtain 
		$$
		\begin{aligned}
		|Q_1(x,y)-Q_2(x,y)|&\leq |Q_1(x,y)-Q_1(0,y)|+|Q_2(0,y)-Q_2(x,y)|  
		\leq 2Lr\leq 4M\rho^2
		\end{aligned}
		$$
		and hence 
		$$
			\bigl\langle Q_2(x,y) - Q_1(x,y), \vec{u}_i \bigr \rangle \leq 4M\rho^2.
		$$
		Applying this in \eqref{Dinosaur} we get that
		\eqn{last}
		$$
			\langle D_y g(x,y), \vec{u}_i\rangle \leq 2M\rho +  \frac{4}{\rho}4M\rho^2\leq 18M\rho.
		$$

		We can express the values of $Dg$ with respect to the basis $\{\vec{u}_i,\vec{v}_i\}$ as 
		$$
			Dg(x,y)= 
			\left(\begin{matrix}
				\langle D_x g(x,y), \vec{u}_i\rangle,& \langle D_x g(x,y), \vec{v}_i\rangle\\
				\langle D_y g(x,y), \vec{u}_i\rangle,& \langle D_y g(x,y), \vec{v}_i\rangle
			\end{matrix}\right)
			= \left(\begin{matrix}
				a_1,& a_2\\
				b_1, &b_2  
			\end{matrix}\right).
		$$
		Therefore, applying \eqref{BlackMamba}, \eqref{MuskTurtle}, \eqref{last}, $\rho<\tfrac{d}{1000LM}$, definition of $\vec{v_i}$ \eqref{defvi}, \eqref{ttt} and \eqref{Sandra} (i.e. $\langle D_y Q_{1}(0,(2i-1)\rho), \vec{v}_i\rangle\geq \tfrac{d}{L}$) we conclude that 
		\begin{equation}\label{Caiman}
			\begin{aligned}
				a_1 &> \frac{123}{125} \bigl\langle D_xQ_1(0,(2i-1)\rho), \vec{u}_i\bigr\rangle\\
				|a_2| &\leq 8L\\
				|b_1| &\leq 18M\rho \leq \frac{d}{50L}\\
				b_2& \geq |D_yQ_{1}(0,(2i-1)\rho)| - 2M\rho \geq \frac{499}{500}|D_yQ_{1}(0,(2i-1)\rho)|
			\end{aligned}
		\end{equation}
		on the entire rectangle $(-r,r)\times((2i-2)\rho, 2i\rho)$. 
		From the definition of $\vec{v}_i$ and $\vec{u_i}\bot \vec{v}_i$ we know that $\langle D_yQ_{1}(0,(2i-1)\rho),\vec{u}_i\rangle=0$ and 
		hence using \eqref{JackOfAllTrades} 
		$$
		\bigl\langle D_xQ_1(0,(2i-1)\rho), \vec{u}_i\bigr\rangle |D_yQ_{1}(0,(2i-1)\rho)|\geq \det DQ_1(0,(2i-1)\rho)\geq d. 
		$$
		Therefore simple computation gives
		\begin{equation}\label{KingCobra}
		J_{g}(x,y) \geq \frac{61377}{62500}d- \frac{8d}{50} \geq \frac{4}{5}d
		\end{equation}
		on $(-r,r)\times((2i-2)\rho, 2i\rho)$. 
		
		\step{4}{The injectivity of $g$}{The Reptile}
		
		By a combination of \eqref{BlackMamba} for $i=0$ and $i=N$ and the fact that $f$ is a homeomorphism %close to $g$ (see \eqref{BeardedDragon}) 
		we get that $g$ is injective on both segments $[-\rho_0,\rho_0]\times\{0\}$ and $[-\rho_0,\rho_0]\times\{\ell\}$. Because $f$ is a homeomorphism we have that
		$$
			\dist\Big(f([-r,r]\times\{0\}), f\big(\partial([-\rho_0,\rho_0]\times[0,\ell])\setminus \bigl([-\rho_0,\rho_0]\times\{0\}\bigr)\big) \Big) >0
		$$
		and similarly
		$$
			\dist\Big(f([-r,r]\times\{\ell\}), f\big(\partial([-\rho_0,\rho_0]\times[0,\ell])\setminus 
			\bigl([-\rho_0,\rho_0]\times\{\ell\}\bigr)\big) \Big) >0.
		$$
		Therefore, by \eqref{BeardedDragon} and $f(x,y)=g(x,y)$ for $|x|\geq r$ there exists an $r_0>0$ (this is the $r_0$ of our claim) such that for all $0<r<r_0$ the mapping $g$ constructed from $\tilde{f}_r$ satisfies
		$$
			g([-\rho_0,\rho_0]\times\{0\})\cap g\big(\partial([-\rho_0,\rho_0]\times[0,\ell])\setminus \bigl([-\rho_0,\rho_0]\times\{0\}\bigr)\big) = \emptyset
		$$
		and
		$$
			g([-\rho_0,\rho_0]\times\{\ell\})\cap g\big(\partial([-\rho_0,\rho_0]\times[0,\ell])\setminus \bigl([-\rho_0,\rho_0]\times\{\ell\}\bigr)\big)= \emptyset
		$$
		for all $r\leq r_0$. But together that means that $g$ is injective on $\partial([-\rho_0,\rho_0]\times[0,\ell])$. 
		Since \eqref{KingCobra} implies local injectivity this is enough to conclude that $g$ is injective everywhere in $[-\rho_0,\rho_0]\times[0,\ell]$ and thus a diffeomorphism (see e.g. \cite{Kr}).
		
		\step{5}{Estimates of $|D^2g|$}{The Carcass}
		
		%It is not hard to observe from the definition that the components of $D^2g$ on the set $\{\tilde{f}_r=g\}$ are estimated by the components of $D^2g$ from \eqref{Terapin1} and \eqref{Terapin2}. 
		We calculate the estimates of $D^2g$ in detail only for the a) to b) type transition given by \eqref{Terapin1}. It is not difficult to check that the computation for b) to a) type transition given by \eqref{Terapin2} are essentially the same and the estimates 
		for the set where $\{\tilde{f}_r=g\}$ given by \eqref{RozsekniTo} are even simpler. 
		
		We have the following elementary estimates (recall that $|D_{xx} f(0,y)| = |D_xQ_2(0,y) - D_xQ_1(0,y)|$)
		\begin{equation}\label{Comodo Dragon}
			|D_xQ_2(x,y) - D_x Q_1(x,y)| \leq |D_{xx}f(0,y)| + 2M|x| \leq |D_{xx}f(0,y)| + 2Mr,
		\end{equation}
		further, since $D_yQ_2(0,y) =D_y Q_1(0,y)$, we have
		\begin{equation}\label{African Spurred}
			|D_yQ_2(x,y) - D_y Q_1(x,y)| \leq 2Mr
		\end{equation}
		and
		\begin{equation}\label{Leviathan}
		\begin{aligned}
			|Q_2(x,y) - Q_1(x,y)| 
			&\leq |D_{xx}f(0,y)|\cdot|x|+\sum_{j=1}^2|Q_j(x,y)-Q_j(0,y)-x D_xQ_j(0,y)|\\
			&\leq |D_{xx}f(0,y)|r+\sum_{j=1}^2\Bigl|\int_0^x \bigl(D_x Q_j(s,y)-D_xQ_j(0,y)\bigr)\; ds\Bigr|\\
			&\leq |D_{xx}f(0,y)|r + Mr^2.\\
		\end{aligned}
		\end{equation}
		The second derivatives of \eqref{Terapin1} are calculated by
		$$
			\begin{aligned}
				D_{xx} g(x,y) =& (1-\eta\big(\tfrac{x}{r} + \eta(y\rho^{-1}-2i)\big))D_{xx}Q_1 +\eta\big(\tfrac{x}{r} + \eta(y\rho^{-1}-2i)\big)D_{xx}Q_2\\
				&\quad + \frac{1}{r^2}\eta''\big(\tfrac{x}{r} + \eta(y\rho^{-1}-2i)\big)(Q_2(x,y) - Q_1(x,y))\\
				&\quad + \frac{1}{r}\eta'\big(\tfrac{x}{r} + \eta(y\rho^{-1}-2i)\big)(D_xQ_2(x,y) - D_x Q_1(x,y)).
			\end{aligned}
		$$
		Using $|D^2Q_j(x,y)|\leq M$, $|\eta'|\leq 2$, $|\eta''|\leq 4$, \eqref{Comodo Dragon} and \eqref{Leviathan} we get
		$$
		|D_{xx} g(x,y)| \leq \frac{C}{r}|D_{xx}f(0,y)|+ CM.
		$$
%		$$
%		|D_{xx} g(x,y)| \leq \frac{6}{r}|D_{xx}f(0,y)|+ 9M.
%		$$
		%Specifically we calculated for an $a)$-type to $b)$-type transition, but the estimates are the same for the reverse. 
		Further
		$$
			\begin{aligned}	
				D_{xy}g(x,y) &= \big(1-\eta\big(\tfrac{x}{r} + \eta(y\rho^{-1}-2i)\big)\big)D_{xy}Q_1(x,y)
				+\eta\big(\tfrac{x}{r} + \eta(y\rho^{-1}-2i)\big)D_{xy}Q_2(x,y) \\
				&\quad + \frac{1}{r}\eta'\big(\tfrac{x}{r} + \eta(y\rho^{-1}-2i)\big)(D_y Q_2(x,y) - D_y Q_1(x,y))\\
				&\quad + \frac{1}{\rho}\eta'\big(\tfrac{x}{r} + \eta(y\rho^{-1}-2i)\big) \eta'(y\rho^{-1}-2i)(D_x Q_2(x,y) - D_x Q_1(x,y))\\
				&\quad + \frac{1}{r\rho}\eta''\big(\tfrac{x}{r} + \eta(y\rho^{-1}-2i)\big) \eta'(y\rho^{-1}-2i)(Q_2(x,y) - Q_1(x,y))
			\end{aligned}
		$$
		and using $|D^2Q_i(x,y)|\leq M$, $|\eta'|\leq 2$, $|\eta''|\leq 4$, \eqref{Comodo Dragon}, \eqref{African Spurred} and \eqref{Leviathan} we get
		$$
			\begin{aligned}	
				|D_{xy}g(x,y)| &\leq CM + \frac{C}{\rho}|D_{xx}f(0,y)|+\frac{CMr}{\rho}
			\end{aligned}
		$$
%		$$
%			\begin{aligned}	
%				|D_{xy}g(x,y)| &\leq 5M + \frac{12}{\rho}|D_{xx}f(0,y)|+\frac{16Mr}{\rho}
%			\end{aligned}
%		$$
		and the estimate holds for all $[x,y] \in [-r,r]\times[0,\ell]$ where \eqref{Terapin1} applies. Finally
		$$
			\begin{aligned}	
				D_{yy}g(x,y) &= \big[1-\eta\big(\tfrac{x}{r} + \eta(y\rho^{-1}-2i)\big)\big]D_{yy}Q_1(x,y) 
				+\eta\big(\tfrac{x}{r} + \eta(y\rho^{-1}-2i)\big)D_{yy}Q_2(x,y) \\
				&\quad + \tfrac{1}{\rho}\eta'\big(\tfrac{x}{r} + \eta(y\rho^{-1}-2i)\big) \eta'(y\rho^{-1}-2i)(D_y Q_2(x,y) - D_y Q_1(x,y))\\
				&\quad + \tfrac{1}{\rho^2}\eta''\big(\tfrac{x}{r} + \eta(y\rho^{-1}-2i)\big) \big[\eta'(y\rho^{-1}-2i)\big]^2(Q_2(x,y) - Q_1(x,y))\\
				&\quad + \tfrac{1}{\rho^2}\eta'\big(\tfrac{x}{r} + \eta(y\rho^{-1}-2i)\big) \eta''(y\rho^{-1}-2i)(Q_2(x,y) - Q_1(x,y))
			\end{aligned}
		$$
		so
		$$
			\begin{aligned}	
				|D_{yy}g(x,y)| &\leq M + \frac{CMr}{\rho} + \frac{Cr}{\rho^2}|D_{xx}f(0,y)| + \frac{CMr^2}{\rho^2}.
			\end{aligned}
		$$
%		$$
%			\begin{aligned}	
%				|D_{yy}g(x,y)| &\leq M + \frac{8Mr}{\rho} + \frac{24r}{\rho^2}|D_{xx}f(0,y)| + \frac{24Mr^2}{\rho^2}.
%			\end{aligned}
%		$$
		
		Integrating the above over $[-r,r]\times[0,\ell]$ and estimating 
		$$
		|D^2g(x,y)|\leq|D_{xx}g(x,y)| + 2|D_{xy}g(x,y)|+|D_{yy}g(x,y)|
		$$ 
		we get using $r\leq \frac{\rho}{40}$
			$$
			\begin{aligned}
				\int_{[-r,r]\times[0,\ell]}|D^2g(x,y)|&\leq 
				Cr\int_0^{\ell}|D_{xx}f(0,y)|\; dy\Bigl(\frac{1}{r}+\frac{1}{\rho}+\frac{r}{\rho^2}\Bigr) +
				Cr\ell\Bigl[M+\frac{Mr}{\rho}+\frac{M r^2}{\rho^2}\Bigr]\\
				&\leq C\int_0^{\ell}|D_{xx}f(0,y)| \ dy + CM\ell r,\
			\end{aligned}
		$$
		%$$
	%		\begin{aligned}
	%			\int_{[-r,r]\times[0,\ell]}|D^2g(x,y)|&\leq 
	%			r\int_0^{\ell}|D_{xx}f(0,y)|\; dy\Bigl(\frac{6}{r}+\frac{24}{\rho}+\frac{24r}{\rho^2}\Bigr) +
	%			r\ell\Bigl[20 M+\frac{40 Mr}{\rho}+\frac{24 M r^2}{\rho^2}\Bigr]\\
	%			&\leq 8\int_0^{\ell}|D_{xx}f(0,y)| \ dy + 22M\ell r,\
%			\end{aligned}
		%$$
		and \eqref{Iguana} follows.
	\end{proof}
	
%We have discussed
%$$
%\eta(y)Q^1+(1-\eta(y))Q^2\text{ along the edges}.
%$$
%How to prove in detail that it is invertible? 

%The most difficult part is to connect the "polar" part with the part around the edges. After one beer :) we have %suggested
%$$
%\eta\bigl(\tau(x)\phi r+(1-\tau(x))y\bigr)Q^1+\bigl[1-\eta\bigl(\tau(x)\phi r+(1-\tau(x))y\bigr)\bigr]Q^2. 
%$$
%We just have to adjust so that on the boundary of the circle we have 
%$$
%[\tR,\tPhi]=\eta(\phi r)Q^1+(1-\eta(\phi r))Q^2. 
%$$

\section{Approximation of piecewise quadratic homeomorphisms around the vertices and proof of Theorem \ref{Dan}}

	Again $\eta$ denotes the function from the Preliminaries, Notation~\ref{Notation}. 

%\ignore{	
	\definecolor{xfqqff}{rgb}{0.4980392156862745,0.,1.}
\definecolor{ffqqqq}{rgb}{1.,0.,0.}
\definecolor{qqqqff}{rgb}{0.,0.,1.}

	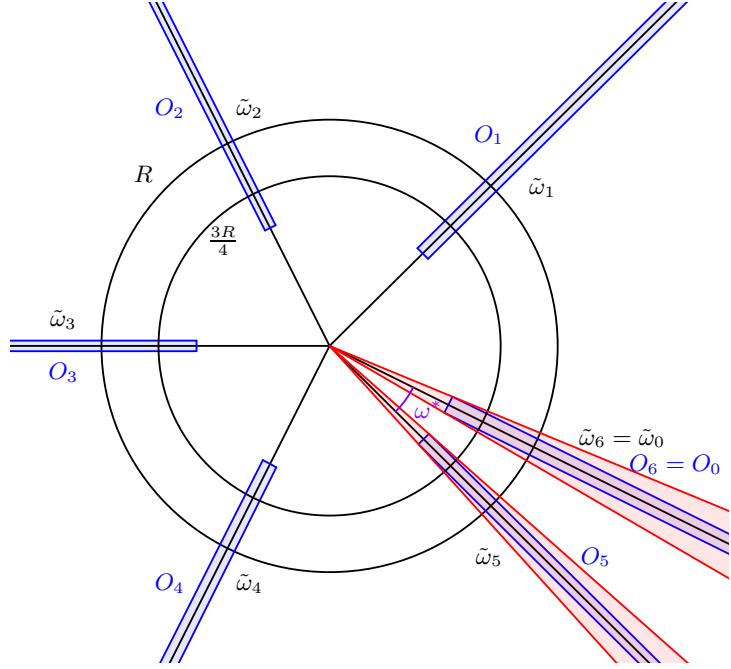
\begin{figure}
		\centering
		\begin{tikzpicture}[line cap=round,line join=round,>=triangle 45,x=3.5cm,y=3.5cm]
		\clip(-1.2,-1.2) rectangle (1.5,1.3);
		\fill[line width=0.7pt,color=qqqqff,fill=qqqqff,fill opacity=0.10000000149011612] (0.33,0.37) -- (0.37,0.33) -- (1.52,1.48) -- (1.48,1.52) -- cycle;
		\fill[line width=0.7pt,color=qqqqff,fill=qqqqff,fill opacity=0.10000000149011612] (-0.20273988595078735,0.45708464026719614) -- (-0.2425277711338731,0.4371906976756542) -- (-0.8414008026636257,1.6304202201131621) -- (-0.8,1.65) -- cycle;
		\fill[line width=0.7pt,color=qqqqff,fill=qqqqff,fill opacity=0.10000000149011612] (-0.5,0.02) -- (-0.5,-0.02) -- (-1.95,-0.02) -- (-1.95,0.02) -- cycle;
		\fill[line width=0.7pt,color=qqqqff,fill=qqqqff,fill opacity=0.10000000149011612] (-0.2,-0.46) -- (-0.2506344538339652,-0.4325150274210417) -- (-0.9501881920278096,-1.825038151765126) -- (-0.8899183415975359,-1.8553359684679116) -- cycle;
		\fill[line width=0.7pt,color=qqqqff,fill=qqqqff,fill opacity=0.10000000149011612] (0.3341412274234713,-0.3722460277662893) -- (0.37331874475444937,-0.333068510435313) -- (1.38,-1.34) -- (1.34,-1.38) -- cycle;
		\fill[line width=0.7pt,color=qqqqff,fill=qqqqff,fill opacity=0.10000000149011612] (0.43105403345273285,-0.252995163609721) -- (0.46198365239824185,-0.1925105754496172) -- (1.9367216556960503,-0.9299595827874171) -- (1.9050699004675513,-0.9904179916508398) -- cycle;
		\draw [line width=0.7pt] (0.,0.)-- (2.,-1.);
		\draw [line width=0.7pt] (0.,0.)-- (1.5,1.5);
		\draw [line width=0.7pt] (0.,0.)-- (-1.,2.);
		\draw [line width=0.7pt] (0.,0.)-- (-2.,0.);
		\draw [line width=0.7pt] (0.,0.)-- (-1.,-2.);
		\draw [line width=0.7pt] (0.,0.)-- (1.5,-1.5);
		\draw [line width=0.7pt] (0.,0.) circle (3.cm);
		\draw [line width=0.7pt,color=qqqqff] (0.33,0.37)-- (0.37,0.33);
		\draw [line width=0.7pt,color=qqqqff] (0.37,0.33)-- (1.52,1.48);
		\draw [line width=0.7pt,color=qqqqff] (1.52,1.48)-- (1.48,1.52);
		\draw [line width=0.7pt,color=qqqqff] (1.48,1.52)-- (0.33,0.37);
		\draw [line width=0.7pt,color=qqqqff] (-0.20273988595078735,0.45708464026719614)-- (-0.2425277711338731,0.4371906976756542);
		\draw [line width=0.7pt,color=qqqqff] (-0.2425277711338731,0.4371906976756542)-- (-0.8414008026636257,1.6304202201131621);
		\draw [line width=0.7pt,color=qqqqff] (-0.8414008026636257,1.6304202201131621)-- (-0.8,1.65);
		\draw [line width=0.7pt,color=qqqqff] (-0.8,1.65)-- (-0.20273988595078735,0.45708464026719614);
		\draw [line width=0.7pt,color=qqqqff] (-0.5,0.02)-- (-0.5,-0.02);
		\draw [line width=0.7pt,color=qqqqff] (-0.5,-0.02)-- (-1.95,-0.02);
		\draw [line width=0.7pt,color=qqqqff] (-1.95,-0.02)-- (-1.95,0.02);
		\draw [line width=0.7pt,color=qqqqff] (-1.95,0.02)-- (-0.5,0.02);
		\draw [line width=0.7pt,color=qqqqff] (-0.2,-0.46)-- (-0.2506344538339652,-0.4325150274210417);
		\draw [line width=0.7pt,color=qqqqff] (-0.2506344538339652,-0.4325150274210417)-- (-0.9501881920278096,-1.825038151765126);
		\draw [line width=0.7pt,color=qqqqff] (-0.9501881920278096,-1.825038151765126)-- (-0.8899183415975359,-1.8553359684679116);
		\draw [line width=0.7pt,color=qqqqff] (-0.8899183415975359,-1.8553359684679116)-- (-0.2,-0.46);
		\draw [line width=0.7pt,color=qqqqff] (0.3341412274234713,-0.3722460277662893)-- (0.37331874475444937,-0.333068510435313);
		\draw [line width=0.7pt,color=qqqqff] (0.37331874475444937,-0.333068510435313)-- (1.38,-1.34);
		\draw [line width=0.7pt,color=qqqqff] (1.38,-1.34)-- (1.34,-1.38);
		\draw [line width=0.7pt,color=qqqqff] (1.34,-1.38)-- (0.3341412274234713,-0.3722460277662893);
		\draw [line width=0.7pt,color=qqqqff] (0.43105403345273285,-0.252995163609721)-- (0.46198365239824185,-0.1925105754496172);
		\draw [line width=0.7pt,color=qqqqff] (0.46198365239824185,-0.1925105754496172)-- (1.9367216556960503,-0.9299595827874171);
		\draw [line width=0.7pt,color=qqqqff] (1.9367216556960503,-0.9299595827874171)-- (1.9050699004675513,-0.9904179916508398);
		\draw [line width=0.7pt,color=qqqqff] (1.9050699004675513,-0.9904179916508398)-- (0.43105403345273285,-0.252995163609721);
		\draw [line width=0.7pt] (0.,0.) circle (2.25cm);
		\draw [shift={(0.,0.)},line width=0.7pt,color=ffqqqq,fill=ffqqqq,fill opacity=0.10000000149011612]  (0,0) --  plot[domain=5.443896159132604:5.554706004709364,variable=\t]({1.*2.145685634274954*cos(\t r)+0.*2.145685634274954*sin(\t r)},{0.*2.145685634274954*cos(\t r)+1.*2.145685634274954*sin(\t r)}) -- cycle ;
		\draw [shift={(0.,0.)},line width=0.7pt,color=ffqqqq,fill=ffqqqq,fill opacity=0.10000000149011612]  (0,0) --  plot[domain=5.752437368628702:5.888362119855261,variable=\t]({1.*2.14395349587217*cos(\t r)+0.*2.14395349587217*sin(\t r)},{0.*2.14395349587217*cos(\t r)+1.*2.14395349587217*sin(\t r)}) -- cycle ;
		\draw [shift={(0.,0.)},line width=0.7pt,color=xfqqff]  plot[domain=5.497787143782138:5.81953769817878,variable=\t]({1.*0.35119700943133625*cos(\t r)+0.*0.35119700943133625*sin(\t r)},{0.*0.35119700943133625*cos(\t r)+1.*0.35119700943133625*sin(\t r)});
		\begin{scriptsize}
		\draw[color=black] (1.1,-0.35) node {$\tilde{\omega}_6 = \tilde{\omega}_0$};
		\draw[color=black] (0.8,0.6) node {$\tilde{\omega}_1$};
		\draw[color=black] (-0.3,0.9) node {$\tilde{\omega}_2$};
		\draw[color=black] (-1,0.1) node {$\tilde{\omega}_3$};
		\draw[color=black] (-0.3,-0.9) node {$\tilde{\omega}_4$};
		\draw[color=black] (0.6,-0.8) node {$\tilde{\omega}_5$};
		\draw[color=qqqqff] (1.3,-0.45) node {$O_6 = O_0$};
		\draw[color=qqqqff] (0.6,0.8) node {$O_1$};
		\draw[color=qqqqff] (-0.6,0.9) node {$O_2$};
		\draw[color=qqqqff] (-1,-0.1) node {$O_3$};
		\draw[color=qqqqff] (-0.6,-0.9) node {$O_4$};
		\draw[color=qqqqff] (1,-0.8) node {$O_5$};
		\draw[color=black] (-0.7,0.65) node {$R$};
		\draw[color=black] (-0.4,0.4) node {$\frac{3R}{4}$};
		\draw[color=xfqqff] (0.3709687064032512,-0.24086448830293206) node {$\omega^*$};
		\end{scriptsize}
		\end{tikzpicture}
		\caption{The sets $O_i$ in blue contained inside red cones around rays parallel to $\tilde{\omega}_i$. Outside $B(0,R)$ we use the same approach as Lemma~\ref{Reptiles}. Inside $B(0,3R/4)$ we use a linear map. In the annulus we interpolate by first squashing onto rings and then rotating.}\label{Fig:Setup}
	\end{figure}
	%}

	\begin{lemma}[Approximation near vertices]\label{America}
		Let $Q_1, Q_2,\dots Q_N:\er^2\to\er^2$ be quadratic mappings. Let $0\leq \omega_0< \omega_1< \dots< \omega_{N-1} < \omega_N = \omega_0+ 2\pi < 4\pi$ and let $\tilde{\omega}_i = [\cos\omega_i, \sin\omega_i]\in \S$ be angles ordered anti-clockwise around $\S$ and call 
		$$
		\omega^* = \min\{ \tfrac{\pi}{8}, \omega_{i+1} - \omega_i; i=0,\dots N  \}.
		$$
		Call $\tilde{\omega}_i^{\bot} =  [-\sin\omega_i, \cos\omega_i] \in \S$ the vector anti-clockwise perpendicular to $\tilde{\omega}_i$. Let $f:B(0,\rho_0)\to\er^2$ be the map defined by 
		$$
		f(t\cos\theta, t\sin\theta) = Q_i(t\cos\theta, t\sin\theta)
		\text{ for all }0\leq t\leq \rho_0\text{ and all }\omega_{i-1} \leq \theta \leq \omega_{i}. 
		$$
		Further assume that this $f$ is a homeomorphism and $\det DQ_i \geq d >0$ on $B(0,\rho_0)$. Let $L$ and $M$ denote positive numbers such that $|DQ_i|\leq L$ on $B(0,\rho_0)$ and $|D^2Q_i|\leq M$. For every $\rho_1, \rho_2 \dots, \rho_N$ and every $R$ such that
		\eqn{defR}
		$$
			0<R<\tfrac{1}{2}\min\{\rho_i, i=1,\dots, N\}<
			\tfrac{1}{2}\min\Bigl\{\rho_0, \frac{\min\{d,d^2\}}{1000 (M+1)(L+1)^4},
		\frac{1}{320}\frac{d^2}{ML^3},\frac{1}{8}\frac{L}{M+1}\Bigr\}
		$$ 
		and every 
		$$
		0<r_i\leq \min\Bigl\{ \frac{d^2R}{432L^4}, \frac{Rd}{1200 L^2}, \frac{\rho_i^2}{2(L+1)}, \frac{R}{2}\tan\frac{\omega^*}{3}\Bigr\}
		\text{ we call }
		\r = (R, \rho_1,\dots , \rho_N, r_1, \dots, r_N).
		$$
		Then for all such $\r$, the rectangles (see Fig. \ref{Fig:Setup})
		$$
		O_i = \bigl\{t\tilde{\omega}_i+ s\tilde{\omega}_i^{\bot}; t\in [\tfrac{R}{2}, \rho_i], s\in [-r_i,r_i] \bigr\}
		$$
		are pairwise disjoint. Further call $\vec{v}_i = \frac{D_{\tilde{\omega}_i}Q_i(\rho_i\tilde{\omega}_i)}{|D_{\tilde{\omega}_i}Q_i(\rho_i\tilde{\omega}_i)|}$ and call $\vec{u}_i \in \S$ the vector clockwise perpendicular to $\vec{v}_i$. Define $\tilde{f}_{\r}$ as
		%$$
		%	\tilde{f}_{\r}(x,y)=f(x,y)
		%$$
		\eqn{deffr}
		$$
		\tilde{f}_{\r}(x,y)=
		\begin{cases}
			f(x,y)\text{ for }[x,y]\notin \bigcup_{i=1}^N O_i,\\
			\bigl[1-\eta\big(\tfrac{1}{r_i}\langle [x,y],-\tilde{\omega}_i^{\bot}\rangle\big)\bigr]Q_i(x,y) 
			+\eta\bigl(\tfrac{1}{r_i}\langle[x,y],-\tilde{\omega}_i^{\bot}\rangle\bigr)Q_{i+1}(x,y)\\
			\phantom{f(x,y)}\text{ for }[x,y]\in O_i\text{ if }\langle D_{-\tilde{\omega}_i^{\bot}} Q_{i+1}(\rho_i\tilde{\omega}_i),\vec{u}_i\rangle \geq \langle D_{-\tilde{\omega}_i^{\bot}} Q_{i}(\rho_i \tilde{\omega}_i),\vec{u}_i\rangle,\\
			\bigl[1-\eta\big(\tfrac{1}{r_i}\langle [x,y],-\tilde{\omega}_i^{\bot}\rangle +1 \bigr)\big]Q_i(x,y) 
			+ \eta\bigl(\tfrac{1}{r_i}\langle [x,y],-\tilde{\omega}_i^{\bot}\rangle+1\bigr)Q_{i+1}(x,y)\\
			\phantom{f(x,y)}\text{ for }[x,y]\in O_i\text{ if }\langle D_{-\tilde{\omega}_i^{\bot}} Q_{i+1}(\rho_i\tilde{\omega}_i),\vec{u}_i\rangle < \langle D_{-\tilde{\omega}_i^{\bot}} Q_{i}(\rho_i \tilde{\omega}_i),\vec{u}_i\rangle.\\
		\end{cases}
		$$
		%for $[x,y]\notin \bigcup_{i=1}^N O_i$, further
		%$$
		%	\tilde{f}_{\r}(x,y)  = \bigl[1-\eta\big(\tfrac{1}{r}\langle(x,y),\tilde{\omega}_i^{\bot}\rangle\big)\bigr]Q_i(x,y) 
		%	+\eta\bigl(\tfrac{1}{r}\langle(x,y),\tilde{\omega}_i^{\bot}\rangle\bigr)Q_{i+1}(x,y)
		%$$
		%for $[x,y]\in O_i$ if $\langle D_{\tilde{\omega}_i^{\bot}} Q_{i+1}(\rho_i\tilde{\omega}_i),\vec{u}_i\rangle \geq \langle D_{\tilde{\omega}_i^{\bot}} Q_{i}(\rho_i \tilde{\omega}_i),\vec{u}_i\rangle$ and finally
	%	$$
		%	\tilde{f}_{\r}(x,y)  = \bigl[1-\eta\big(\tfrac{1}{r}\langle(x,y),\tilde{\omega}_i^{\bot}\rangle +1 \bigr)\big]Q_i(x,y) 
		%	+ \eta\bigl(\tfrac{1}{r}\langle(x,y),\tilde{\omega}_i^{\bot}\rangle+1\bigr)Q_{i+1}(x,y)
		%$$
		%for $[x,y]\in O_i$ if $\langle D_{\tilde{\omega}_i^{\bot}} Q_{i+1}(\rho_i\tilde{\omega}_i),\vec{u}_i\rangle < \langle D_{\tilde{\omega}_i^{\bot}} Q_{i}(\rho_i \tilde{\omega}_i),\vec{u}_i\rangle$.
		
		Then there exists a $\mathcal{C}^{\infty}$ diffeomorphism $g_{\r}$ defined on $B(0,2R)$ with $g_{\r}(x,y) = \tilde{f}_{\r}(x,y)$ for all $R \leq |[x,y]| \leq 2R$ and
		\begin{equation}\label{Houston}
			\int_{B(0,R)}|D^2g_{\r}| < CR
		\end{equation}
		where the constant $C$ depends on $d$, $L$, $M$ and $N$ %the behaviour of $f$ close to the vertex 
		but is independent of $R$.
	\end{lemma}
	\begin{proof}
		Without loss of generality we may assume that $f(0,0) = [0,0]$.
		
		\step{1}{Proving that $O_i$ are pair-wise disjoint}{EastCoast}
		
		The first claim we prove is that $O_i\cap O_j = \emptyset$ for any $1\leq i < j \leq N$. On the one hand we have that $r_i \leq \tfrac{R}{2}\tan\tfrac{\omega^*}{3}$ and on the other hand we have that 
		$$
		\min\{|[x,y]|; [x,y]\in O_i\} = \tfrac{1}{2}R. 
		$$
		Therefore $O_i$ lies inside a cone whose axis goes through $\omega_i$ and the angle at the apex is $\tfrac{2}{3}\omega^*$. These cones are pairwise disjoint and therefore so are $O_i$ (see Fig. \ref{Fig:Setup}). 
		
		From $r_i \leq \tfrac{R}{2}$ we get $\sqrt{\tfrac{R^2}{4} + \tfrac{R^2}{4}} < \tfrac{3R}{4}$ and hence 
		\begin{equation}
			\big\{\tfrac{R}{2}\tilde{\omega}_i + s\tilde{\omega}_i^{\bot}; s\in [-r_i,r_i] \big\} \subset B(0,\tfrac{3R}{4}).
		\end{equation}
		It follows that this inner edge of $O_i$ (where $\tilde{f}_{\r}$ is discontinuous) is a subset of $B(0,\tfrac{3R}{4})$ and thus we can use
		 Lemma~\ref{Reptiles} to conclude that $\tilde{f}_{\r}$ is a diffeomorphism on $B(0,2R) \setminus B(0,\tfrac{3}{4}R)$ since \eqref{deffr} agrees with rotated and translated version of \eqref{RozsekniTo} there (our $r_i$ and $\rho_i$ play the role of $r$ and $\rho$ in Lemma~\ref{Reptiles}). Note that $d$, $L$ and $M$ play the same role as in Lemma~\ref{Reptiles} and that \eqref{defR} verifies \eqref{defrho}. In the following computation we will use some estimates from Lemma~\ref{Reptiles}.

		We have shown that $\tilde{f}_{\r}$ is smooth for $\tfrac{3R}{4} \leq |[x,y]|\leq 2R <\min\{\rho_i; i=1,2,\dots, N\}$.

		\step{2}{Proving $\langle \frac{\partial}{\partial\theta}\phi_{f}(t\cos\theta, t\sin\theta)  , \phi^{\bot}_{f}(t\cos\theta, t\sin\theta)  \rangle \geq C>0$}{MidWest}
		
		Now we express $\tilde{f}_{\r}$ in polar coordinates in the image, i.e. we define the pair of functions $\R_f:B(0,2R)\to[0, \infty)$ as $\R_f(x,y) = |\tilde{f}_{\r}(x,y) |$ and $\phi_f:B(0,2R)\setminus\{[0,0]\}\to\S \subset \er^2$ as $\phi_f(x,y) = \tfrac{\tilde{f}_{\r}(x,y)}{|\tilde{f}_{\r}(x,y)|}$. Then
		$$
			\tilde{f}_{\r}(x,y) =  \R_f(x, y)\phi_f(x,y)\text{ on }B(0,2R). 
		$$
		 Since $\tilde{f}_{\r}$ is $\mathcal{C}^{\infty}$ smooth on $B(0,2R)\setminus B(0,\tfrac{3}{4}R)$ and $|\tilde{f}_{\r}(x,y)| = 0$ if and only if $[x,y] = [0,0]$, we have that $\R_f$ and $\phi_f$ are $\mathcal{C}^{\infty}$ smooth there. Further we define 
		 $$
		 \phi_f^{\bot}(x,y) = \bigl[-(\phi_f(x,y))_2,(\phi_f(x,y))_1\bigr]
		 $$ 
		 the $\tfrac{\pi}{2}$ anti-clockwise rotation of $\phi_f$. For brevity call $\tilde{\theta} = [\cos\theta, \sin\theta]$ and $\tilde{\theta}^{\bot} = [-\sin\theta, \cos\theta]$. Our aim is to prove that in $B(0,R)$
		 $$
		 \bigl\langle  D_{\tilde{\theta}^{\bot}}\phi_{f}(t\tilde{\theta})  , \phi^{\bot}_{f}(t\tilde{\theta})  \bigr\rangle
		=\bigl\langle  \frac{\partial}{\partial\theta}\phi_{f}(t\tilde{\theta})  , \phi^{\bot}_{f}(t\tilde{\theta})  \bigr\rangle
		\geq C>0
		%\text{ and } \frac{\partial}{\partial \theta} \phi_{f}(t\cos\theta,t\sin\theta) \geq C>0
		.
		 $$

		\step{2.A}{The $[x,y]\notin O_i$ case}{Illinois}
		
		By Lemma~\ref{GirlsNames} the map 
		$$
		h(t\cos\theta, t\sin\theta) = DQ_i(0,0)(t\cos\theta, t\sin\theta)\text{ for }
		t\in [0,\infty)\text{ and }
		\omega_{i-1}\leq \theta \leq \omega_i
		$$ 
		is a piecewise linear homeomorphism. 
		From Lemma~\ref{Partials} we have
		\begin{equation}\label{Denver}
				\langle \frac{\partial}{\partial\theta}\phi_f(t\cos\theta, t\sin\theta) , \phi_f^{\bot}(t\tilde{\theta}) \rangle = \frac{t}{\R_f(t\tilde{\theta})}\langle D_{\tilde{\theta}^{\bot}}\tilde{f}_{\r}(t\tilde{\theta}), \phi_f^{\bot}(t\tilde{\theta}) \rangle.
		\end{equation}
		We call 
		$$
		\phi_h(t\tilde{\theta}) = \tfrac{h(t\tilde{\theta})}{|h(t\tilde{\theta})|}
		\text{ and }
		\phi_h^{\bot}(x,y) = \bigl[-(\phi_h(x,y))_2,(\phi_h(x,y))_1\bigr] 
		$$ 
		the $\tfrac{\pi}{2}$ anti-clockwise rotation of $\phi_h$. For brevity we use the notation $[x,y] = t\tilde{\theta}$, where $t = |[x,y]|$ and $\tilde{\theta} = \tfrac{[x,y]}{|[x,y]|}$. 
	  By linearity $\phi_h$ depends only on $\theta$ and not $t$ and hence $D_{\tilde{\theta}}(\phi_h(x,y))=0$ which implies 
		$$
		D_{\tilde{\theta}}h(x,y)%=D_{\tilde{\theta}}\bigl(|h(x,y)|\phi_h(x,y)\bigr)
		=D_{\tilde{\theta}}\bigl(|h(x,y)|\bigr)\phi_h(x,y)
		+|h(x,y)|D_{\tilde{\theta}}\bigl(\phi_h(x,y)\bigr)=D_{\tilde{\theta}}\bigl(|h(x,y)|\bigr)\phi_h(x,y).
		$$
		It follows that 
		$$
				0< \bigl\langle D_{\tilde{\theta}}h(x,y), \phi_h(x,y)\bigl\rangle \leq L \quad \text{and} \quad 
				\bigl\langle D_{\tilde{\theta}}h(x,y), \phi_h^{\bot}(x,y)\bigr\rangle = 0.
		$$
		Using \eqref{JackOfAllTrades} we obtain 
		$$
				d\leq J_h(x,y)= \bigl\langle D_{\tilde{\theta}}h(x,y), \phi_h(x,y)\bigr\rangle 
				\bigl\langle D_{\tilde{\theta}^{\bot}} h(x,y), \phi_h^{\bot}(x,y) \bigr\rangle
		$$
		and together with $| D_{w} h(x,y)| \leq L$ for all $[x,y]\in B(0,R)$ and all $w\in \S$ this implies 
		\begin{equation}\label{Dallas}
				\frac{d}{L} \leq \langle D_{\tilde{\theta}}h(x,y), \phi_h(x,y)\rangle \leq L \quad  \text{and} \quad
				\frac{d}{L}\leq \langle D_{\tilde{\theta}^{\bot}} h(x,y), \phi_h^{\bot}(x,y) \rangle \leq L.
		\end{equation}
		Therefore $\tfrac{d}{L}|[x,y]|\leq |h(x,y)| \leq L|[x,y]|$ and $|[x,y]|<R<\frac{1}{8}\min\{L,\frac{d}{L}\}$ gives (see \eqref{Becca}) that 
		\eqn{Kansas City}
		$$
		\frac{15d}{L16}|[x,y]| \leq |f(x,y)| \leq \frac{17L}{16}|[x,y]|,\text{ i.e. }
			\frac{16}{17L} \leq \frac{|[x,y]|}{\R_f(x,y)} \leq \frac{16L}{15d}.
		$$
		%Therefore in $B(0,R) \setminus \bigcup_{i=1}^{N}O_i$ we have the following, $\tilde{f}_{\r} = f$ and so $\R_f \leq \tfrac{17}{16}Lt$ and $|D\tilde{f}_{\r} - Dh|< Mt$. 
		Further for all $|[x,y]| = t\leq R \leq \tfrac{\min\{d,d^2\}}{1000 (M+1)(L+1)^3}$ we have using \eqref{elementary} and \eqref{Becca}  
		\eqn{rozdiluhlu}
		$$
		\begin{aligned}
		\bigl|\phi_h(t\tilde{\theta}) -\phi_f(t\tilde{\theta})\bigr|	&=\Bigl|\frac{h(t\tilde{\theta})}{|h(t\tilde{\theta})|}-\frac{f(t\tilde{\theta})}{|f(t\tilde{\theta})|}\Bigr|\leq \frac{|f(t\tilde{\theta})-h(t\tilde{\theta})|}{|h(t\tilde{\theta})|}2\\
			&
			\leq \frac{\tfrac{1}{2}Mt^2}{\frac{d}{L}t  }2  < \frac{1}{1000}\min\bigl\{1,\frac{d}{L^2}\bigr\}.\\
		\end{aligned}
		$$
		Therefore, using \eqref{Dallas}, we get
		$$
			\frac{99d}{100L}\leq \bigl\langle D_{\tilde{\theta}^{\bot}} h(x,y), \phi_f^{\bot}(x,y) \bigr\rangle 
			\leq \frac{101}{100}L.
		$$
		In this case we estimate for all $0<|[x,y]| = t\leq R \leq \tfrac{d^2}{1000 (M+1)(L+1)^3}$ using \eqref{Denver}, \eqref{Kansas City} and \eqref{Ella} to get
		\begin{equation}\label{Nashville}
			\begin{aligned}
				\bigl\langle \frac{\partial}{\partial\theta}\phi_f(t\cos\theta, t\sin\theta) , \phi_f^{\bot}(t\tilde{\theta}) \bigr\rangle 
				= &\frac{t}{\R_f(t\tilde{\theta})}\bigl\langle D_{\tilde{\theta}^{\bot}}\tilde{f}_{\r}(t\tilde{\theta}), \phi_f^{\bot}(t\tilde{\theta}) \bigr\rangle \\
				\geq& \frac{16}{17L}\langle D_{\tilde{\theta}^{\bot}} h(t\tilde{\theta}), \phi_f^{\bot}(t\tilde{\theta}) \rangle  -\frac{16L}{15d}\bigl|\langle D_{\tilde{\theta}^{\bot}} h(t\tilde{\theta}) - D_{\tilde{\theta}^{\bot}}\tilde{f}_{\r}(t\tilde{\theta}), \phi_f^{\bot}(t\tilde{\theta}) \rangle\bigr|\\
				\geq & \frac{16}{17L}\langle D_{\tilde{\theta}^{\bot}} h(t\tilde{\theta}), \phi_f^{\bot}(t\tilde{\theta}) \rangle - \frac{16L}{15d}Mt\\
				\geq &\frac{16}{17L}\frac{99d}{100L} - \frac{d}{800L^2}\\
				\geq & \frac{9d}{10L^2}.
			\end{aligned}
		\end{equation}
		
		\step{2.B}{The $[x,y]\in O_i$ case}{Idaho}

		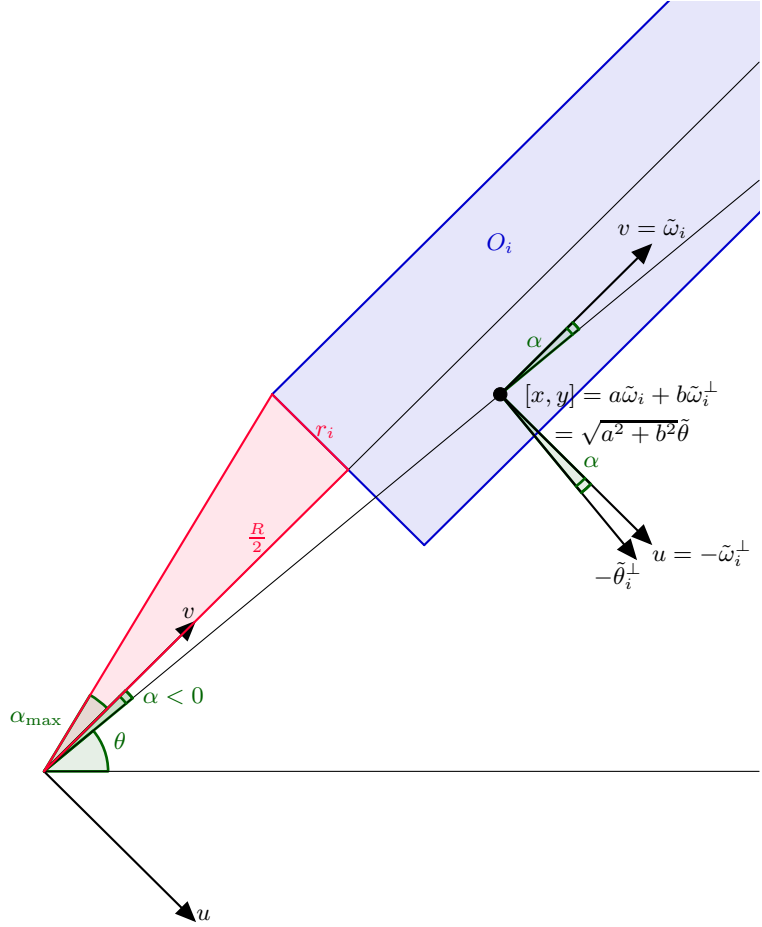
\begin{figure}
			\centering
		\definecolor{qqqqcc}{rgb}{0.,0.,0.8}
		\definecolor{qqwuqq}{rgb}{0.,0.39215686274509803,0.}
		\definecolor{ududff}{rgb}{0.30196078431372547,0.30196078431372547,1.}
		\definecolor{ffqqtt}{rgb}{1.,0.,0.2}
		\begin{tikzpicture}[line cap=round,line join=round,>=triangle 45,x=1cm,y=1cm]
		\clip(-1.1103711670922258,-2.9619179222925074) rectangle (9.404510668168342,10.215385916126914);
		\fill[line width=2.pt,color=qqqqcc,fill=qqqqcc,fill opacity=0.10000000149011612] (3.,5.) -- (9.,11.) -- (11.,9.) -- (5.,3.) -- cycle;
		\draw [shift={(0.,0.)},line width=1.pt,color=qqwuqq,fill=qqwuqq,fill opacity=0.10000000149011612] (0,0) -- (0.:0.8425386085945967) arc (0.:39.8055710922652:0.8425386085945967) -- cycle;
		\draw [shift={(0.,0.)},line width=1.pt,color=qqwuqq,fill=qqwuqq,fill opacity=0.20000000298023224] (0,0) -- (39.8055710922652:1.516569495470274) arc (39.8055710922652:45.:1.516569495470274) -- cycle;
		\draw [shift={(6.,5.)},line width=1.pt,color=qqwuqq,fill=qqwuqq,fill opacity=0.10000000149011612] (0,0) -- (-50.71059313749966:1.6850772171891935) arc (-50.71059313749966:-45.:1.6850772171891935) -- cycle;
		\fill[line width=2.pt,color=ffqqtt,fill=ffqqtt,fill opacity=0.10000000149011612] (0.,0.) -- (4.,4.) -- (3.,5.) -- cycle;
		\draw [shift={(0.,0.)},line width=1.pt,color=qqwuqq,fill=qqwuqq,fill opacity=0.10000000149011612] (0,0) -- (45.:1.1795540520324355) arc (45.:59.03624346792648:1.1795540520324355) -- cycle;
		\draw [shift={(6.,5.)},line width=1.pt,color=qqwuqq,fill=qqwuqq,fill opacity=0.10000000149011612] (0,0) -- (39.80557109226519:1.3480617737513547) arc (39.80557109226519:45.:1.3480617737513547) -- cycle;
		\draw [line width=0.8pt,color=qqqqcc] (3.,5.)-- (9.,11.);
		\draw [line width=0.8pt,color=qqqqcc] (9.,11.)-- (11.,9.);
		\draw [line width=0.8pt,color=qqqqcc] (11.,9.)-- (5.,3.);
		\draw [line width=0.8pt,color=qqqqcc] (5.,3.)-- (3.,5.);
		\draw [->,line width=0.8pt] (0.,0.) -- (2.,-2.);
		\draw [->,line width=0.8pt] (0.,0.) -- (2.,2.);
		\draw [line width=0.4pt,domain=0.0:9.404510668168342] plot(\x,{(-0.-0.*\x)/7.});
		\draw [line width=0.4pt,domain=0.0:9.404510668168342] plot(\x,{(-0.--5.*\x)/6.});
		\draw [line width=0.4pt,domain=0.0:9.404510668168342] plot(\x,{(-0.--2.*\x)/2.});
		\draw [shift={(0.,0.)},line width=1.pt,color=qqwuqq] (39.8055710922652:1.516569495470274) arc (39.8055710922652:45.:1.516569495470274);
		\draw [shift={(0.,0.)},line width=1.pt,color=qqwuqq] (39.8055710922652:1.4070394763529763) arc (39.8055710922652:45.:1.4070394763529763);
		\draw [->,line width=0.8pt] (6.,5.) -- (8.,3.);
		\draw [->,line width=0.8pt] (6.,5.) -- (8.,7.);
		\draw [->,line width=0.8pt] (6.,5.) -- (7.8,2.8);
		\draw [shift={(6.,5.)},line width=1.pt,color=qqwuqq] (-50.71059313749966:1.6850772171891935) arc (-50.71059313749966:-45.:1.6850772171891935);
		\draw [shift={(6.,5.)},line width=1.pt,color=qqwuqq] (-50.71059313749966:1.5755471980718958) arc (-50.71059313749966:-45.:1.5755471980718958);
		\draw [line width=0.8pt,color=ffqqtt] (0.,0.)-- (4.,4.);
		\draw [line width=0.8pt,color=ffqqtt] (4.,4.)-- (3.,5.);
		\draw [line width=0.8pt,color=ffqqtt] (3.,5.)-- (0.,0.);
		\draw [shift={(6.,5.)},line width=1.pt,color=qqwuqq] (39.80557109226519:1.3480617737513547) arc (39.80557109226519:45.:1.3480617737513547);
		\draw [shift={(6.,5.)},line width=1.pt,color=qqwuqq] (39.80557109226519:1.238531754634057) arc (39.80557109226519:45.:1.238531754634057);
		\begin{scriptsize}
		\draw[color=black] (2.1,-1.9) node {$u$};
		\draw[color=black] (1.9,2.1) node {$v$};
		\draw [fill=black] (6.,5.) circle (2.5pt);
		\draw[color=black] (7.6,5) node {$[x,y] = a\tilde{\omega}_i + b\tilde{\omega}^{\bot}_i$};
		\draw[color=black] (7.6,4.5) node {$ = \sqrt{a^2+b^2}\tilde{\theta}$};
		\draw[color=qqwuqq] (1,0.4) node {$\theta$};
		\draw[color=qqwuqq] (1.7,1) node {$\alpha<0$};
		\draw[color=black] (8.65,2.9) node {$u = -\tilde{\omega}^{\bot}_i$};
		\draw[color=black] (8,7.2) node {$v = \tilde{\omega}_i$};
		\draw[color=black] (7.55,2.6) node {$-\tilde{\theta}^{\bot}_i$};
		\draw[color=qqwuqq] (7.2,4.1) node {$\alpha$};
		\draw[color=ffqqtt] (2.8,3.1) node {$\frac{R}{2}$};
		\draw[color=ffqqtt] (3.7,4.5) node {$r_i$};
		\draw[color=qqwuqq] (-0.1,0.7) node {$\alpha_{\text{max}}$};
		\draw[color=qqwuqq] (6.45,5.7) node {$\alpha$};
		\draw[color=qqqqcc] (6,7) node {$O_i$};
		\end{scriptsize}
		\end{tikzpicture}
		\caption{Position of vectors and points in $O_i$.}\label{Fig:Prime}
	\end{figure}

		In the case $[x,y]\in O_i$ we calculate as follows. 
		Let $u, v\in \S$ satisfy $u\bot v$ and set $w=u\cos\alpha +v\sin\alpha $ for some $\alpha \in [-\pi/2,\pi/2]$, then 
		$w \in \S$ and $\alpha$ is the anti-clockwise oriented angle between $u$ and $w$. 
		By linearity we obtain  
		\begin{equation}\label{Minnesota}
			\begin{aligned}
				\langle D_w \tilde{f}_r , \phi_f^{\bot}\rangle =
				& \cos\alpha\langle D_u \tilde{f}_r , \phi_f^{\bot}\rangle+ \sin\alpha\langle  D_v \tilde{f}_r , \phi_f^{\bot}\rangle\\
				=&\cos\alpha\langle D_u \tilde{f}_r , \vec{u}_i\rangle\langle \phi_f^{\bot}, \vec{u}_i\rangle 
				+ \cos\alpha\langle D_u \tilde{f}_r , \vec{v}_i\rangle\langle \phi_f^{\bot}, \vec{v}_i\rangle 
				+\sin\alpha\langle D_v \tilde{f}_r , \phi_f^{\bot}\rangle.
			\end{aligned}
		\end{equation}
		Given that $[x,y]\in O_i\cap B(0,2R)$, then we can uniquely express 
		$$
		[x,y] = a\tilde{\omega}_i + b\tilde{\omega}_i^{\bot}\text{ for }
		a\in [\tfrac{1}{2}R, 2R]\text{ and }b\in[-r_i,r_i]. 
		$$
		Further there exists a unique $\theta \in[0,2\pi)$ and using our standard notation that $\tilde{\theta} = [\cos\theta,\sin\theta]$ and $\tilde{\theta}^{\bot} = [-\sin\theta, +\cos\theta]$ we have $[x,y] = \sqrt{a^2+b^2}\tilde{\theta}$. 
		We plan to use \eqref{Minnesota}, with $w = -\tilde{\theta}^{\bot}$, $u = - \tilde{\omega}_i^{\bot}$ and $v = \tilde{\omega}_i$. The situation is depicted in Fig.~\ref{Fig:Prime}.
		The angle between $u$ and $w$ is the same as the angle between $v$ and $\tilde{\theta}$ and using 
 $r_i\leq \tfrac{d^2R}{432L^4}$ and $d\leq L^2$ we calculate that (see Fig~\ref{Fig:Prime})
\eqn{alphasize}
		$$
		|\sin\alpha| \leq |\tan\alpha| \leq \frac{r_i}{\frac{R}{2}} \leq \frac{d^2}{216L^4}\leq \frac{d}{216L^2}\leq \frac{1}{216}
		\text{ implying }\cos\alpha\geq\frac{9}{10}.
		$$ 
		Using also \eqref{MuskTurtle} ($|D\tilde{f}_{\r}| \leq 8L$) in \eqref{Minnesota} we get
		\begin{equation}\label{Boise}
			\begin{aligned}
				\langle D_{\tilde{\theta}^{\bot}} \tilde{f}_r (t \tilde{\theta}), \phi_f^{\bot}(t \tilde{\theta})\rangle \geq 
				& \frac{9}{10}\langle D_{\tilde{\omega}_i^{\bot}} \tilde{f}_r(t \tilde{\theta}) , \vec{u}_i\rangle \langle \phi_f^{\bot}(t \tilde{\theta}), \vec{u}_i\rangle 
				- \bigl|\langle D_{\tilde{\omega}_i^{\bot}} \tilde{f}_r (t \tilde{\theta}), \vec{v}_i\rangle\bigr|
				\langle \phi_f^{\bot}(t \tilde{\theta}), \vec{v}_i\rangle \\
				&-\frac{d}{216 L^2}\bigl|\langle D_{\tilde{\omega}_i} \tilde{f}_r (t \tilde{\theta}), \phi_f^{\bot}(t \tilde{\theta})\rangle\bigr|\\
				\geq &\frac{9}{10}\langle D_{\tilde{\omega}_i^{\bot}} \tilde{f}_r(t \tilde{\theta}) , \vec{u}_i\rangle\langle \phi_f^{\bot}(t \tilde{\theta}), \vec{u}_i\rangle  
				- 8L\langle \phi_f^{\bot}(t \tilde{\theta}), \vec{v}_i\rangle -\frac{d}{216L^2}8L.
			\end{aligned}
		\end{equation}		
		By \eqref{Gecco} we have that $\langle D_{-\tilde{\omega}_i^{\bot}} \tilde{f}_r(t \tilde{\theta}) , \vec{u}_i\rangle \geq \tfrac{9d}{10L}$ (note that in order to apply \eqref{Gecco} we take $-\tilde{\omega}_i^{\bot}$ as the clockwise rotation of $\tilde{\omega}_i$ because also $[1,0]$ is the clockwise rotation of $[0.1]$). Note that $\phi_f^{\bot}$ is anti-clockwise perpendicular
		to $\phi_f$ but $\vec{u}_i$ is clockwise perpendicular to $\vec{v}_i$ and hence $\langle \phi_f^{\bot}(t \tilde{\theta}), \vec{u}_i\rangle$ is negative. Combining the two previous facts we get that $\langle D_{\tilde{\omega}_i^{\bot}} \tilde{f}_r(t \tilde{\theta}) , \vec{u}_i\rangle\langle \phi_f^{\bot}(t \tilde{\theta}), \vec{u}_i\rangle \geq \tfrac{9d}{10L}|\langle \phi_f^{\bot}(t \tilde{\theta}), \vec{u}_i\rangle|$. Applying this in \eqref{Boise} we get
		\begin{equation}\label{El Paso}
			\begin{aligned}
				\langle D_{\tilde{\theta}^{\bot}} \tilde{f}_r(t \tilde{\theta}) , \phi_f^{\bot}(t \tilde{\theta})\rangle 
				\geq &\frac{81 d}{100L}|\langle \phi_f^{\bot}(t \tilde{\theta}), \vec{u}_i\rangle|  - 8L\langle \phi_f^{\bot}(t \tilde{\theta}), \vec{v}_i\rangle -\frac{d}{27 L}.
			\end{aligned}
		\end{equation}

		The factors $\langle \phi_f^{\bot}, \vec{u}_i\rangle$ and $\langle\phi_f^{\bot}, \vec{v}_i\rangle$ are a question of the geometry of $\tilde{f}_{\r}(O_i)$. 
		%From \eqref{deffr} we obtain that $\tilde{f}_{\r} = f$ on the ray $\tilde{\omega}_i\er^+$ since for $[x,y]=t\tilde{\omega}_i$ we have $\langle[x,y], \tilde{\omega}_i^{\bot} \rangle=0$. 
			We express 
			$$
			[x,y] = a\tilde{\omega}_i + b\tilde{\omega}_i^{\bot}\text{ for }\tfrac{1}{2}R\leq a\leq R\text{ and }-r_i \leq b\leq r_i.
				$$
		We use \eqref{BeardedDragon} ($|\tilde{f}_{\r} - f| \leq 2Lr_i\text{ on }O_i$), \eqref{Becca}, \eqref{Andrea}, 	$r_i \leq \tfrac{dR}{432L^2}$, $a^2+b^2\leq 2R^2$ and $R\leq \tfrac{d}{1000ML}$
		%. We express $[x,y] = a\tilde{\omega} + b\tilde{\omega}_i^{\bot}$ for $\tfrac{1}{2}R\leq a\leq R$ and $-r_i \leq b\leq r_i$ and we get that
		and we get 
		\begin{equation}\label{Portland}
			\begin{aligned}
				\bigl|\tilde{f}_{\r}(a\tilde{\omega}_i + b\tilde{\omega}_i^{\bot}) \bigr| 
				\geq& \bigl|h(a\tilde{\omega}_i) \bigr|-\bigl|h(a\tilde{\omega}_i + b\tilde{\omega}_i^{\bot}) -h(a\tilde{\omega}_i)\bigr|\\
				&-\bigl|\tilde{f}_{\r}(a\tilde{\omega}_i + b\tilde{\omega}_i^{\bot}) -f(a\tilde{\omega}_i + b\tilde{\omega}_i^{\bot})\bigr|
				-\bigl|h(a\tilde{\omega}_i + b\tilde{\omega}_i^{\bot}) -f(a\tilde{\omega}_i + b\tilde{\omega}_i^{\bot})\bigr|\\
				\geq& |h(a\tilde{\omega}_i)|-Lr_i - 2Lr_i - \frac{M}{2}(a^2+b^2) \\
				\geq & a|D_{\tilde{\omega}_i} h(\tilde{\omega}_i)|  -\frac{dR}{16L} - \frac{dR}{16L}\\
				%\geq & \frac{3}{4}R\frac{d}{L}  - \frac{d}{8L}R\\
				\geq &\frac{3d}{8L}R.
			\end{aligned}
		\end{equation}
	By \eqref{MuskTurtle}
		$$
			\bigl|\tilde{f}_{\r}(a\tilde{\omega}_i +b\tilde{\omega}_i^{\bot})  - \tilde{f}_{\r}(a\tilde{\omega}_i) \Bigr| 
			\leq 8Lr_i \leq \frac{d^2}{54 L^3}R.
		$$
		Combining these two facts and calling $\zeta_1$ the angle between $\phi_f(a\tilde{\omega}_i +b\tilde{\omega}_i^{\bot})$ and $\phi_f(a\tilde{\omega}_i)$ we get
		\begin{equation}\label{Phoenix}
				|\tan\zeta_1|\leq \frac{\frac{d^2}{54 L^3}R}{\frac{3d}{8L}R} \leq \frac{d}{20 L^2}.
		\end{equation}

		On the other hand using $\rho < \tfrac{d^2}{1000 (M+1)(L+1)^4}$ and $|D^2Q_i|\leq M$ we have
		$$
			\bigl|D_{\tilde{\omega}_i}Q_i(\rho\tilde{\omega}_i) - D_{\tilde{\omega}_i}Q_i(0,0)\bigr|
			\leq M\rho \leq \frac{d^2}{1000(L+1)^4}.
		$$
		Therefore, because $\vec{v}_i = \tfrac{D_{\tilde{\omega}_i}Q_i(\rho\tilde{\omega}_i)}{|D_{\tilde{\omega}_i}Q_i(\rho\tilde{\omega}_i)|}$ and $|D_{\tilde{\omega}_i}Q_i(t\tilde{\omega}_i)| \geq \tfrac{d}{L}$ (see \eqref{Andrea}) we have analogously to \eqref{rozdiluhlu} that
		\begin{equation}\label{Pensacola}
			\Bigl|\vec{v}_i - \frac{D_{\tilde{\omega}_i}Q_i(0,0)}{|D_{\tilde{\omega}_i}Q_i(0,0)|}\Bigr|
			\leq \frac{d^2}{1000(L+1)^4} \frac{L}{d}2 = \frac{d}{500(L+1)^3}.
		\end{equation}
		From \eqref{deffr} we obtain that $\tilde{f}_{\r} = f$ on the ray $\tilde{\omega}_i\er^+$ since for $[x,y]=t\tilde{\omega}_i$ we have $\langle[x,y], \tilde{\omega}_i^{\bot} \rangle=0$. 
		Hence $\tilde{f}_{\r}$ is smooth along this ray and 
		$$
			\phi_{f}(t\tilde{\omega}_i)=\frac{\tilde{f}_{\r}(t\tilde{\omega}_i)}{|\tilde{f}_{\r}(t\tilde{\omega}_i)|} 
			= \frac{\int_0^t D_{\tilde{\omega}_i} \tilde{f}_{\r}(s\tilde{\omega}_i) ds}{|\int_0^t D_{\tilde{\omega}_i} \tilde{f}_{\r}(s\tilde{\omega}_i) ds|} 
			= \frac{\tfrac{1}{t}\int_0^t D_{\tilde{\omega}_i} Q_i (s\tilde{\omega}_i) ds}
			{|\frac{1}{t}\int_0^t D_{\tilde{\omega}_i} Q_i(s\tilde{\omega}_i) ds|}
		$$
		and so using $t\leq R \leq \tfrac{d^2}{1000(M+1)(L+1)^4}$ using \eqref{elementary}
		\begin{equation}\label{Chicago}
			\Bigl|\phi_{f}(t\tilde{\omega}_i) - \frac{D_{\tilde{\omega}_i}Q_i(0,0)}{|D_{\tilde{\omega}_i}Q_i(0,0)|}\Bigr| 
			\leq \frac{\tfrac{1}{t}\int_0^t|D_{\tilde{\omega}_i} Q_i (s\tilde{\omega}_i) - D_{\tilde{\omega}_i}Q_i(0,0) ds|}{|D_{\tilde{\omega}_i} Q_i(0,0)| }2 \leq \frac{Mt}{\frac{d}{L}}2 \leq \frac{d}{500(L+1)^3}.
		\end{equation}
		Combining \eqref{Pensacola} and \eqref{Chicago} we obtain that the angle between $\phi_f(a\tilde{\omega}_i)$ and $\vec{v}_i$ (call it $\zeta_2$) satisfies 
		$$
		\tan \frac{\zeta_2}{2}\leq \frac{d}{500(L+1)^3}.
		$$
		Call $\zeta_3$ the angle between $\phi_f(a\tilde{\omega}_i + b\tilde{\omega}_i^{\bot})$ and $\vec{v}_i$. From the previous inequality 
		and \eqref{Phoenix} we obtain that $|\zeta_3|\leq |\zeta_1|+|\zeta_2|$ implies 
		$$
		|\sin(\zeta_3)|\leq |\sin\zeta_1|+|\sin \zeta_2|
		\leq |\sin\zeta_1|+2\bigl|\sin\frac{\zeta_2}{2}\bigr|\leq \frac{d}{20L^2}+2\frac{d}{500(L+1)^3}  \leq \frac{d}{15L^2}. 
		$$
		Then also 
		$$
		|\langle \phi_f^{\bot}, \vec{v}_i\rangle| = |\sin(\zeta_3)| \leq \frac{d}{15L^2}
		$$
		and since $\tfrac{d}{L^2}\leq 1$ also $|\langle \phi_f^{\bot}, \vec{u}_i\rangle| = |\cos(\zeta_3)| \geq \tfrac{9}{10}$.
		%Note that $\phi_f^{\bot}$ is anticlockwise perpendicular to $\phi_f$ and $\vec{u}_i$ is clockwise perpendicular to $\vec{v}_i$ and hence $\langle \phi_f^{\bot}, \vec{u}_i\rangle$ is negative. 
		Applying this in \eqref{El Paso} we get
		$$
			\langle D_{\tilde{\theta}^{\bot}} \tilde{f}_r(x,y) , \phi_f^{\bot}(x,y)\rangle \geq \frac{9}{10}\frac{81d}{100L}  - \frac{8d}{15L} -\frac{d}{27L} \geq \frac{d}{10L}
		$$
		for all $[x,y]\in O_i\cap B(0,R)$. Because together \eqref{Portland} and \eqref{MuskTurtle} imply that $\R_f(t)\approx t$ we conclude from the above using \eqref{Denver} that 
		$$
		\langle \frac{\partial}{\partial\theta}\phi_f(t\cos\theta, t\sin\theta) , \phi_f^{\bot}(t\tilde{\theta}) \rangle \geq C. 
		$$
		%$$
		%	\frac{\partial}{\partial \theta} \phi_{f}(t\cos\theta,t\sin\theta) \geq C
		%$$ 
		%also for $[x,y]\in O_i\cap B(0,R)$ and with respect to \eqref{Nashville} $\frac{\partial}{\partial \theta} \phi_{f}(t\cos\theta,t\sin\theta) \geq C$ for all $[x,y] \in B(0,R)$.
		
		\step{3}{Proving that $\tfrac{\partial}{\partial t} \R_f(t\tilde{\theta})\geq C>0$}{Wyoming} 
		
	%	Let us denote $D_{t} \R_f(t\tilde{\theta})=\tfrac{\partial}{\partial t}\R_f(t\tilde{\theta})$. 
		In this section we show that $\tfrac{\partial}{\partial t} \R_f(t\tilde{\theta})>0$ for all $\tfrac{3}{4}R\leq t\leq R$. For $[x,y] = t\tilde{\theta}$, where $t = |[x,y]|$ and $\tilde{\theta} = \frac{[x,y]}{|[x,y]|}$ we consider firstly $t\tilde{\theta} \notin \bigcup_{i=1}^N O_i$ using the following facts. Firstly, for all $w\in \S$, we have $h(tw) = t D_wh(w)$ and $|D_{w}h| \geq \tfrac{d}{L}$. This means that 
		\eqn{uuu}
		$$
		%\frac{\partial}{\partial t}\R_h(tw)=
		\begin{aligned}
		\bigl|D_wh(w)\bigr|=\Bigl|\frac{\partial}{\partial t} h(t w)\Bigr| \geq \frac{d}{L}\text{ and }\\		
		\Bigl\langle \frac{\partial}{\partial t} h(tw) , \phi_h(tw) \Bigr\rangle =
		\Bigl\langle \frac{\partial}{\partial t} \bigl(t D_wh(w)\bigr) , \frac{t D_wh(w)}{|t D_wh(w)|} \Bigr\rangle
		=\Bigl|\frac{\partial}{\partial t} h(t w)\Bigr|
		\text{ for all }w\in \S.\\
		\end{aligned}
		$$ 
		%{\color{red}because $h$ is linear on rays. A linear function $l$ in the variable $t$ on a segment has the form $l(t) = at$. In this specific case $l = \R_h$ on the ray $w\er^+$ then $a = |D_wh(w)|$ which we have established $a\geq \tfrac{d}{L}$. The cooeficient $a$ for linear functions is constant and so $l'(t)= a$ for all $t$. In other words $\tfrac{\partial}{\partial t}\R_h(tw) \geq \tfrac{d}{L}$}. 
	Secondly, because $t\leq R\leq \frac{d}{1000ML}$ and $t\tilde{\theta} \notin \bigcup_{i=1}^N O_i$ we have using \eqref{Becca},
		$$
				\bigl|\tilde{f}_{r}(t\tilde{\theta}) - h(t\tilde{\theta})\bigr|
				= \bigl|f(t\tilde{\theta}) - h(t\tilde{\theta})\bigr|
				\leq \frac{M}{2}t^2
				\leq \frac{dt}{L 2000} 
				\leq \frac{|h(t\tilde{\theta})|}{2000 }.
		$$
		This implies (using the fact that $|\phi_h- \phi_f|$ is less than the arclength between them on $\S$) that
		$$
			|\phi_h- \phi_f| \leq \arctan\frac{1}{1000} \leq \frac{1}{1000}.
		$$
		Finally we obtain using \eqref{Ella} and $t\leq R\leq \frac{d}{1000ML}$
		$$
		|\tfrac{\partial}{\partial t}\tilde{f}_{\r}(t\tilde{\theta}) - \tfrac{\partial}{\partial t} h(t\tilde{\theta})| 
		\leq Mt 
		\leq \frac{d}{1000 L}
		\leq \frac{|\tfrac{\partial}{\partial t} h(t\tilde{\theta})|}{1000}.
		$$ 
		%In order to conclude that $\tfrac{\partial}{\partial t} \R_f(t\tilde{\theta})>0$ we need to prove that $\langle \tfrac{\partial}{\partial t}\tilde{f}_{\r}(t\tilde{\theta}), \phi_f(t\tilde{\theta})\rangle>0$. We have already concluded that $\tfrac{\partial}{\partial t}\R_h(t\tilde{\theta}) = |\tfrac{\partial}{\partial t} h(t\tilde{\theta})| = \langle \tfrac{\partial}{\partial t} h(t\tilde{\theta}) , \phi_h(t\tilde{\theta}) \rangle \geq \tfrac{d}{L}$ and so
		We estimate with the help of \eqref{uuu}
		$$
			\begin{aligned}
				\langle \tfrac{\partial}{\partial t}\tilde{f}_{\r}(t\tilde{\theta}), \phi_f(t\tilde{\theta})\rangle
				\geq & \langle \tfrac{\partial}{\partial t} h(t\tilde{\theta}) , \phi_h(t\tilde{\theta}) \rangle - |\langle \tfrac{\partial}{\partial t}\tilde{f}_{\r}(t\tilde{\theta}) - \tfrac{\partial}{\partial t} h(t\tilde{\theta}), \phi_f(t\tilde{\theta})\rangle| \\
				&\quad - |\langle \tfrac{\partial}{\partial t} h(t\tilde{\theta}), \phi_f(t\tilde{\theta}) - \phi_h(t\tilde{\theta})\rangle|\\
				\geq & |\tfrac{\partial}{\partial t} h(t\tilde{\theta})| - \frac{|\tfrac{\partial}{\partial t} h(t\tilde{\theta})|}{1000} - \frac{|\tfrac{\partial}{\partial t} h(t\tilde{\theta})|}{1000}\\
				\geq& \frac{499d}{500L}
			\end{aligned}
		$$
		and using $\tilde{f}_{\r}(t\tilde{\theta})=|\tilde{f}_{\r}(t\tilde{\theta})|\phi_f(t\tilde{\theta})$ and 
		$\tfrac{\partial}{\partial t}\langle\phi_f,\phi_f \rangle=0$ we obtain
		\eqn{trivial}
		$$
		\langle \tfrac{\partial}{\partial t}\tilde{f}_{\r}(t\tilde{\theta}), \phi_f(t\tilde{\theta})\rangle=
		\tfrac{\partial}{\partial t}\bigl|\tilde{f}_{\r}(t\tilde{\theta})\bigr|\langle \phi_f(t\tilde{\theta}), \phi_f(t\tilde{\theta})\rangle
		+\bigl|\tilde{f}_{\r}(t\tilde{\theta})\bigr|\langle \tfrac{\partial}{\partial t}\phi_f(t\tilde{\theta}), \phi_f(t\tilde{\theta})\rangle
		=\tfrac{\partial}{\partial t} \R_f(t\tilde{\theta})
		$$
		and hence $\tfrac{\partial}{\partial t} \R_f(t\tilde{\theta})>C$.
		
		When $t\tilde{\theta} \in \bigcup_{i=1}^N O_i$ we use \eqref{trivial} and calculate similarly as in \eqref{Minnesota} and \eqref{Boise} (again $\alpha$ denotes the angle between $\tilde{\theta}$ and $\tilde{\omega}_i$)
		$$
			\begin{aligned}
				\tfrac{\partial}{\partial t}\R_f(t\tilde{\theta})=&\langle \tfrac{\partial}{\partial t}\tilde{f}_{\r}(t\tilde{\theta}), \phi_f(t\tilde{\theta})\rangle=\langle D_{\tilde{\theta}}\tilde{f}_{\r}(t\tilde{\theta}), \phi_f(t\tilde{\theta})\rangle\\
				\geq& \cos\alpha \bigl\langle D_{\tilde{\omega}_i} \tilde{f}_{\r},\phi_f(t\tilde{\theta})\bigr\rangle 
			- \bigl|\sin\alpha \bigl\langle D_{\tilde{\omega}_i^{\bot}} \tilde{f}_{\r},
				\phi_f(t\tilde{\theta})\bigr\rangle\bigr|\\
				\geq&\cos\alpha \Bigl\langle D_{\tilde{\omega}_i} \tilde{f}_{\r},
				\frac{D_{\tilde{\omega}_{i}}f(\rho_i\tilde{\omega}_i)}
				{|D_{\tilde{\omega}_{i}}f(\rho_i\tilde{ \omega}_i)|}\Bigr\rangle\langle \phi_f(t\tilde{\theta}) ,\phi_f(t\tilde{\omega}_i)\rangle\\
				& \quad  - \Bigl| \Bigl\langle D_{\tilde{\omega}_i} \tilde{f}_{\r}, \Bigl(\frac{D_{\tilde{\omega}_{i}}f(\rho_i\tilde{\omega}_i)}{|D_{\tilde{\omega}_{i}}f(\rho_i\tilde{\omega}_i)|}\Bigr)^{\bot}\Bigr\rangle\Bigr|- |\sin\alpha|\ |D_{\tilde{\omega}_i^{\bot}}\tilde{f}_{\r}|.\\
			\end{aligned}
		$$
		In \eqref{Caiman} (term corresponding to $b_2$) we estimated that 
		$$
		\Bigl\langle D_{\tilde{\omega}_i} \tilde{f}_{\r},\frac{D_{\tilde{\omega}_{i}}f(\rho_i\tilde{\omega}_i)} {|D_{\tilde{\omega}_{i}}f(\rho_i\tilde{ \omega}_i)|}\Bigr\rangle 
		\geq \frac{499}{500}|D_{\tilde{\omega}_{i}}f(\rho_i\tilde{ \omega}_i)| \geq \frac{499d}{500L}
		$$ 
		and (term corresponding to $b_1$) that 
		$$\Bigl| \Bigl\langle D_{\tilde{\omega}_i} \tilde{f}_{\r}, \Bigl(\frac{D_{\tilde{\omega}_{i}}f(\rho_i\tilde{\omega}_i)}{|D_{\tilde{\omega}_{i}}f(\rho_i\tilde{\omega}_i)|}\Bigr)^{\bot}\Bigr\rangle\Bigr|\leq \frac{d}{50L}.
		$$ 
		Now computing similarly as in \eqref{Boise} we obtain that $\sin\alpha \leq \tfrac{d}{216L^2}$, $\cos\alpha \geq \tfrac{9}{10}$ (see \eqref{alphasize}) and applying the previous to the above estimate we get
		\eqn{ahaaa}
		$$
			\begin{aligned}
				%&\text{ VYSVETLIT ODHAD PRVNIHO A DRUHEHO CLENU?- ODKAZY?}\\
				\tfrac{\partial}{\partial t}\R_f(t\tilde{\theta})\geq& \frac{9}{10} \frac{499d}{500L}\langle \phi_f(t\tilde{\theta}) ,\phi_f(t\tilde{\omega}_i)\rangle - \frac{d}{50L} - 8L\frac{d}{216L^2}.\\
			\end{aligned}
		$$
		Call $t\tilde{\theta} = a\tilde{\omega}_i + b\tilde{\omega}_i^{\bot}$. Then we obtain
		%$\tilde{f}_{\r} = f$ on $\tilde{\omega}_i\times [0,R]$, 
		using \eqref{elementary}, \eqref{MuskTurtle}, \eqref{Portland} and $r_i \leq \frac{Rd}{1200 L^2}$ that
		$$
			| \phi_f(t\tilde{\theta}) - \phi_{f}(a\tilde{\omega}_i) |
			\leq\frac{|\tilde{f}_{\r}(t\tilde{\theta}) - \tilde{f}_{\r}(a\tilde{\omega}_i) |}{|\tilde{f}_{\r}(t\tilde{\theta})|} 2 
			\leq \frac{8Lr_i} {\frac{3d}{8L}R} 2\leq \frac{1}{100}.
		$$
		Similarly using \eqref{elementary}, \eqref{Becca} and \eqref{Andrea} (obviously $\phi_h(\rho_i\tilde{\omega}_i) =\phi_{h}(a\tilde{\omega}_i)$) we have
		$$
			\begin{aligned}
				| \phi_f(\rho_i\tilde{\omega}_i) - \phi_{f}(a\tilde{\omega}_i) |
				\leq& | \phi_f(\rho_i\tilde{\omega}_i) - \phi_h(\rho_i\tilde{\omega}_i) | +|  \phi_{f}(a\tilde{\omega}_i) - \phi_{h}(a\tilde{\omega}_i) | \\
				\leq& \frac{M\rho^2_i}{2} \frac{L}{d\rho_i} 2+ \frac{Ma^2}{2} \frac{L}{da}2\\
				\leq&  2\frac{ML}{d}\rho_i
				\leq \frac{1}{100}.
			\end{aligned}
		$$
		Since $| \phi_f(t\tilde{\theta}) - \phi_{f}(\rho_i\tilde{\omega}_i) |< \frac{1}{50}$ we obtain $\langle \phi_f(t\tilde{\theta}) ,\phi_f(t\tilde{\omega}_i)\rangle \geq \frac{1}{2}$ and so continuing the estimate \eqref{ahaaa}
		$$
				\frac{\partial}{\partial t}\R_f(t\tilde{\theta})\geq \frac{9}{10}\frac{499d}{1000L}\frac{1}{2} - \frac{d}{50L}- \frac{d}{27L} \\
				\geq C >0.
		$$

		\step{4}{Proving that $g_{\r}$ is a diffeomorphism}{Nebraska}
		
		Call $\lambda= \tfrac{d}{4L}$. We need to redefine our mapping close to the origin so it is smooth there. We define it as a proper 
		interpolation between a linear mapping $[x,y]\to\lambda [x,y]$ and our mapping $\tilde{f}_r$. We define it as 
		$$
			g_{\r}(x, y) = \R_g(x, y)\phi_g(x, y),
		$$
		where
		$$
			\begin{aligned}
				\R_g(t\tilde{\theta})
				&= \bigl(1-\eta\big(\tfrac{8t-7R}{R}\big)\big)\lambda t +  \eta\big(\tfrac{8t-7R}{R}\big) \R_f(t\tilde{\theta}\bigr) \quad \text{ and}\\
				%\R_g(x, y) &= \bigl(1-\eta\big(\tfrac{8|[x,y]|-7R}{R}\big)\bigr)\lambda |[x,y]| +  \eta\big(\tfrac{8|[x,y]|-7R}{R}\bigr) |\tilde{f}_{\r}(x,y)|\\
				\phi_g(t\tilde{\theta})
				&= \bigl(1-\eta\big(\tfrac{8t-6R}{R}\big)\big)\tilde{\theta} +  \eta\big(\tfrac{8t-6R}{R}\big) \phi_f(t\tilde{\theta}\bigr).\\% \quad\text{ that is}\\
				%\phi_g(x,y)&= \bigl(1-\eta\big(\tfrac{8|[x,y]|-6R}{R}\big)\bigr)\frac{[x,y]}{|[x,y]|} +  \eta\big(\tfrac{8|[x,y]|-6R}{R}\bigr) \frac{\tilde{f}_{\r}(x,y)}{|\tilde{f}_{\r}(x,y)|}.
			\end{aligned}
		$$
		Note that this is equal to $\lambda[x,y]$ on $B(0,\tfrac{6}{8}R)$ and it is equal to $\tilde{f}_r$ outside of $B(0,R)$. 
		It is changing the angle on $B(0,\tfrac{7}{8}R)\setminus B(0,\tfrac{6}{8}R)$ while keeping the distance from the origin of a map $\lambda [x,y]$ and it is changing the distance from the origin on $B(0,R)\setminus B(0,\tfrac{7}{8}R)$ while keeping the angle of $\tilde{f}_r$. 
		
		Immediately from the definition of $g_{\r}$ it is obvious that it is smooth since $[x,y]\to\lambda \cdot [x,y]$ is smooth and $\eta, \R_f$ and $\phi_f$ are all smooth away from the origin. 
		%Further by the definition of $\eta$, $g_{\r}(x,y) = \tilde{f}_{r}(x,y)$ as soon as $|(x,y)| \geq R$. 
		Let us define $\hat{\phi}_f\in[0,2\pi)$ (resp. $\hat{\phi}_g$) as the corresponding angle of $\phi_f\in\S$ (resp. $\phi_g$) modulo $2\pi$. From Step~\ref{MidWest} we know
		$$
		\bigl\langle \tfrac{\partial}{\partial\theta}\phi_f(t\cos\theta, t\sin\theta) , \phi_f^{\bot}(t\cos\theta, t\sin\theta) \bigr\rangle\geq C
		\text{ for all }t\in[\tfrac{3}{4}R,R]\text{ and }\theta.
		$$
		Using derivative of composed mapping for
		$$
		\phi_f(t\cos\theta, t\sin\theta)
		=\bigl[\cos \hat{\phi}_f(t\cos\theta, t\sin\theta),\sin \hat{\phi}_f(t\cos\theta, t\sin\theta)\bigr]$$ 
		in the above inequality and 
		$$
		\phi_f^{\bot}(t\cos\theta, t\sin\theta)=\bigl[-\sin \hat{\phi}_f(t\cos\theta, t\sin\theta),\cos \hat{\phi}_f(t\cos\theta, t\sin\theta)\bigr]
		$$
		this implies that $\tfrac{\partial}{\partial \theta}\hat{\phi}_f(t\cos\theta, t\sin\theta)\geq C$. %This means that the angle of $\tilde{f}_{\r}(t\cos\theta,t\sin\theta)$ is an increasing function of $\theta$ and thus the same holds for 
		%$g_{\r}$ as this angle is a convex combination of two increasing functions, i.e. 
		It follows that
		$$
		\frac{\partial}{\partial \theta}\hat{\phi}_g(t\cos\theta, t\sin\theta)=
		\bigl(1-\eta\big(\tfrac{8t-6R}{R}\big)\big)\frac{\partial}{\partial \theta}(\theta) +  \eta\big(\tfrac{8t-6R}{R}\big) \frac{\partial}{\partial \theta}\hat{\phi}_f(t\cos\theta, t\sin\theta)\geq C>0. 
		$$
		Further, because (see \eqref{Kansas City} and \eqref{Portland}) 
		$$
		|\tilde{f}_{\r}(t\cos\theta, t\sin\theta)| \geq \frac{3dR}{8L} > \lambda t\text{ for all }\frac{3}{4}R \leq t \leq R
		$$
		we have that
		$$
			\tfrac{\partial}{\partial t} \R_g(t\tilde{\theta})  = \bigl(1-\eta\big(\tfrac{8t-7R}{R}\big)\bigr)\lambda +  \eta\bigl(\tfrac{8t-7R}{R}\bigr) \tfrac{\partial}{\partial t}\R_f  + \frac{8}{R}\eta'\bigl(\tfrac{8t-7R}{R}\bigr)(\R_f(t\tilde{\theta}) - \lambda t)
		$$
		but as shown above each of the above terms is positive. Because 
		$$
		\frac{\partial}{\partial t}\R_{g}(t\cos\theta, t\sin\theta)\geq C>0\text{ and }\frac{\partial}{\partial \theta}\hat{\phi}_g(t\cos\theta, t\sin\theta)\geq C>0\text{ for all }0<t\leq R\text{ and all }\theta
		$$ 
		we easily conclude that $g_{\r}$ is a diffeomorphism on $\overline{B(0,R)}$ by considering the three parts $B(0,\tfrac{6}{8}R)$, $B(0,\tfrac{7}{8}R)\setminus B(0,\tfrac{6}{8}R)$ and $B(0,R) \setminus B(0,\tfrac{7}{8}R)$ separately. Further, because $g_{\r}$ coincides with the diffeomorphism $\tilde{f}_{\r}$ on $\overline{B(0,2R) \setminus B(0,R)}$ %and with diffeomorphisms $\lambda[x,y]$ on $B(0,\tfrac{R}{2})$ 
	it must be a diffeomorphism on $B(0,2R)$.
		
		\step{5}{Estimating $\int_{B(0,R)}|D^2g_{\r}|$}{Oregon}
		
		Clearly $D^2g_{\r}=D^2(\lambda [x,y])=0$ for $[x,y]\in B(0,\tfrac{6}{8}R)$ so it remains to estimate it for 
		$[x,y]\in B(0,R)\setminus B(0,\tfrac{6}{8}R)$. 
		We have $|D^2g_{\r}|=|D^2(R_g\phi_g)|$. We calculate
		$$
				\begin{aligned}
					D \R_g =& \bigl(1-\eta\big(\tfrac{8|[x,y]|-7R}{R}\big)\bigr)\lambda D|[x,y]| +  \eta\bigl(\tfrac{8|[x,y]|-7R}{R}\bigr) D|\tilde{f}_{\r}(x,y)| \\
					& + D[x,y]\frac{8}{R}\eta'(\tfrac{8|[x,y]|-7R}{R})\bigl(|\tilde{f}_{\r}(x,y)|- \lambda |[x,y]|\bigr)\\
					\end{aligned}
		$$
		and (see Section \ref{higherderivatives})
		$$
		\begin{aligned}			
					D^2\R_g = & \bigl(1-\eta\big(\tfrac{8|[x,y]|-7R}{R}\big)\bigr)\lambda D^2|[x,y]|  + \frac{8}{R}\eta'\bigl(\tfrac{8|[x,y]|-7R}{R}\bigr)\lambda D[x,y]D|[x,y]|\\
					&+  \eta\bigl(\tfrac{8|[x,y]|-7R}{R}\bigr) D^2|\tilde{f}_{\r}(x,y)| + \frac{8}{R}\eta'\bigl(\tfrac{8|[x,y]|-7R}{R}\bigr)D[x,y]D|\tilde{f}_{\r}(x,y)| \\
					& + D^2[x,y]\frac{8}{R}\eta'\bigl(\tfrac{8|[x,y]|-7R}{R}\bigr)
					\bigl(|\tilde{f}_{\r}(x,y)|- \lambda [x,y]\bigr)\\
					& + \frac{64}{R^2}D[x,y]D[x,y]\eta''\bigl(\tfrac{8|[x,y]|-7R}{R}\bigr)
					\bigl(|\tilde{f}_{\r}(x,y)|- \lambda |[x,y]|\bigr)\\
					& + \frac{8}{R}D[x,y]\eta'\bigl(\tfrac{8|[x,y]|-7R}{R}\bigr)
					\bigl(D|\tilde{f}_{\r}(x,y)|- \lambda D|[x,y]|\bigr).\\
				\end{aligned}
		$$
		We now separate $B(0,R)\setminus B(0,\tfrac{6}{8}R)$ into parts $B(0,R)\setminus [B(0,\tfrac{6}{8}R)\cup \bigcup_{i=1}^NO_i]$ and the parts $\bigcup_{i=1}^NO_i$. For $[x,y]\notin\bigcup_{i=1}^NO_i$ we know that $\tilde{f}_{\r}=f$ and hence
		$$
		|\tilde{f}_{\r}(x,y)| \leq CLR,\ D|\tilde{f}_{\r}| \leq L\ \text{ and } D^2|\tilde{f}_{\r}|\leq M. 
		$$
		By elementary computation 
		$$D|[x,y]|\leq 1,\ \big|D^2|[x,y]|\big| \leq \frac{C}{|[x,y]|}, \ |D[x,y] |\leq C\text{ and }D^2[x,y]=0. 
		$$
		Therefore 
		$$
			\begin{aligned}
				|\R_g(x,y)| \leq  CR,\ |D\R_g(x,y)| \leq  C\text{ and } |D^2\R_g(x,y)| \leq  \frac{C}{R} \\
%				|D\R_g(x,y)| \leq &  C \quad&\text{ for } |[x,y]|\leq R\\
				%|D^2\R_g(x,y)| \leq &  \frac{C}{R} + M \quad &\text{ for } |[x,y]|\leq R\\
			\end{aligned}
		$$
		for all $[x,y]\in B(0,R)\setminus [B(0,\tfrac{6}{8}R)\cup \bigcup_{i=1}^NO_i]$. Further
		$$
			\begin{aligned}
				D\phi_g = & \bigl(1-\eta\big(\tfrac{8|[x,y]|-6R}{R}\big)\bigr) D\frac{[x,y]}{|[x,y]|} 
				+\frac{8}{R}D[x,y]\eta'\bigl(\tfrac{8|[x,y]|-6R}{R}\bigr)\Bigl(\frac{\tilde{f}_{\r}(x,y)}{|\tilde{f}_{\r}(x,y)|} - \frac{[x,y]}{|[x,y]|}\Bigr)   \\
				& +\eta\big(\tfrac{8|(x,y)|-6R}{R}\big) D \frac{\tilde{f}_{\r}(x,y)}{|\tilde{f}_{\r}(x,y)|}, \\
		\end{aligned}
		$$		
				and (see Section \ref{higherderivatives})
		$$
				\begin{aligned}
				D^2\phi_g = & \bigl(1-\eta\big(\tfrac{8|[x,y]|-6R}{R}\big)\bigr) D^2\frac{[x,y]}{|[x,y]|} - \frac{8}{R}D[x,y]\eta'\big(\tfrac{8|[x,y]|-6R}{R}\big) D\frac{[x,y]}{|[x,y]|}  \\
				&+\frac{8}{R}D^2[x,y]\eta'\bigl(\tfrac{8|[x,y]|-6R}{R}\bigr)
				\Bigl(\frac{\tilde{f}_{\r}(x,y)}{|\tilde{f}_{\r}(x,y)|} - \frac{[x,y]}{|[x,y]|}\Bigr) \\
				&+\frac{64}{R^2}D[x,y] D[x,y]\eta''\bigl(\tfrac{8|[x,y]|-6R}{R}\bigr)
				\Bigl(\frac{\tilde{f}_{\r}(x,y)}{|\tilde{f}_{\r}(x,y)|} - \frac{[x,y]}{|[x,y]|}\Bigr) \\
				&+ \frac{8}{R}D[x,y] \eta' \bigl( \tfrac{8|[x,y]|-6R}{R} \bigr)  
				\Bigl(D\frac{\tilde{f}_{\r}(x,y) } {|\tilde{f}_{\r}(x,y)|}-D \frac{[x,y]}{|[x,y]|}\Bigr)\\
				& +\frac{8}{R}D[x,y]\eta'\bigl(\tfrac{8|[x,y]|-6R}{R}\bigr) 
				D \frac{\tilde{f}_{\r}(x,y)}{|\tilde{f}_{\r}(x,y)|} + \eta\bigl(\tfrac{8|[x,y]|-7R}{R}\bigr)D^2\frac{\tilde{f}_{\r}(x,y)}{|\tilde{f}_{\r}(x,y)|}.
			\end{aligned}
		$$
		It is easy to calculate that 
		$$
		\Bigl|D\frac{[x,y]}{|[x,y]|}\Bigr| \leq \frac{C}{R}\text{ and }\Bigl|D^2\frac{[x,y]}{|[x,y]|} \Bigr|\leq \frac{C}{R^2}.
		$$ 
		Further basic calculus (and $|\tilde{f}_{\r}(x,y)|\approx |[x,y]|$) gives 
		$$
		\Bigl|D\frac{\tilde{f}_{\r}(x,y) } {|\tilde{f}_{\r}(x,y)|} \Bigr|\leq \frac{C}{R}\text{ and } \Bigl|D^2\frac{\tilde{f}_{\r}(x,y) } {|\tilde{f}_{\r}(x,y)|} \Bigr|\leq \frac{C}{R^2}\text{ for all }[x,y]\in B(0,R)\setminus [B(0,\tfrac{6}{8}R)\cup \bigcup_{i=1}^NO_i].
		$$ 
		Therefore
		$$
				|\phi_g(x,y)| \leq C,\ |D\phi_g(x,y)| \leq  \frac{C}{R},\ 
				 |D^2\phi_g(x,y)| \leq \frac{C}{R^2} 
		$$
		for all $[x,y]\in B(0,R)\setminus [B(0,\tfrac{6}{8}R)\cup \bigcup_{i=1}^NO_i]$. 
		Therefore
		$$
			|D^2g_{\r}| = |D^2(\R_g\phi_g)| \leq C\bigl(  |D^2\R_g|\cdot|\phi_g| + |D\R_g|\cdot|D\phi_g| + |\R_g|\cdot|D^2\phi_g|\bigr) \leq \frac{C}{R}
		$$
		%on $B(0,R)\setminus \bigcup_{i=1}^NO_i$ and integrating over $B(0,R)\setminus \bigcup_{i=1}^NO_i$ we get
		and by integrating 
		\begin{equation}\label{Anchorage}
		\int_{B(0,R)\setminus [B(0,\frac{6}{8}R)\cup \bigcup_{i=1}^NO_i]}|D^2g_{\r}| < CR.
		\end{equation}
		
		Now we continue with the case $[x,y] \in \bigcup_{i=1}^NO_i$. We work in each $O_i$ separately. Use $K_i$ to denote
		$$
			K_i = \max\bigl\{|D^2_sf|(t\tilde{\omega}_i); t\in[0,R]\bigr\}\leq 2L. 
		$$
		It still holds that $|\tilde{f}_{\r}(x,y)| \leq CR$ and $D|\tilde{f}_{\r}(x,y)| \leq C$ but the difference is that 
		$$
		D^2|\tilde{f}_{\r}(x,y)|\leq C + \frac{CK_i}{r_i}
		$$ 
		(see estimates in step~\ref{The Carcass} of Lemma~\ref{Reptiles}, most importantly the $D_{xx}$ term). Therefore
		$$
			|\R_g(x,y)| \leq  CR,\ |D\R_g(x,y)| \leq  C\text{ and } |D^2\R_g(x,y)| \leq  \frac{C}{R} + \frac{CK_i}{r_i} \text{ for all } [x,y]\in O_i
		$$
		Similarly as before basic calculus with $|f(x,y)|\approx |[x,y]|$ gives 
		$$
		\Bigl|D\frac{\tilde{f}_{\r}(x,y) } {|\tilde{f}_{\r}(x,y)|} \Bigr|\leq \frac{C}{R}\text{ and } \Bigl|D^2\frac{\tilde{f}_{\r}(x,y) } {|\tilde{f}_{\r}(x,y)|} \Bigr|\leq \frac{C}{R^2} + \frac{C}{Rr_i}K_i,
		$$ 
		where the constants $C$ depend on $d,M$ and $L$ etc., but not $R$. Therefore
		$$
		|\phi_g(x,y)| \leq C,\ |D\phi_g(x,y)| \leq  \frac{C}{R},\ 
		|D^2\phi_g(x,y)| \leq \frac{C}{R^2} + \frac{CK_i}{Rr_i} \text{ for }|[x,y]| \in O_i.
		$$
		Calculating as before
		$$
			|D^2g_{\r}| = |D^2(\R_g\phi_g)| \leq C\bigl(  |D^2\R_g|\cdot|\phi_g| + |D\R_g|\cdot|D\phi_g| + |\R_g|\cdot|D^2\phi_g|\bigr) \leq \frac{C}{R} + \frac{CK_i}{r_i}
		$$
		Integrating the above estimates over $O_i$ we get
		$$
		\int_{O_i} |D^2g_{\r}| \leq \frac{C}{R}Rr_i +  \frac{CK_i}{r_i}Rr_i \leq Cr_i + CLR \leq CR.
		$$
		Summing the above over $i\in\{1,2,\hdots,N\}$ and adding to \eqref{Anchorage} we get \eqref{Houston}.
	\end{proof}
	
	\begin{proof}[Proof of Theorem~\ref{Dan}]
		Let $A$ denote the finite set $A = \{a_1,a_2,\dots, a_I \}$ of vertices and $S$ denote the finite set $S = \{s_1,s_2,\dots, s_J \}$ of sides of triangles of the definition of polygonal domain. We know that our quadratic mappings $Q_j$ defined on triangles $T_j$ satisfy $\det DQ_j \geq d > 0$ and we can fix constants $L>0$ and $M>0$ so that $|DQ_j|\leq L$ and $|D^2Q_j|\leq M$ for all $j$. 
Thus we can apply Lemma~\ref{America} to a translation of $f$ in the image and preimage at each vertex $a_i\in A$. 
We find $\rho_0>0$ so that $B(a_i,\rho_0)$ are pairwise disjoint.
		
		%For each $a_i \in A$ and each $s\in S$ ending at $a_i$ we choose numbers $\rho_{s,a_i}$ as follows. 
		We find $N\in \en, N\geq 4$ such that (when we call the $\ell_s$ length of $s\in S$) we have 
		$$
		\frac{\max\{\ell_s: s\in S\}}{N} < \min\Bigl\{\rho_0, \frac{\min\{d,d^2\}}{2000 (M+1)(L+1)^4}\Bigr\}.
		$$ 
		Then we call $\rho_{s}=\tfrac{\ell_s}{2N}$ for each $s\in S$. 
		
		%We call $C_1$ the constant $C$ from \eqref{Houston}.  
		Having chosen a $\rho_s$ for every $s$ ending at $a_i$ we choose 
		$$
		R_i \leq \frac{1}{2} \min\Bigl\{\frac{\epsilon}{I(M+1)}, \rho_s\Bigr\} < 
		\frac{1}{4} \min\Bigl\{\rho_0, \frac{\min\{d,d^2\}}{1000 (M+1)(L+1)^4},\frac{1}{8}\frac{L}{M+1}\Bigr\}.
		$$
		
		For each $s_j\in S$ we choose an $r_{s_j}>0$ as follows. We require that $r_{s_j}$ is smaller than the corresponding $r_0(s_j)$, the number from Lemma~\ref{Reptiles}. Further we require that 
		$$
		r_{s_j}\leq \min\Bigl\{ \frac{d^2R_i}{432L^4},\frac{R_id}{1200 L^2}, \frac{\rho_{s_j}^2}{2(L+1)}, \frac{R_i}{2}\tan\frac{\omega^*_i}{3}, 
		\frac{1}{J}\frac{\epsilon}{(M+1) \ell_s }\Bigr\}
		$$ 
		for both endpoints $a_i = a_{s_j,1}$ and $a_{i} = a_{s_j,2}$. %Finally we require $r_{s_j}\leq $. 
		For each $a_i$ we call $\r_{i} = (R_i, \rho_{1},\dots \rho_{n_i}, r_{s_{j_{i,1}}}, \dots,r_{s_{j_{i,n_i}}}  )$.
		
		Having made the above choices, we have satisfied the hypothesis of Lemma~\ref{Reptiles} (up to appropriate rotations and translations) for each side $s_j \in S$ by the choice of $r = r_{s_j}$ and hence we can construct a smooth $g = g_s$ on a small rectangular neighborhood of each side $s$. Similarly we satisfy the hypothesis of Lemma~\ref{America} and because of the same choice of parameters the smooth map $g = g_{a_i}$ is equal to $g_s$ ($a_i \in s$) as soon as the argument is $R_i$ distant from $a_i$. Both of the maps equal the original homeomorphism $f$ as soon as we are $R_i$ distant from $a_i$ and $r_{s}$ distant from $s$, which is $\mathcal{C}^{\infty}$ smooth on that set. Therefore the map
		$$
			g(x,y) = \begin{cases}
				g_{a_i}(x,y) \quad &|[x,y] - a_i| \leq R_i\\
				g_{s}(x,y) \quad & \dist([x,y], s)\leq r_{s} \text{ and } |[x,y]- a_i| \geq R_i\\
				f(x,y) \quad & \text{otherwise}
			\end{cases}
		$$
		is a $\mathcal{C}^{\infty}$-diffeomorphism. Notice that the balls $B(a_i, R_i)$ are pairwise disjoint and so the definition is correct. 
		
		For each $s_j\in S$ we call $O_j$ the $2r_{s_j}$-wide rectangular neighborhood of the line $s_j$ as in Lemma~\ref{Reptiles}. Now we use $g(x,y) = f(x,y)$ for all $[x,y] \notin \bigcup_iB(a_i, R_i) \cup \bigcup_j O_j$, which implies that
		$$
			\int_{\Omega}|D^2f - D^2g| \leq \sum_{i=1}^I\int_{B(a_i,R_i)}(|D^2f| + |D^2g|) + \sum_{j=1}^J \int_{O_j}(|D^2f| + |D^2g|).
		$$
		We estimate by summing \eqref{Houston} over $a_i\in  A$ (recall $R_i < \tfrac{\epsilon}{I (M+1)}$) and using $|D^2f| \leq M$ to obtain
		$$
		\sum_{i=1}^I\int_{B(a_i,R_i)}(|D^2f|+|D^2g|)\leq \sum_{i=1}^I(M\pi R_i^2+C R_i)\leq C\epsilon.
		$$
		Finally we sum \eqref{Iguana} over $s_j\in S$ (recall $r_{s_j}\leq \frac{1}{J}\tfrac{\epsilon}{(M+1) \ell_s }$) 
		and we get using Theorem \ref{Dan}
		$$
				\sum_{j=1}^J \int_{O_j}(|D^2f| +|D^2g|) \leq C\sum_{j=1}^J (l_{s_j}r_{s_j} M+|D^2_s f|(s))
				\leq C(\delta + \epsilon).
		$$
		
		By \eqref{MuskTurtle} and the estimates in Lemma~\ref{America} Step~\ref{The Carcass} it is immediately obvious that 
		$$
		\|f-g\|_{\infty} < \max_i CR_i + \max_s 8Lr_s, 
		$$
		which is as small as we like.
	\end{proof}

%STANDA - PROMYSLI REMARK
	%\begin{remark}[Recovering \cite{MP} in full]
	%	From \eqref{MuskTurtle} and the estimates in Lemma~\ref{America}, Step~\ref{The Carcass} we have $|Dg| \leq CL_s$ near each side $s$ and $|Dg|\leq CL_a$ near each $a\in A$. Because the set where $f\neq g$ can be made as small as we like we have can construct a sequence $Dg_k \to Df$ in  $X$ any function space $X$ satisfying the absolute continuity of the norm, especially in every $L^p$, $1\leq p<\infty$. Similarly, it is not difficult to calculate the inverse of the derivative in Lemma~\ref{Reptiles} and see that $|Dg^{-1}| \leq C|Df^{-1}|$ near sides $s$. Similarly near vertices we have that $|Dg|$ stays bounded by $C|Df|$ and $J_g\geq C$ independent of $R_i$ and therefore also $|Dg^{-1}| \leq |Df^{-1}|$ near vertices. Altogether this implies that we can construct a sequence of $g_k$ satisfying $\|Dg_k - Df\|_X + \|Dg^{-1} - Df^{-1}\|_{Y} \to 0$ for any pair $X,Y$ of function spaces with the absolute continuity of the norm. Especially we can construct $g_k$ converging in $W^{1,p}$ and its inverse converging in $W^{1,q}$ for any $1\leq p,q<\infty$ For the same reasoning also $\int_{\Omega}J_{g_k}^{-p} - J_{f}^{-p} \to 0$. Although the detailed calculations are not complicated we have not pursued this avenue.
	%\end{remark}

\section{Piecewise quadratic approximation on good squares - Proof of Theorem \ref{Standa}}

	In this section we first show that the quadratic polynomials constructed in subsection \ref{FEMapr} approximate our homeomorphism $f$ well on some good squares and then we show Theorem \ref{Standa}. As noted in subsection \ref{FEMapr} we know that the two quadratic polynomials on adjacent triangles have the same values on $T_1\cap T_2$ but the derivatives (in the orthogonal direction) are not necessarily the same. 
	The key observation \eqref{close} $(ii)$ below shows that they do not differ too much. 
	
	\prt{Theorem}
	\begin{proclaim}\label{obluda}
		Let $T_1$ be a triangle with vertices $v_1=[0,0]$, $v_2=[r,0]$ and $v_3=[0,r]$ for some $r>0$. 
		Let $T_2$ be an adjacent triangle, i.e. either with vertices $\{[r,0], [0,r], [r,r]\}$ or $\{[0,0], [r,0], [r,-r]\}$ or 
		$\{[0,0], [0,r], [-r,r]\}$. 
		Let us assume that we have a homeomorphism $f\in W^{2,1}(Q([0,0],2r),\er^2)$. Let $0<\delta<1$ and assume that 
		\eqn{start}
		$$
		J_f(0,0)>\delta,\ \|Df(0,0)\|<\frac{1}{\delta}%,\ \|D^2f(0,0)\|<\frac{1}{\delta}
		$$
		and for $\epsilon>0$ we have
		\eqn{key}
		$$
		|f(z)-f(0,0)-Df(0,0)z|<\epsilon|z|\text{ for }z\in Q([0,0],3r),
		$$
		\eqn{key1}
		$$
		\oint_{Q([0,0],3r)}|Df(z)-Df(0,0)|\; dz<\epsilon
		$$
		and 
		\eqn{key2}
		$$
		\oint_{Q([0,0],3r)}|D^2f(z)-D^2f(0,0)|\; dz<\epsilon. 
		$$
		
		Then there are absolute constant $C_0>0$ and quadratic mappings $A_1,\ A_2:Q([0,0],2r)\to\er^2$ so that 
		\eqn{close}
		$$
		\begin{aligned}
		&(i)\ D^2A_i\text{ is constant and }|D^2A_i(0,0)-D^2f(0,0)|<C\epsilon,\\
		&(ii)\ |DA_1(z)-DA_2(z)|<\epsilon r\text{ for every }z\in T_1\cap T_2,\\
		&(iii)\ A=\begin{cases}A_1\text{ on }T_1\\
		A_2\text{ on }T_2\\
		\end{cases}
		\text{is homeomorphism with }\det A>\frac{\delta}{2}\text{ on }T_1\cup T_2,\text{ if }\epsilon<C_0\delta^2,\\
		&(iv)\ |f(z)-A_1(z)|< C_1 r\epsilon\text{ for every }z\in T_1.
		\end{aligned}
		$$
		Further the map $A_1$ is independent of the choice of $T_2$.
	\end{proclaim}
	
	\begin{proof}
		We define $A_1$ on $T_1$ by \eqref{FEM} and $A_2$ on $T_2$ by a similar procedure, i.e. values of $A_2$ in corners of $T_2$ are determined by the average values of $f$ nearby and a derivative at each vertex along a given side is determined by the average of the corresponding derivative of $f$. We just make sure that on the side $T_1\cap T_2$ both $A_1$ and $A_2$ use the same vertex for the definition of derivative along that side. 
		In fact we can divide the whole $\er^2$ into squares of sidelength $r$, divide them into two triangles (by segment in direction $[-1,1]$) and assign to each vertex a direction along one of the sides (where we define the derivative of the approximating quadratic polynomial) so that it matches the definition for $T_1$ and $T_2$ above (see Fig. \ref{Fig:Triangulization}). 
		
		{\bf Part $(i)$: } We have
		\eqn{1}
		$$
		\begin{aligned}
		A_1([r,0])-A_1([0,0])-rD_xA_1([0,0])
		&=\int_0^r D_{x}A_1([t,0])\; dt-\int_0^r D_{x}A_1([0,0])\; dt\\
		&=\int_0^r (r-a) D_{xx}A_1([a,0])\; da.
		\end{aligned}
		$$
		In preparation for \eqref{2} we define a function
		$$
			w(z_1,z_2) = \int_{ \max \big\{ - \sqrt{r^2/100 - z_2^2}, z_1-r \big\} }^{ \min\big\{ \sqrt{r^2/100 - z_2^2}, z_1 \big\} } \frac{r+s-z_1}{\mathcal{L}^2(B(0,\tfrac{r}{10}))} \; ds
		$$
		on the set $([0,r]\times\{0\})+B(0,\frac{r}{10})$. Because
		$$
		\min\{ \sqrt{r^2/100 - y^2}, z_1 \} - \max \{ - \sqrt{r^2/100 - y^2}, z_1-r \} \leq \frac{r}{5}
		$$
		on which $0 \leq |r+s-z_1| \leq Cr$ we have a geometric constant $C$ such that
		\begin{equation}\label{defw1}
			0\leq w(z) \leq C.
		\end{equation}
		By the definition of $A_1$ (see \eqref{FEM}), the ACL condition and straight forward Fubini theorem \eqref{1} is equal to
		\begin{equation}\label{2}
			\begin{aligned}
				\oint_{B(0,\frac{r}{10})}&\bigl[f([r,0]+z)-f([0,0]+z)-rD_xf([0,0]+z)\bigr]\; dz=\\
				&= \oint_{B(0,\frac{r}{10})}\int_0^r (r-a) D_{xx}f([a,0]+z)\; da\; dz\\
				& = \int_{-\tfrac{r}{10}}^{\tfrac{r}{10}}\int_{-\sqrt{\frac{r^2}{100} - z_2^2}}^{\sqrt{\frac{r^2}{100} - z_2^2}}
				\int_{z_1}^{z_1+r} D_{xx}f(t,z_2)\frac{r+z_1-t}{\mathcal{L}^2(B(0,r/10))} \; dt \; dz_1 \; dz_2\\
				& = \int_{-\tfrac{r}{10}}^{\tfrac{r}{10}}\int_{-\sqrt{\frac{r^2}{100} - z_2^2}}^{r+\sqrt{\frac{r^2}{100} - z_2^2}}D_{xx}f(t,y)\int_{ \max \{ - \sqrt{r^2/100 - z_2^2}, t-r \} }^{ \min\{ \sqrt{r^2/100 - z_2^2}, t \} } \frac{r+s-t}{\mathcal{L}^2(B(0,r/10))} \; ds \; dt \; dz_2\\
				&=\int_{([0,r]\times\{0\})+B(0,\frac{r}{10})}w(z)D_{xx}f(z)\; dz.
			\end{aligned}
		\end{equation}	
		Further
		\eqn{defw2}
		$$
		\int_{([0,r]\times\{0\})+B(0,\frac{r}{10})}w(z)\; dz=\int_0^r(r-a)\; da=\frac{r^2}{2}, 
		$$ 
		can be easily deduced by considering the special case 
		$D_{xx}f\equiv 1$. 
		%It follows that
		%$$
		%\int_{([0,r]\times\{0\})+B(0,\frac{r}{10})}w(z)\; dz\geq C r^2. 
		%$$
		Since $D_{xx}A_1$ is constant we can use equality of \eqref{1} and \eqref{2} together with \eqref{defw2}, \eqref{defw1}, and \eqref{key2} to obtain
		\eqn{ffff}
		$$
		\begin{aligned}
		\Bigl|D_{xx}A_1([0,0])-D_{xx}f([0,0])\Bigr|&=\frac{1}{\int_0^r(r-a)\; da}\Bigl|\int_0^r (r-a) \bigl(D_{xx}A_1([a,0])-D_{xx}f([0,0])\bigr)\; da\Bigr|\\
		&=\frac{2}{r^2}\Bigl|\int_{([0,r]\times\{0\})+B(0,\frac{r}{10})}w(z)\bigl(D_{xx}f(z)-D_{xx}f([0,0])\bigr)\; dz\Bigr|\\
		&\leq\frac{C}{r^2}\int_{([0,r]\times\{0\})+B(0,\frac{r}{10})}\bigl|D_{xx}f(z)-D_{xx}f([0,0])\bigr|\; dz\\
		&<C\epsilon.
		\end{aligned}
		$$
		
		%\eqn{ffff}
		%$$
		%\bigl|D_{xx}A_1([0,0])-D_{xx}f([0,0])\bigr|<C\epsilon. 
		%$$
		
		By similar reasoning on side $[0,0]$, $[0,r]$ with the help of $D_yA_1([0,r])$ we obtain that
		\eqn{fff}
		$$
		\bigl|D_{yy}A_1([0,0])-D_{yy}f([0,0])\bigr|<C\epsilon. 
		$$
		It remains to consider $D_{xy}$. We use $A_1([r,0])$, $A_1([0,r])$, $(-D_x+D_y)A_1([r,0])$ and similar formulas for the one dimensional function $h(t)=f([r,0]+t[-1,1])$. By the chain rule 
		$$
		\begin{aligned}
		h'(t)&=-D_xf([r-t,t])+D_yf([r-t,t])\text{ and }\\
		h''(t)&=D_{xx}f([r-t,t])-D_{yx}f([r-t,t])-D_{xy}f([r-t,t])+D_{yy}f([r-t,t]). 
		\end{aligned}
		$$
		Now $D_{xy}f=D_{yx}f$ as distributional derivatives are always interchangeable. An analogy of the inequality \eqref{ffff} above together with the fact that we already know \eqref{ffff} and \eqref{fff} for $D_{xx}$ and $D_{yy}$ implies that
		$$
		\bigl|D_{xy}A_1([0,0])-D_{xy}f([0,0])\bigr|<C\epsilon. 
		$$
		The proof for $|D^2A_2(0,0)-D^2f(0,0)|<C\epsilon$ on $T_2$ is similar.  Therefore \eqref{close} $i)$ has been proved.
		
		{\bf Part $(ii)$:} We know that $T_1$ has vertices $[0,0]$, $[r,0]$ and $[0,r]$. We assume that $T_2$ has vertices $[0,0]$, $[r,0]$ and $[r,-r]$ as other cases can be treated similarly. Our $A_1$ on $T_1$ is defined by \eqref{FEM} and $A_2$ on $T_2$ is defined using (average) values at vertices and derivatives along sides
		$$
		\begin{aligned}
		D_xA_2([0,0])&=\oint_{B([0,0],\frac{r}{10})}D_xf,\ 
		-D_yA_2([r,0])=\oint_{B([r,0],\frac{r}{10})}-D_yf\\
		&\text{ and }
		(-D_{x}+D_y)A_2([r,-r])=\oint_{B([r,-r],\frac{r}{10})}(-D_{x}+D_y)f.\\
		\end{aligned}
		$$ 
		In this way we have $D_xA_2([0,0])=D_xA_1([0,0])$ as they are defined by the same expression.

		For any $x\in[0,r]$ we have with the help of $D_xA_2([0,0])=D_xA_1([0,0])$, $D^2A_i$ is constant and \eqref{close} $(i)$ (which was proved in part (i))
		\eqn{alongline}
		$$
		\begin{aligned}
		\bigl|D_x (A_1-A_2)([x,0])\bigr|
		&=\Bigl|D_x (A_1-A_2)([0,0])+\int_0^x D_{xx}(A_1-A_2)([a,0])\; da\Bigr|\\
		&\leq r\Bigl(|D_{xx}A_1([0,0])-D_{xx}f([0,0])|+|D_{xx}A_2([0,0])-D_{xx}f([0,0])|\Bigr)\\
		&\leq Cr\epsilon.\\
		\end{aligned}
		$$
		
		It remains to show that $D_y (A_1-A_2)$ along $T_1\cap T_2$ is small. 
		%We know that $|D_x(A_1-A_2)([r,0])|\leq Cr\epsilon$ and 
		By the definition of $A_i$
		$$
		D_yA_2([r,0])=\oint_{B([r,0],\frac{r}{10})}D_yf\text{ and }(-D_x+D_y)A_1([r,0])=\oint_{B([r,0],\frac{r}{10})}(-D_x+D_y)f. 
		$$
		%giving
		%$$
		%	\begin{aligned}
		%		D_y(A_1-A_2)([r,0]) 
			%	=& D_xA_1(r,0) + (-D_x+D_y)A_1(r,0) - D_yA_2(r,0)\\
			%	 =& D_x A_1([r,0])-\oint_{B([r,0],\frac{r}{10})}D_xf
			%\end{aligned}
		%$$
		and hence
		$$
		\begin{aligned}
		\bigl|D_y(A_1-A_2)([r,0])\bigr|\leq & \Bigl|D_x A_1([r,0])-\oint_{B([r,0],\frac{r}{10})}D_xf\Bigr|\\
		\leq & \Bigl|D_x A_1([r,0])-D_xA_1([0,0])-rD_{xx}A_1([0,0])\Bigr|+\\
		&+r\Bigl|D_{xx}f([0,0])-D_{xx}A_1([0,0])\Bigr|+\\
		&+\Bigl|\oint_{B([r,0],\frac{r}{10})}D_xf-\oint_{B([0,0],\frac{r}{10})}D_xf-rD_{xx}f([0,0])\Bigr|.\\
		\end{aligned}
		$$ 
		The first expression on the righthand is zero by the fundamental theorem of calculus for $D_x$ as $D_{xx}A_1$ is constant and the second one is bounded by $Cr\epsilon$ by $(i)$. 
		It remains to estimate the last term using ACL condition, fundamental theorem of calculus and \eqref{key2}
		$$
		\begin{aligned}
		\Bigl|\oint_{B([r,0],\frac{r}{10})}&D_xf-\oint_{B([0,0],\frac{r}{10})}D_xf-rD_{xx}f([0,0])\Bigr|=\\
		&=\Bigl|\oint_{B([0,0],\frac{r}{10})}\int_0^r \bigr[D_{xx}f([t,0]+z)-D_{xx}f([0,0])\bigl]\; dt\; dz\Bigr|\\
		&\leq Cr \oint_{Q([0,0],3r)}\bigr|D_{xx}f(z)-D_{xx}f([0,0])\bigl|\; dz\\
		&\leq Cr \epsilon. 
		\end{aligned}
		$$ 
		It follows that $|D(A_1-A_2)([r,0])|\leq Cr\epsilon$. 
		
		Similarly to \eqref{alongline} we obtain for $a\in[0,r]$
		\eqn{a}
		$$
		\begin{aligned}
		\bigl|D_y (A_1-A_2)([0,a])-D_y (A_1-A_2)([0,0])\bigr|
		&=\Bigl|\int_0^a D_{yy}(A_1-A_2)([0,t])\; dt\Bigr|\\
		&\leq Cr\epsilon\\
		\end{aligned}
		$$
		and with the help of $|D(A_1-A_2)([r,0])|\leq Cr\epsilon$ also for $t\in[0,r]$
		\eqn{aa}
		$$
		\begin{aligned}
		\bigl|(-D_x+D_y) (A_1-A_2)([t,r-t])\bigr|
		&\leq \bigl|(-D_x+D_y) (A_1-A_2)([r,0])\bigr|+Cr\epsilon\\
		&\leq Cr\epsilon.\\
		\end{aligned}
		$$
		The integral of the derivative along the closed curve is zero and thus 
		\eqn{triangle}
		$$
		0=\int_0^r D_x(A_1-A_2)([a,0])\; da+ \int_0^r (-D_x+D_y)(A_1-A_2)([r-a,a])\; da+\int_0^r (-D_y)([0,r-a])\; da. 
		$$
		It follows using \eqref{a}, \eqref{alongline} and \eqref{aa} that
		$$
		\begin{aligned}
		r\bigl|D_y (A_1-A_2)([0,0])\bigr|\leq& \Bigl|\int_0^r D_y(A_1-A_2)([0,a])\; da\Bigr|+Cr^2\epsilon\\
		\leq& \Bigl|\int_0^r D_x(A_1-A_2)([a,0])\; da\Bigr|+\\
		&+\Bigl|\int_0^r (-D_x+D_y)(A_1-A_2)([r-a,a])\; da\Bigr|+Cr^2\epsilon\\
		\leq & Cr^2\epsilon.\\
		\end{aligned}
		$$
		We have just shown that $|D (A_1-A_2)([0,0])|\leq Cr\epsilon$. 
		Similar reasoning can estimate the derivative of $A_1-A_2$ at other points $[x,0]\in T_1\cap T_2$, we just use a triangle with vertices $[x,0]$, $[r,0]$ and $[x,r-x]$ in an analogy of \eqref{triangle}. 
		
		{\bf Part $(iii)$: } 
		We know by the definition of $T_1$ and $T_2$ that $A_1=A_2$ on $T_1\cap T_2$ (see subsection \ref{FEMapr}) and hence $A$ is continuous. 
		It remains to show that $\det A>\frac{\delta}{2}$ and that $A$ is $1-1$ on $T_1\cup T_2$. 
		
		The definition \eqref{FEM} of $A_1$ in fact means that to determine the coefficients of quadratic function $A_1$ we solve the equation $Ma=c$, where $c$ determines the averaged values of $f$ and $Df$ along the sides (in vertices), $a$ is the vector of coefficients of $A_1$ and
		\eqn{matrix}
		$$
		M=\begin{pmatrix}
		1&0&0&0&0&0\\ 
		1&r&0&r^2&0&0\\ 
		1&0&r&0&0&r^2\\ 
		0&1&0&0&0&0\\ 
		0&-1&1&-2r&2r&0\\
		0&0&-1&0&0&-2r\\ 
		\end{pmatrix}
		\text{ and }
		M^{-1}=\begin{pmatrix}
		1&0&0&0&0&0\\ 
		0&0&0&1&0&0\\ 
		-\frac{2}{r}&0&\frac{2}{r}&0&0&1\\ 
		-\frac{1}{r^2}&\frac{1}{r^2}&0&-\frac{1}{r}&0&0\\ 
		0&\frac{1}{r^2}&-\frac{1}{r^2}&-\frac{1}{2r}&\frac{1}{2r}&-\frac{1}{2r}\\ 
		\frac{1}{r^2}&0&-\frac{1}{r^2}&0&0&-\frac{1}{r}\\ 
		\end{pmatrix}.
		$$ 
		If fact we solve this for $a=[a_1,a_2,a_3,a_4,a_5,a_6]$ where $c$ is determined by the first coordinate function of $f$ and we solve it for 
		$a=[b_1,b_2,b_3,b_4,b_5,b_6]$ where $c$ is determined by the second coordinate function of $f$. 
		
		We know that \eqref{key} and \eqref{key1} hold for $f$ and thus we can divide it into linear part $L(z)=f(0,0)+Df(0,0)z$ plus 
		$E:=f-L$ and $|f-L|\leq \epsilon |z|$ on $Q([0,0],2r)$. 
		Thus we can divide the right-hand side $c$ into two terms $c_L+c_E$, $c_L$ corresponding to the linear part $L$ of $f$ and $c_{E}$ corresponding to the remaining $(f-L)$-term. Our equation is linear and the unique solution to the linear part $L$ is the same linear function (with determinant $>\delta$). Let us estimate the derivative of $E=f-L$. 
		From $|f-L|\leq \epsilon r$ on $Q([0,0],r)$ (see \eqref{key}), definition of $A_1$ \eqref{FEM} and $\mir2(B(v_i,\frac{r}{10}))\geq C r^2$ we see that 
		$$
		|(c_E)_1|\leq C\epsilon r,\ |(c_E)_2|\leq C\epsilon r\text{ and }|(c_E)_3|\leq C\epsilon r. 
		$$
		Similarly we obtain from \eqref{key1} and \eqref{FEM} that
		$$
		|(c_E)_4|\leq C\epsilon ,\ |(c_E)_5|\leq C\epsilon \text{ and }|(c_E)_6|\leq C\epsilon . 
		$$
		Given the form of $M^{-1}$ \eqref{matrix} it is now easy to see that the solution $a_E:=M^{-1}c_E$ satisfies
		$$
		|(a_E)_i|\leq C\epsilon\text{ for }i=1,2,3\text{ and }|(a_E)_i|\leq C\frac{\epsilon}{r}\text{ for }i=4,5,6. 
		$$
		Now for every $z\in Q([0,0],r)$ we have 
		\eqn{defE}
		$$
		|DE(z)|\leq C(a_2+a_3+a_4r+a_5r+a_6r)\leq C\epsilon. 
		$$
		Thus we have a quadratic function $A_1=L+E$, where (see \eqref{start}) $\det DL>\delta$, $|DL|<\frac{1}{\delta}$ and $|DE|\leq C\epsilon$. 
		Assume that $0<\epsilon<C_0 \delta^2$. 
		Now $\det D(L+E)$ contains $\det DL$ plus other terms whose sum is smaller than 
		$$
		C |DE|(|DL|+|DE|)\leq C\epsilon\bigl(\frac{1}{\delta}+\epsilon\bigr)\leq C C_0 \delta.  
		$$
		Now it is easy to see that we can choose an absolute constant $C_0$ so that 
		$$
		\det DA_1(z)=\det(L+E)(z)>\frac{\delta}{2}\text{ for every }z\in Q([0,0],2r). 
		$$

		Now we prove that $A_1$ is $1-1$ on $T_1$. Let us denote by $\lambda_1, \lambda_2$ the eigenvalues of the matrix $Df(0,0)$. From \eqref{start} we know that 
		$$
		\lambda_1\lambda_2>\delta\text{ and }\max\{|\lambda_1|,|\lambda_2|\}<\frac{1}{\delta}\text{ and hence }\min\{|\lambda_1||,|\lambda_2|\}>\delta^2. 
		$$
		It follows that the linear function $L(z)=f(0,0)+Df(0,0)z$ satisfies
		\eqn{analogy}
		$$
		|L(z)-L(w)|\geq \delta^2|z-w|\text{ for every }z,w\in Q([0,0],2r). 
		$$
		From $A_1=L+E$, $|DE|\leq C\epsilon$ and $\epsilon<C_0\delta^2$ we obtain for $z,w\in Q([0,0],2r)$
		$$
		|A_1(z)-A_1(w)|\geq |L(z)-L(w)|-|E(z)-E(w)|\geq \delta^2 |z-w|-C\epsilon|z-w|\geq \frac{\delta^2}{2}|z-w|
		$$
		once $C_0$ is chosen sufficiently small. It follows that $A_1$ is $1-1$ on $T_1$ and similarly we can show that $A_2$ is $1-1$ on $T_2$. 
		
		It remains to show that we cannot have $A_1(z)=A_2(w)$ for $z\in T_1$ and $w\in T_2$. We find $v\in T_1\cap T_2$ on the line segment 
		between $z$ and $w$. %Without loss of generality we can assume that $|z-v|\geq \tfrac{|z-w|}{2}$. 
		We know that $A_1=L+E_1$ and $A_2=L+E_2$ with $|DE_1|\leq C\epsilon$ and $|DE_2|\leq C\epsilon$. Analogously as above we use $A_1(v)=A_2(v)$ to obtain
		\eqn{clear}
		$$
		\begin{aligned}
		|A_1(z)-A_2(w)|&\geq |L(z)-L(w)|-|E_1(z)-E_1(v)|-|E_2(v)-E_2(w)|\\
		&\geq \delta^2 |z-w|-C\epsilon|z-w|-C\epsilon|z-w|\\
		&\geq \frac{\delta^2}{2}|z-w|\\
		\end{aligned}
		$$
		once $C_0$ is chosen sufficiently small. Hence $A$ is $1-1$ and thus a homeomorphism on $T_1\cup T_2$. 
		
		Moreover, we can divide $Q([0,0],2r)$ into $32$ triangles, define quadratic functions $A$ on each of them by (translated and rotated version of) \eqref{FEM}. Similarly to \eqref{clear} we can even show that $A$ is a homeomorphism on the whole $Q([0,0],2r)$ (once $C_0$ is sufficiently small but fixed absolute constant) since we subtract only bounded number of terms $C\epsilon|z-w|$ in analogy of \eqref{clear}. 
		
		{\bf Part $(iv)$: } We know that $A_1=L+E$ where $L(z)=f(0,0)+Df(0,0)z$ and $|DE|\leq C\epsilon$ (see \eqref{defE}). 
		It follows using \eqref{key} and $|DE|\leq C\epsilon$ that for $z\in T_1$ %and $E(0,0)=0$ NOT TRUE 
		\eqn{t}
		$$
		\begin{aligned}
		|f(z)-A_1(z)|&\leq |f(z)-L(z)|+|A_1(z)-L(z)|\\
		&\leq \epsilon|z|+|E(z)-E(0,0)|+|E(0,0)| \\
		&\leq \epsilon r+C\epsilon r+|E(0,0)|.\\
		\end{aligned}
		$$
		Clearly $E(0,0)=A_1(0,0)-f(0,0)=[a_1,b_1]-f(0,0)$ and the coefficients $[a_1,b_1]$ are given by (see \eqref{FEM})
		$$
		[a_1,b_1]=\oint_{B([0,0],\frac{r}{10})}f(z)\; dz.
		$$
		Hence we obtain using \eqref{key}
		\eqn{tt}
		$$
		\begin{aligned}
		|E(0,0)|&=\Bigl|\oint_{B([0,0],\frac{r}{10})}\bigl(f(z)-L(z)\bigr)\; dz\Bigr|\\
		&=\Bigl|\oint_{B([0,0],\frac{r}{10})}\bigl(f(z)-f(0,0)-Df(0,0)z\bigr)\; dz\Bigr|\\
		&\leq \epsilon r.
		\end{aligned}
		$$
		Our conclusion for $C_1:=2+C$ follows from \eqref{t} and \eqref{tt}. 
	\end{proof}

	\begin{proof}[Proof of Theorem \ref{Standa}]
		Let us recall that we have a $W^{2,1}$ homeomorphism so that $J_f>0$ a.e. We fix $\eta>0$ so that the set
		$$
		\Omega_{\eta}:=\{z\in\Omega:\ \dist(z,\partial\Omega)>\eta\}\text{ satisfies }\mir2(\Omega\setminus \Omega_{\eta})<\frac{\nu}{2}. 
		$$ 
		Since $J_f>0$ a.e. we can fix $\delta>0$ small enough so that 
		$$
		\Omega_{\delta}:=\Bigl\{z\in\Omega:\ J_f(z)>\delta,\ \|Df(z)\|<\frac{1}{\delta}\Bigr\}\text{ satisfies }
		\mir2(\Omega\setminus \Omega_{\delta})<\frac{\nu}{4}.
		$$
		
		We know that $f$ is differentiable a.e. and that a.e. point is a Lebesgue point for both $Df$ and $D^2f$. It follows that for a.e. $z\in\Omega$ we have 
		$$
		\lim_{w\to z}\frac{|f(w)-f(z)-Df(z)(w-z)|}{|w-z|}=0,
		$$
		$$
		\lim_{r\to 0+}\oint_{Q(z,r)}|Df(w)-Df(z)|\; dw=0 \text{ and }\lim_{r\to 0+}\oint_{Q(z,r)}|D^2f(w)-D^2f(z)|\; dw=0. 
		$$
		We fix $0<\epsilon<\min\{C_0\delta^2,\eta,\tfrac{\delta^2}{8},\frac{\delta^2}{4C_1}\}$, where $C_0$ and $C_1$ are constants from Theorem \ref{obluda} $(iii)$ and $(iv)$. 
		From previous limits we know that for a.e. $z$ there is $r_z>0$ so that for every $0<r\leq r_z$ we have
		\eqn{klic}
		$$
		\bigl|f(w)-f(z)-Df(z)(w-z)\bigl|<\epsilon|z-w|\text{ for }w\in Q(z,3r),
		$$
		\eqn{klic2}
		$$
		\oint_{Q(z,3r)}|Df(w)-Df(z)|\; dw<\epsilon
		\text{ and }
		\oint_{Q(z,3r)}|D^2f(w)-D^2f(z)|\; dw<\epsilon. 
		$$
		Now we fix $0<r_0<\frac{\eta}{100}$ small enough so that the good set
		$$
		G:=\{z\in \Omega_{\delta}:\ r_z>r_0\}\text{ satisfies }\mir2(\Omega\setminus G)<\frac{\nu}{2}.
		$$ 
		
		Now we would like to cover $\Omega_{\eta}$ by squares of sidelength $2r_0$ so that most corners of those squares belong to $G$. 
		That is for $z_0\in Q(0,r)$ we consider
		$$
		\mathcal{Q}_{z_0}:=\Bigl\{Q(z_0+2kr_0,r_0):\ k\in\zet^2,\ Q(z_0+kr_0,r_0)\cap \Omega_{\eta}\neq\emptyset \Bigr\}.
		$$ 
		Since $\mir2(\Omega\setminus G)<\frac{\nu}{2}$ we can find and fix $z_0$ so that the number of good vertices (with $z_0+2kr_0\in G$) is bigger than the average and we have that 
		\eqn{ahoj}
		$$
		\mathcal{Q}:=\Bigl\{Q(z_0+2kr_0,r_0)\in \mathcal{Q}_{z_0}:\ z_0+2kr_0\in G \Bigr\}\text{ satisfies }
		\mir2\Bigl(\bigcup_{Q\in \mathcal{Q}_{z_0}\setminus\mathcal{Q}}Q\Bigr)<\frac{\nu}{2}. 
		$$
		Now we choose a Whitney type covering of $\Omega\setminus \bigcup_{Q\in \mathcal{Q}_{z_0}}Q$ and our set of squares $\{Q_i\}_{i=1}^{\infty}$ for the statement consists of 
		$$
		\text{ squares in }\mathcal{Q}_{z_0}\setminus \mathcal{Q}\text{ together with all cubes covering } \Omega\setminus \bigcup_{Q\in \mathcal{Q}_{z_0}}Q. 
		$$ 
		It is clear that these squares are locally finite and \eqref{ahoj} and $\mir2(\Omega\setminus \Omega_{\eta})<\frac{\nu}{2}$ imply that
		$$
		\mir2\Bigl(\bigcup_{i=1}^{\infty}Q_i\Bigr)<\nu. 
		$$
		
		It remains to define an approximation of $f$ on $\bigcup_{Q\in \mathcal{Q}}Q$. We first divide each such $Q$ into two triangles $T_Q, \tilde{T}_Q$ by joining the lower-right corner with upper-left corner. We denote 
		$$
		\mathcal{T}:=\bigcup_{Q\in \mathcal{Q}}\{T_Q,\tilde{T}_Q\}. 
		$$ 
		As a first step we use Theorem \ref{obluda} for each $T\in \mathcal{T}$ to obtain a piecewise quadratic approximation $A_T$ there. The assumption \eqref{start} is verified by the definition of $\Omega_{\delta}$ and $G$ above (recall that corners of $Q\in\mathcal{Q}$ belong to $G$) and assumptions \eqref{key}, \eqref{key1} and \eqref{key2} are verified by \eqref{klic} and \eqref{klic2}. We define
		$$
		A(z)=A_T(z)\text{ for }z\in T\text{ and }T\in\mathcal{T}. 
		$$ 
		We know that $A$ is a homeomorphism on each $T$, $T\in\mathcal{T}$, by Theorem \ref{obluda} $(iii)$ and moreover it is a homeomorphism on each $Q(z_T,2r)\cap \bigcup_{Q\in\mathcal{Q}} Q$, where $z_T$ is the corresponding vertex of $T\in\mathcal{T}$, as we have discussed at the end of proof of Theorem \ref{obluda} $(iii)$. 
		
		We claim that it is a homeomorphism on the whole $\bigcup_{Q\in\mathcal{Q}} Q$. Assume for contrary that $A(z)=A(w)$ for some 
		$z,w\in \bigcup_{Q\in\mathcal{Q}} Q$, $z\neq w$. We find $z_0$ a vertex of some triangle $T\in\mathcal{T}$ so that $z_0\in G$, $z\in T$ and 
		\eqref{klic} holds for $z_0$. Since $A$ is a homeomorphism on $B(z_0,2r)\cap \bigcup_{Q\in\mathcal{Q}} Q$ we obtain that $w\notin B(z_0,2r)$. 
		From Theorem~\ref{obluda} $(iv)$ and $A(z)=A(w)$ we obtain 
		\eqn{blizko}
		$$
		|f(z)-f(w)|\leq |f(z)-A(z)|+|A(w)-f(w)|\leq 2C_1r\epsilon. 
		$$ 
		For every $v\in\partial B(z_0,2r)$ we obtain from analogy of \eqref{analogy}, \eqref{klic} for $z_0$ and $\epsilon<\frac{\delta^2}{8}$
		$$
		\begin{aligned}
		|f(v)-f(z)|\geq &|Df(z_0)(v-z)|-|f(v)-f(z_0)-Df(z_0)(v-z_0)|- \\
		&-|f(z)-f(z_0)-Df(z_0)(z-z_0)|\\
		\geq &\delta^2 |v-z|-\epsilon(|v-z_0|+|z-z_0|)\\
		\geq &\delta^2 r-\epsilon4r>\frac{\delta^2}{2}r.\\
		\end{aligned}
		$$
		Since $f$ is a homeomorphism and $w\notin B(z_0,2r)$ we obtain now that 
		$$
		|f(z)-f(w)|\geq \inf\{|f(z)-f(v)|:\ v\in \partial B(z_0,2r)\}\geq \frac{\delta^2}{2}r. 
		$$
		This is a contradiction with \eqref{blizko} by our choice of $\epsilon<\frac{\delta^2}{4C_1}$. 
		
		We know that $\Omega$ has bounded measure and triangles in $\mathcal{T}$ have sidelength $r$ and thus $\#\mathcal{T}\leq \frac{C}{r^2}$.
		Now clearly $A\in WBV$ and singular part of second derivative $D^2_sA$ is supported on $\bigcup_{T\in\mathcal{T}}\partial T$ and corresponds to jump of the derivative there. We estimate it with the help of Theorem \ref{obluda} $(ii)$ as 
		\eqn{singular}
		$$
		\begin{aligned}
		\int_{\bigcup_{T\in\mathcal{T}}T}|D^2_sA|&\leq \#\mathcal{T}C\max_{T_1,T_2\in\mathcal{T}}|D^2_sA|(T_1\cap T_2)\\
		&\leq \#\mathcal{T} Cr\epsilon \mathcal{H}^1(\partial T)\\
		&\leq \frac{C}{r^2} Cr\epsilon Cr=C\epsilon. 
		\end{aligned}
		$$ 
		Moreover, the absolutely continuous part $D^2_aA$ satisfies by Theorem \ref{obluda} $(i)$ and \eqref{klic2} (call $v_T$ the corresponding vertex of $T$)
		$$
		\begin{aligned}
		\int_{\bigcup_{T\in\mathcal{T}}T}|D^2_aA-D^2f|&\leq 
		\sum_{T\in\mathcal{T}}\int_T \bigl(|D^2A-D^2f(v_T)|+|D^2 f-D^2f(v_T)|\bigr)\\
		&\leq 
		\sum_{T\in\mathcal{T}}C\epsilon \mir2(T)\leq C\epsilon\mir2(\Omega).\\
		\end{aligned}
		$$
		
		Finally we use Theorem \ref{Dan} for our mapping $A$ to obtain a $C^{\infty}$ diffeomorphism $g$ on $\bigcup_{T\in\mathcal{T}}T$ such that 
		$\|f-g\|_{L^{\infty}}<\nu$ and using \eqref{singular} 
		$$
		\begin{aligned}
		\int_{\bigcup_{T\in\mathcal{T}}T}|D^2f-D^2g|&\leq \int_{\bigcup_{T\in\mathcal{T}}T}\bigl(|D^2f-D^2_aA|+|D^2 g-D^2_aA|\bigr) \\
		&\leq C\epsilon\mir2(\Omega)+\epsilon+C\epsilon\leq C\nu. 
		\end{aligned}
		$$
		It follows that $\|f-g\|_{W^{2,1}(\bigcup_{T\in\mathcal{T}}T,\er^2)}<C\nu$. 
	\end{proof}


\begin{thebibliography}{00}

\bibitem{A}
\by{\name{Ambrosio}{L.}, \name{Fusco}{N.} and \name{Pallara}{D.}}
\book{Functions of bounded variation and free discontinuity problems}
\publ{Oxford Mathematical Monographs.
The Clarendon Press, Oxford University Press, New York, 2000}
\endbook

\bibitem{B}
\by{\name{Ball}{J.M.}}
\paper{Convexity conditions and existence theorems in nonlinear elasticity}
\jour{Arch. Rational Mech. Anal.}
\vol{63\rm, no. 4,} 
\pages{337--403} 
\yr{1977}
\endpaper


\bibitem{BB}
\by{\name{Ball}{J.M.}}
\paper{The calculus of variations and material science, Current and future challenges in the applications of mathematics, (Providence, RI, 1997)}
\jour{Quart. Appl. Math.}
\vol{56\rm, no. 4,} 
\pages{719--740} 
\yr{1998}
\endpaper

\bibitem{B2}
\by{\name{Ball}{J.M.}}
\book{Singularities and computation of minimizers for variational problems}
\publ{Foundations of computational mathematics (Oxford, 1999), 1--20, London Math. Soc.
Lecture Note Ser., 284, Cambridge Univ. Press, Cambridge, 2001}
\endbook

\bibitem{Ball2}
\by{\name{Ball}{J.M.}}
\book{Progress and puzzles in Nonlinear Elasticity}             
\publ{Proceedings of course on Poly-, Quasi- and Rank-One Convexity in Applied Mechanics. CISM Courses and Lectures. Springer, 2010}
\endbook



\bibitem{BCO}
\by{\name{Ball}{J.}, \name{Currie}{J.C.}, \name{Olver}{P.J.}}
\paper{Null Lagrangians, weak continuity, and variational problems of arbitrary order}
\jour{J. Funct. Anal.}
\vol{41\nom 3--4}
\pages{135--174}
\yr{1981}
\endpaper

\bibitem{BMC}
\by{\name{Ball}{J.M.}, \name{Mora-Corral}{C.}}
\paper{A variational model allowing both smooth and sharp phase boundaries in solids}
\jour{Commun. Pure Appl. Anal.}
\vol{8\nom 1}
\pages{55--81}
\yr{2009}
\endpaper

\bibitem{Ca}
\by{\name{Campbell}{D.}}
\paper{Diffeomorphic approximation of Planar Sobolev Homeomorphisms in Orlicz-Sobolev spaces}
\jour{J. Funct. Anal.}
\vol{273}
\pages{125--205}
\yr{2017}
\endpaper


\bibitem{CDH}
\by{\name{Campbell}{D.}, \name{D'Onofrio}{L.} and \name{Hencl}{S.}}
\paper{A sense preserving Sobolev homeomorphism with negative Jacobian almost everywhere}
\jour{arXiv:2003.03214, 2020}
\endprep

\bibitem{CHT}
\by{\name{Campbell}{D.}, \name{Hencl}{S.} and \name{Tengvall}{V.}}
\paper{Approximation of $W^{1,p}$ Sobolev homeomorphism by diffeomorphisms and the signs of the Jacobian}
\jour{Adv. Math.}
\vol{331,}
\pages{748--829}
\yr{2018}
\endpaper

\bibitem{CS}
\by{\name{Campbell}{D.}, \name{Greco}{L.}, \name{Schiattarella}{R} and \name{Soudsk\'y}{F.}}
\paper{Diffeomorphic approximation of Planar Sobolev Homeomorphisms in rearrangement invariant spaces}
\jour{arXiv:2005.04998, 2020}
\endprep


\bibitem{Ci}
\by{\name{Ciarlet}{P.G.}}
\book{Mathematical Elasticity. Vol. I. Three-dimensional elasticity.}
\publ{Studies in Mathematics and its Applications, {20}. North-Holland Publishing Co., Amsterdam, 1987}
\endbook

\bibitem{DP} 
\by{\name{Daneri}{S.} and \name{Pratelli}{A.}} 
\paper{Smooth approximation of bi-Lipschitz orientation-preserving homeomorphisms}
\jour{Ann. Inst. H. Poincar\'e Anal. Non Lin\'eaire}
\vol{31 \nom 3}
\pages{567--589}
\yr{2014}
\endpaper

\bibitem{DPP}
\by{\name{De Philippis}{G.} and \name{Pratelli}{A.}}
\paper{The closure of planar diffeomorphisms in Sobolev spaces}
\jour{preprint arXiv:1710.07228 }
\endprep

\bibitem{E}
\by{\name{Evans}{L.C.}} \paper{Quasiconvexity and partial regularity in the calculus of variations} \jour{Ann. of Math.}
\vol{95} \pages{227--252} \yr{1986}
\endpaper



\bibitem{HKr}
\by{\name{Healey}{T.J.} and \name{Kr\"omer}{S.}}
\paper{Injective weak solutions in second-gradient nonlinear elasticity}
\jour{ESAIM Control Optim. Calc. Var.}
\vol{15}
\pages{863--871}
\yr{2009}
\endpaper


\bibitem{HP}
\by{\name{Hencl}{S.} and \name{Pratelli}{A.}}
\paper{Diffeomorphic approximation of $W^{1,1}$ planar Sobolev homeomorphisms}
\jour{Jour. of the Eur. Math Soc.}
\vol{20\rm, no. 3,}
\pages{597--656}
\yr{2018}
\endpaper


\bibitem{HV}
\by{\name{Hencl}{S.} and \name{Vejnar}{B.}}
\paper{Sobolev homeomorphisms that cannot be approximated by diffeomorphisms in $W^{1,1}$}
\jour{Arch. Rational Mech. Anal.}
\vol{219\rm, no. 1,}
\pages{183--202}
\yr{2016}
\endpaper



\bibitem{IKO}
\by{\name{Iwaniec}{T.}, \name{Kovalev}{L.V.} and \name{Onninen}{J.}}
\paper{Diffeomorphic Approximation of Sobolev Homeomorphisms}
\jour{Arch. Rational Mech. Anal.}
\vol{201}  \pages{1047-–1067}
\yr{2011}
\endpaper

\bibitem{IO}
\by{\name{Iwaniec}{T.} and \name{Onninen}{J.}}
\paper{Monotone Sobolev Mappings of Planar Domains and Surfaces}
\jour{Arch. Rational Mech. Anal.}
\vol{219}
\pages{159--181}
\yr{2016}
\endpaper

\bibitem{IO2}
\by{\name{Iwaniec}{T.} and \name{Onninen}{J.}}
\paper{Triangulation of diffeomorphisms}
\jour{Math. Ann.}
\vol{368}
\pages{1133--1169}
\yr{2017}
\endpaper

\bibitem{Kr}
\by{\name{Kr\"omer}{S.}}
\paper{Global invertibility for orientation-preserving Sobolev maps via invertibility on or near the boundary}
\jour{preprint  arXiv:1912.11086 }
\endprep

\bibitem{Mbook}
\by{\name{M\"uller}{S.}}
\book{Variational models for microstructure and phase transition}
\publ{in "Calculus of Variations and Geometric Evolution Problems" (eds. S. Hildebrandt and M. Struwe), Springer-Verlag, (1999), 85-210}
\endbook

\bibitem{M}
\by{\name{Mora-Coral}{C.}} \paper{Approximation by piecewise  approximations of homeomorphisms of Sobolev homeomorphisms that are smooth outside a point} \jour{Houston J. Math.} \vol{35}  \pages{515-–539} \yr{2009}
\endpaper

\bibitem{MP}
\by{\name{Mora-Coral}{C.} and \name{Pratelli}{A.}} \paper{Approximation of Piecewise Affine Homeomorphisms by Diffeomorphisms} \jour{J. Geom. Anal.} \vol{24}  \pages{1398--1424} \yr{2014}
\endpaper

\bibitem{PR}
\by{\name{Pratelli}{A.} and \name{Radici}{E.}} 
\paper{Approximation of planar BV homeomorphisms by diffeomorphisms} 
\jour{Journal of Functional Analysis} 
\vol{276}  
\pages{659-–686} 
\yr{2019}
\endpaper

\bibitem{BiP}
\by {\name{Pratelli}{A.}}
\paper{On the bi-Sobolev planar homeomorphisms and their approximation}
\jour{Nonlinear Anal. TMA} 
\vol{154}  
\pages{258-–268} 
\yr{2017}
\endpaper


\bibitem{T}
\by{\name{Toupin}{R.A.}}
\paper{Elastic materials with couple-stresses}
\jour{Arch. Rational Mech. Anal.}
\vol{11}
\pages{385--414}
\yr{1962}
\endpaper




\bibitem{T2}
\by{\name{Toupin}{R.A.}}
\paper{Theories of elasticity with couple-stress}
\jour{Arch. Rational Mech. Anal.}
\vol{17}
\pages{85--112}
\yr{1964}
\endpaper

\end{thebibliography}
\end{document}